\documentclass[aos]{imsart}
 
\RequirePackage{amsthm,amsmath,amsfonts,amssymb}
\RequirePackage[numbers]{natbib}
\usepackage[applemac]{inputenc}
\usepackage[colorlinks]{hyperref}
 
\usepackage{amsthm, amsmath, natbib}
\usepackage{amsfonts, amssymb}
\usepackage{graphicx, color, latexsym}
\usepackage{eufrak}

\usepackage{enumerate}
\usepackage{booktabs}

\usepackage{tikz}
\usepackage{tikz-qtree}

\startlocaldefs
\theoremstyle{plain}                 

\newtheorem{theorem}{Theorem}

\newtheorem{lemma}{Lemma}

\theoremstyle{remark}
\newtheorem{remark}{Remark}

{}

\def\1{1\!{\rm l}}

\newcommand{\leqa}{\lesssim}
\newcommand{\geqa}{\gtrsim}

\newcommand{\EM}{\ensuremath}

\newcommand{\al}{\alpha}
\newcommand{\be}{\beta}

\newcommand{\ka}{\kappa}

\newcommand{\La}{\Lambda}

\newcommand{\si}{\sigma}
\newcommand{\te}{\theta}
\newcommand{\ta}{\tau}
\newcommand{\veps}{\varepsilon}
\newcommand{\vphi}{\varphi}

\newcommand{\cA}{\EM{\mathcal{A}}}
\newcommand{\cB}{\EM{\mathcal{B}}}
\newcommand{\cC}{\EM{\mathcal{C}}}

\newcommand{\cF}{\EM{\mathcal{F}}}

\newcommand{\cH}{\EM{\mathcal{H}}}

\newcommand{\cJ}{\EM{\mathcal{J}}}
\newcommand{\cK}{\EM{\mathcal{K}}}
\newcommand{\cL}{\EM{\mathcal{L}}}
\newcommand{\cM}{\EM{\mathcal{M}}}
\newcommand{\cN}{\EM{\mathcal{N}}}
\newcommand{\cP}{\EM{\mathcal{P}}}

\newcommand{\cS}{\EM{\mathcal{S}}}

\newcommand{\cX}{\EM{\mathcal{X}}}

\newcommand{\psg}{{\langle}}
\newcommand{\psd}{{\rangle}}

\definecolor{blendedblue}{rgb}{0.2,0.2,0.7}

\DeclareMathAlphabet{\mathpzc}{OT1}{pzc}{m}{it}

\newcommand{\RR}{\mathbb{R}}

\newcommand{\given}{\,|\,}

\newcommand{\rn}{\sqrt{n}}

\newcommand{\mockalph}[1]{}

\newcommand{\eps}{\varepsilon}

\newcommand{\bi}{\begin{enumerate}[label=\roman*)]}
\newcommand{\ei}{\end{enumerate}}
\newcommand{\ba}{\begin{array}{rcl}}
\newcommand{\ea}{\end{array}}

\newcommand{\Sbl}{\cS^\be(L)}
\newcommand{\Hbl}{\mathcal{H}^\be(L)}
\newcommand{\Bblr}{\mathcal{B}^\be_{rr}(L)}

\endlocaldefs

\begin{document}

\begin{frontmatter}
\title{Heavy-tailed Bayesian nonparametric adaptation}
\runtitle{Heavy-tailed Bayesian nonparametric adaptation}

\begin{aug}
\author[A]{\fnms{Sergios}~\snm{Agapiou}\ead[label=e1]{agapiou.sergios@ucy.ac.cy}}
\and
\author[B]{\fnms{Isma\"el}~\snm{Castillo}\ead[label=e2]{ismael.castillo@upmc.fr}}
\address[A]{Department of Mathematics and Statistics, University of Cyprus\printead[presep={,\ }]{e1}}

\address[B]{Laboratoire de Probabilit\'es, Statistique et Mod\'elisation, Sorbonne University, Paris, France\printead[presep={,\ }]{e2}}
\end{aug}

\begin{abstract}
We propose a new Bayesian strategy for adaptation to smoothness in nonparametric models based on heavy tailed series priors. We illustrate it in a variety of settings, showing in particular that the corresponding Bayesian posterior distributions achieve adaptive rates of contraction in the minimax sense (up to logarithmic factors) without the need to sample hyperparameters. Unlike many existing  procedures, where a form of direct model (or estimator) selection is performed, the method can be seen as performing a {\em soft} selection through the prior tail.  In Gaussian regression, such heavy tailed priors are shown to lead to (near-)optimal simultaneous adaptation both in the $L^2$-- and $L^\infty$--sense. Results are also derived for linear inverse problems, for anisotropic Besov classes, and for certain losses in more general models through the use of tempered posterior distributions. 
We present numerical simulations corroborating the theory.
\end{abstract}



\end{frontmatter}


\tableofcontents

\section{Introduction}\label{sec:intro}

Adaptation to smoothness is a central topic in nonparametric statistics. In a regression setting to fix ideas, convergence rates of estimators of the unknown regression function generally depend on the assumed degree of smoothness and this raises the question of finding {\em adaptive} estimators, which can recover the unknown truth at (near-)optimal rate in the minimax sense,  without assuming any prior knowledge of regularity. Popular non-Bayesian adaptation methods include Lepski's method \cite{lepski90},  thresholding \cite{djkp95}, and model selection \cite{bbm99}. 
Here we follow a Bayesian nonparametric approach and draw the unknown function randomly according to some prior distribution. In this setting, a possible way to derive `adaptation' is by following a hierarchical Bayes principle: for instance, one first randomly draws a function of given regularity say $\alpha>0$ and then draws $\alpha$ itself at random; this provides a hierarchical prior distribution which is then updated by conditioning on the observed data to form the posterior distribution. To give an example, initial prior draws can be for instance from an $\al$--smooth Gaussian process (such as Brownian motion when $\alpha=1/2$; and more generally e.g. Brownian motion integrated a fractional number of  times); and $\al$ itself can be sampled according to a Gamma distribution. 


In this work we focus on prior distributions given in the form of a stochastic process characterised by a sequence of coefficients into an expansion basis. A popular related example in statistics and machine learning (e.g. \cite{rw06}) is the one of Gaussian process priors, for which van der Vaart and van Zanten \cite{vvvz08} proved the following generic result. Suppose we are in a simple one-dimensional nonparametric regression setting with Gaussian errors (e.g. a white noise model). If the true regression function is $\be$--smooth in the Sobolev sense, and one considers an $\alpha$--smooth Gaussian process as a prior distribution, then the posterior distribution contracts at rate 
\begin{equation} \label{rateg}
 \veps_n\leqa 
\begin{cases}
n^{-\be/(1+2\al)},&\qquad \text{if } \al\ge \be,\\
n^{-\al/(1+2\al)},&\qquad \text{if } \al \le \be,
\end{cases}
\end{equation}
and these rates cannot be improved \cite{ic08}. 

In inverse problems and medical imaging, processes which feature  tails heavier than Gaussian are increasingly used. The tails can be `moderate', for instance in-between Gaussian and Laplace tails,  see e.g.  \cite{lassasetal09}, \cite{dhs12}, \cite{ABDH18} and \cite{klns12}, \cite{BG15} for numerical aspects, considering the case of so--called Besov priors; but also heavier than Laplace, including polynomially decreasing tails, such as recently considered e.g. in \cite{zoubin14}, \cite{S17}, \cite{SCR22}. Yet overall there is up to now much less theoretical understanding of processes featuring a form of heavy tails. 

The class of $\alpha$--regular $p$--exponential measures for $p\in[1,2]$ is introduced in  \cite{adh21}, where the authors model the coefficients onto a basis {$\{\varphi_k\}$} of the function of interest with a prior with moderate tails (in-between Gaussian and Laplace) and variance decreasing as $k^{-1-2\al}$;  therein the authors prove posterior contraction at rate 
\begin{equation} \label{ratep}
 \veps_n\leqa 
\begin{cases}
n^{-\be/\{1+2\be+p(\al-\be)\}},&\qquad \text{if } \al\ge \be,\\
n^{-\al/(1+2\al)},&\qquad \text{if } \al \le \be. 
\end{cases} 
\end{equation}
Here the index $p$ corresponds to the tail behaviour of individual coefficients of the prior, with $p=2$ recovering the Gaussian case \eqref{rateg} and $p=1$ corresponding to Laplace (double--exponential) tails; and $\be$ again refers to Sobolev smoothness of the true regression function. 

A few remarks in the light of \eqref{rateg}--\eqref{ratep} are
\begin{itemize}
\item the optimal (minimax) rate is achieved in both cases for $\al=\be$ (only);
\item when $\al\ge \be$ (the `oversmoothing case'), when $p$ decreases from $2$ to $1$,  the rate slightly improves in terms of powers of $n^{-1}$ from $\be/(1+2\al)$ to $\be/(1+\be+\al)$;
\item when $\be\ge \al$ (the `undersmoothing case'), the rate is always $n^{-1}$ to the power $\al/(1+2\al)$, driven by the prior's own regularity.
\end{itemize}
As $\be$ is rarely known in practice, the previous prior distributions have to be made more complex if one wishes to derive {\em adaptation}. And indeed, two remarkable papers  \cite{vvvz09, ksvv16} proved that adaptation in $L^2$--sense can be achieved from $\al$--smooth Gaussian processes, up to logarithmic terms, by either using an additional rescaling variable
, or by `estimating' the prior's regularity $\al$, either in a hierarchical Bayes or an empirical Bayes way (see also below for related references). Analogous adaptation results were derived very recently for $p$--exponential series priors in \cite{as22}; see also \cite{giordano22} for results on related priors in density estimation.

A natural question is thus what happens when tails heavier than Laplace are considered. This would formally correspond to taking the index $p$ of $p$--exponential priors to be smaller than $1$ and even going to $0$. Let us `do' the formal manipulation $p\to 0$ in \eqref{ratep}, and  see what happens. The `limiting' rate would then become $n^{-\be/(1+2\be)}$ for $\al\ge \be$.  This would mean that at least for the `oversmoothing' case, the `adaptive' rate would be automatically obtained. Of course we just did a formal substitution that could perhaps not be valid; in particular, the techniques employed in \cite{adh21} rely on the logarithmic concavity of $p$-exponential priors, a property which no longer holds for $p<1$. We will see that, somewhat surprisingly at first, heavy tails enable, in a variety of settings, to derive adaptation in a fully automatic way.

The meaning of `heavy tail' in the title and through the paper can be understood in the relatively loose sense of having tails that have a polynomial-type decrease, although the actual conditions under which we work are somewhat milder. \\


{\em Heavy tailed series priors.} 
 Depending on the setting, we construct priors in $L^2:=L^2[0,1]$ via series expansions in either an orthonormal basis $\{\vphi_k: \,k\ge 1\}$ or an orthonormal boundary-corrected wavelet basis $\{\psi_{lk}: \,l\ge 0, \,k\in\cK_l\}$ where $\cK_l=\{0,\dots, 2^{l}-1\}$ and where we denote the scaling function as the first wavelet $\psi_{00}$. Without loss of generality we have taken the coarsest scale to be 1, while it is straightforward to accommodate for coarsest scales finer than 1. For more details, see \cite[Section 4.3]{ginenicklbook}.\\

{\em Prior $\Pi$ on functions.} For $\psg\cdot,\cdot\psd$ the usual inner product on $L^2$, let $f_k:=\psg f,\vphi_k \psd$, respectively $f_{lk}:=\psg f,\psi_{lk} \psd$, denote the coefficients of $f\in L^2$ onto the considered bases, so that
\[ f = \sum_{k=1}^\infty f_k \vphi_k,\qquad   \text{or}\qquad
f = \sum_{l=0}^\infty \sum_{k\in \cK_l}  f_{lk} \psi_{lk}.
\]
Let us define a prior $\Pi$ on $f$ by  letting, for $(\sigma_k), (s_l)$ sequences to be chosen below, and $(\zeta_k), (\zeta_{lk})$  independent identically distributed random variables of common law $H$ with heavy tails, also to be specified,
\begin{equation} \label{prior}
f_k \overset{\text{ind.}}{\sim} \sigma_k \zeta_k,
\end{equation}
in the case of a single--index basis $(\vphi_k)$, or for a double--index basis
\begin{equation} \label{priord}
f_{lk} \overset{\text{ind.}}{\sim} s_l \zeta_{lk}.
\end{equation}
A key choice of scale parameters $\sigma_k$ and $s_l$ throughout the paper is, for any $k\ge 1$ and $l\ge 0$,
\begin{equation} \label{defsigr}
\sigma_k = e^{-(\log k)^2},\qquad s_l  = 2^{-l^2}.
\end{equation}
Another possible choice we consider, again for any such $k,l$ and some $\alpha>0$ is
\begin{equation} \label{defsig}
\sigma_k = k^{-1/2-\alpha},\qquad s_l  = 2^{-l(1/2+\alpha)}.
\end{equation}  
{The choice \eqref{defsig} corresponds to the same scaling as in  \eqref{ratep} and, since here we consider heavy tails, to (formally at least) setting $p=0$ for a $p$-exponential prior. }
Contrary to \eqref{defsig}, where the value of $\alpha$ should be chosen, note that \eqref{defsigr} is in principle free of any parameter. The choice of the square in \eqref{defsigr} is mostly to fix ideas and results below also hold for $\sigma_k=e^{-a(\log k)^{1+\delta}}$ with any given constants $a, \delta>0$. The fact that $(\sigma_k)$ in \eqref{defsigr} decreases faster than any polynomial in $k^{-1}$, but not exponentially fast (as for instance $e^{-k}$), is key for the results ahead. Similarly, for double-index bases we can use $s_l=2^{-l^{1+\delta}}$ for any fixed $\delta>0$ in \eqref{defsigr}.

%

To complete the prior's description, let us now specify the distribution $H$ of the $\zeta$ variables as above. Suppose that $H$ admits a density $h$ on $\RR$ and that for $c_1>0$ and $\kappa\ge 0$, 
\begin{align}
h   \text{ is symmetric, positive, bounded and decreasing on } [0,\infty),  \label{conds} \\
\log(1/h(x)) \le c_1(1+\log^{1+\kappa}(1+x)), \qquad x\ge 0. \label{condti}
\end{align}
A leading example throughout the paper is the case $\kappa=0$ corresponding to  polynomial tails (sometimes called fat tails):  Student distributions, including Cauchy,  satisfy these conditions for $c_1$ large enough constant. Yet, some flexibility is allowed with $\ka>0$ permitting slightly lighter tails. Depending on the setting, we sometimes assume a mild integrability or moment condition, that will still accommodate most Student-type tails.  

We call priors $\Pi$ verifying \eqref{defsigr}--\eqref{conds}--\eqref{condti} {\em Oversmoothed heavy-Tailed} priors or simply OT--priors while HT$(\al)$ for {\em $\al$--Heavy Tailed} priors stand for those satisfying \eqref{defsig}--\eqref{conds}--\eqref{condti}.\\

{\em Frequentist analysis of posterior distributions.} Consider a statistical model $\{P_f^{(n)},\ f\in \cF \}$ indexed by a function $f$ with observations $X=X^{(n)}$. Examples considered below include nonparametric regression, density estimation and classification. Given a prior distribution $\Pi$ on $f$, the Bayesian model sets $X\given f\sim  P_f^{(n)}$ and $f\sim \Pi$. The posterior distribution $\Pi[\cdot\given X]$ is the conditional distribution $f\given X$. 
Assuming the model is dominated, the posterior is given as usual by Bayes' formula. Taking a frequentist approach, we analyse the posterior $\Pi[\cdot\given X]$ under the assumption that $X$ has actually been generated from $P_{f_0}^{(n)}$ for some fixed true function $f_0$. We refer to the book \cite{gvbook} for more context and references. \\

%

{\em Classical regularity balls.}  Before describing our main results, let us recall three types of standard smoothness assumptions on the underlying truth $f_0$: Sobolev, H\"older and Besov.\\

[Sobolev--type]$\ $ When working with an orthonormal basis $\{\vphi_k\}$, we consider Sobolev-type assumptions. Recalling that $f_k=\langle f, \vphi_k\rangle$, for $\be, L>0$, denote
 \begin{equation}\label{sobolev} \Sbl=\Big\{f=(f_k),\quad \sum_{k\ge 1} k^{2\be} f_k^2\le L^2 \Big\}.\end{equation}
 For certain choices of $\vphi_k$, the sets $\Sbl$ correspond to balls of classical Hilbert-Sobolev spaces of functions in $L^2[0,1]$ possessing $\be$ square integrable derivatives.

When working with an orthonormal wavelet basis $\{\psi_{lk}\}$ we consider either hyper-rectangle (H\"older-type) or Besov-type assumptions. \\

[H\"older--type] $\ $ For $f_{lk}=\langle f, \psi_{lk}\rangle$ and $\be, L>0$, let
\begin{equation}\label{rectangle}
\Hbl=\Big\{f=(f_{lk}),\quad \max_{k\in\cK_l}|f_{lk}|\leq 2^{-l(1/2+\be)}L\ \ 
\text{for all }\,
l\ge 0\Big\}.
\end{equation}
For wavelet bases with classical H\"older regularity higher than $\beta$, the sets  $\Hbl$ correspond to $L$-balls of the H\"older-Zygmund spaces $\cC^\be[0,1]$, see \cite[Section 4.3]{ginenicklbook}. For non-integer $\be$ the latter spaces coincide with the classical H\"older spaces $C^\be[0,1]$, while for $\be$ an integer it holds $\cC^{\be'}\subset C^\be\subset \cC^\be$ for all $\be'>\be$ where inclusions are all strict. \\

[Besov--type] $\ $ For $\be, L>0$ and $1\leq r\leq 2$, let
\begin{equation}\label{besov} \Bblr = \Bigg\{f=(f_{lk}),\ \ \ \sum_{l\ge 0} 2^{rl(\be+1/2-1/r)} 
\sum_{k\in \cK_l} |f_{lk}|^r <L^r\Bigg\}.
\end{equation}
Again for appropriate wavelet bases, 
 the sets $\Bblr$ correspond to $L$-balls of Besov spaces $B^\beta_{rr}[0,1]$ defined via moduli of continuity, see \cite[Section 4.3]{ginenicklbook}. 
For $r=2$, Besov spaces coincide with the Hilbert-Sobolev spaces, while for $r<2$ Besov spaces are useful for modelling spatially inhomogeneous functions, that is functions which are smooth in some areas of the domain and irregular or even discontinuous in other areas, see \cite{DJ98} or \cite[Section 9.6]{IJ19}. 
Here, we restrict to Besov spaces $\cB^\beta_{rq}$ with $r=q$ for simplicity, see Section \ref{sec:disc} for a discussion. \\

{\em Outline and informal description of the results.} In what follows, we will substantiate the intuition that heavy-tailed series priors achieve adaptation to smoothness without the need to sample any hyperparameters. Our results show that OT-priors are fully adaptive, and that HT$(\alpha)$-priors are partially adaptive (essentially) for smoothness of the truth $\beta\leq\alpha$. More precisely, in Section \ref{sec:reg} we consider white noise regression and show (near-) adaptive posterior contraction rates in the minimax sense in the following settings:
\begin{itemize}
\item in $L^2$--loss for Sobolev regularity in both the direct and a 
linear inverse problem setting;
\item in $L^\infty$--loss under H\"older smoothness in the direct setting;
\item in $L^2$--loss under (spatially inhomogeneous) Besov smoothness in the direct setting.
\end{itemize}
A result on the limiting shape of the posterior distribution is also given, in the form of an adaptive nonparametric Bernstein--von Mises theorem. In Section \ref{sec:prmass}, we establish generic bounds for the mass that heavy-tailed priors put on $L^2$ and $L^\infty$--balls around Sobolev and H\"older truths, respectively. By themselves, such bounds allow the derivation of contraction rates for \emph{tempered} posterior distributions in general models, in terms of R\'enyi divergence. Indeed, we exemplify this approach in three nonparametric settings, in which we achieve (near-) adaptive rates of contraction of tempered posteriors in the minimax sense:
\begin{itemize}
\item in density estimation, in $L^1$--loss and under H\"older smoothness of the true log-density;
\item in binary classification, in an $L^1$--type loss 
and under Sobolev smoothness of the logit of the true binary regression function;
\item in (direct) white noise regression, in $L^2$--loss and under Besov smoothness of the truth.
\end{itemize}
A simulation study is presented in Section \ref{sec:sim}, while a brief discussion and review of open questions can be found in Section \ref{sec:disc}. Proofs are presented in Section \ref{sec:proof} as well as in the Supplementary material. The Supplement  also includes additional simulations and a discussion on 
extending the results of Section \ref{sec:prmass} to contraction of standard posteriors.\\

{\em Comparison with other priors.} While our results below shall answer positively the question of obtaining adaptation with heavy-tailed series priors, it is of interest to compare our priors with other priors leading to adaptation. The list below is by far not exhaustive; we mention a few classes of priors that bear some similarity with the priors here considered.
\begin{itemize}
\item {\em Sieve priors} (e.g. \cite{scricciolo06, arbel13, ray13, shenghosal15}). Here adaptation is obtained by truncating the series prior and taking the truncation parameter $K$ random; note that the distribution on the modelled $K$ coefficients in general plays little role on the obtained rate. In regression the above references show that (near)-optimal adaptive rates are achieved  in the $L^2$--norm; but generally this is not the case in the $L^\infty$ norm (\cite{cr21}, Theorem 5). In contrast, we will see that heavy-tailed series priors in white noise regression are adaptive in both norms.
\item {\em Spike--and--slab priors and sparsity inducing priors.} Due to their links to  thresholding rules, spike--and--slab (SAS) priors are also particularly natural: \cite{hrs15} show in white noise regression that SAS posteriors achieve adaptive rates both in $L^2$ (nearly) and $L^\infty$ (there are few results in other models, except \cite{cm21, naulet22}  in density estimation).  While heavy-tailed priors share the same desirable properties, they do not model sparsity so have no `mass at $0$' part; this can be an advantage computationally, as in more complex models, sampling from SAS posteriors typically requires exploration of a combinatorial number of models. While sampling from OT posteriors is relatively easy using MCMC, we are not aware of posterior samplers for SAS in density estimation for instance (the posterior in \cite{cm21} is computable but uses partial conjugacy and is limited to regularities up to $1$). The horseshoe prior \cite{carvalhopolsonscott} is in a sense closer to our proposal as it has density with Cauchy tails. Note though that similarly to SAS priors and unlike our heavy-tailed prior construction it directly models sparsity through a diverging density at $0$. We expect that horseshoe priors have good adaptation properties, although we do not know any proof in a  nonparametric context (\cite{carvalhopolsonscott} present simulations in one such setting) -- the techniques we develop here could be used precisely to derive such results.
\item {\em Mixtures.} It may be argued that heavy-tailed distributions can be represented as mixtures of lighter tailed distributions: for instance, Laplace and Student distributions are scale mixtures of normals. So, one could view the heavy-tailed prior in a hierarchical manner with independent Gaussians and one hyper-parameter per coefficient. Note, however, that this in general does not suffice for adaptation: for instance Laplace series priors do not adapt optimally; and even if the resulting distribution on coordinates is e.g. Student, the choice of scale parameters $\sigma_k$ or $s_l$ in \eqref{defsigr}--\eqref{defsig} is essential, as for instance \eqref{defsig} does not adapt if $\al<\be$. This shows that only some well-designed mixture priors work. Furthermore, even if a heavy-tailed law has a mixture representation, this does not mean that it is advantageous computationally to use it (e.g. using a Gibbs sampling to approximate the posterior; this  may face computational difficulties due to the high number of hyper-parameters), and in fact we do not do so in Section \ref{sec:sim}, where we use direct sampling from the posterior via MCMC in all considered examples. Also related to mixtures,  
\cite{gaozhou16} construct a hierarchical block--prior that enables to derive contraction rates in $L^2$--sense (or with testing distances) without additional logarithmic terms.  The resulting construction requires a specific hyper-prior, and may be difficult to sample from in complex settings (e.g. beyond white noise regression); also, although optimal in $L^2$ it is presumably suboptimal in the $L^\infty$--sense.

\end{itemize}

\section{Nonparametric regression} \label{sec:reg}


To avoid technicalities independent of the ideas at stake, we focus in this section on the  Gaussian white noise model, that can be seen as the prototypical nonparametric model \cite{t09, ginenicklbook}. Up to dealing with discretisation effects, similar results as the ones below are expected to hold also e.g. for fixed-design nonparametric regression. For $f\in L^2$ and $n\ge 1$, the Gaussian white noise model writes
\begin{equation} \label{gwn}
 dY^{(n)}(t) = f(t)dt + dW(t)/\sqrt{n},\qquad t\in[0,1],
\end{equation}
where $W$ is standard Brownian motion. 

\subsection{$L^2$--loss and Sobolev smoothness}
By projecting \eqref{gwn} onto a single-index orthonormal basis $\{\vphi_k\}$ of $L^2$, one obtains the normal sequence model, with $f_k=\psg f,\vphi_k \psd$,
\begin{equation}\label{normseq1}
X_{k}|f_k\sim \cN(f_k, 1/n),
\end{equation}
independently for $k\ge 1$, with $X_k=\int_0^1 \vphi_k(t)dY^{(n)}(t)$. 
We denote $X=X^{(n)}=(X_1, X_2, \ldots )$ the corresponding observation sequence. Here for simplicity of notation we consider only single-index bases, but the results in the present Section and the next hold as well for double-indexed wavelet bases, such as ones considered in Section \ref{sec:linf}, with the corresponding appropriately chosen scalings as in \eqref{defsigr}--\eqref{defsig}.

Early  results for Bayesian series priors {(we discuss a few other priors in Section \ref{sec:disc})} in this setting include \cite{z00}, who established non-adaptive convergence rates for the posterior mean under Gaussian priors, while \cite{bg03} derived adaptive rates using a hyperprior over a discrete set of regularities. Still for Gaussian series priors, in \cite{svv13} partial adaptation was achieved with fixed regularity using either a hierarchical or an empirical Bayes choice of a universal scaling parameter provided the truth is not too smooth compared to the prior, while in the work \cite{ksvv16} full adaptation (up to logarithmic factors) was established using either a hierarchical or an empirical Bayes choice of the prior regularity. Gaussian series priors on manifolds with an extra random time parameter were shown to be adaptive to smoothness in broad geometric contexts \cite{ckp14}. 
 More recently, both fixed-regularity and adaptive results were derived for $p$-exponential priors in \cite{adh21} and \cite{as22}.

In this section, we consider series priors as in \eqref{prior}-\eqref{defsigr}, defined via a heavy-tailed density $h$ 
 satisfying the moment assumption, for some $q\ge1$,
 \begin{equation} \label{qmom}
\int_{-\infty}^{\infty} |x|^qh(x)dx <\infty.
\end{equation}


\begin{theorem} \label{thm-seq}
In the regression model \eqref{normseq1},  consider the heavy-tailed series prior \eqref{prior} with parameters specified by \eqref{defsigr} and \eqref{conds}--\eqref{condti} as well as \eqref{qmom} with $q=2$. Suppose $f_0\in \Sbl$ for some $\be, L>0$. Then,  as $n\to\infty$,
\[ E_{f_0}\Pi\left[ \{f:\ \|f-f_0\|_2 > \cL_n  n^{-\be/(2\be+1)} \} \given X^{(n)} \right] \to 0, \]
where $\cL_n=(\log{n})^{d}$ for some $d>0$. Further, the same result holds, with a possibly different $d$, for the choice of $(\sigma_k)$ as in \eqref{defsig}, provided $\alpha\ge \be$. Both results also hold for truncated priors at $k=n$, that are the same as the ones considered except they set $f_k=0$ for $k>n$. 
\end{theorem}

Theorem \ref{thm-seq} shows that the oversmoothed heavy-tailed (OT) prior as in \eqref{defsigr}--\eqref{condti} leads to full adaptation to smoothness $\be>0$, without any restriction to the range over Sobolev balls $S^\be$ (see also Remark \ref{rem:basis}) and without the need of tuning of any {smoothness hyper--parameter}. 
The fact that the second part of Theorem \ref{thm-seq} holds shows that the heuristic presented below \eqref{ratep} letting $p\to 0$ is correct: if $(\sigma_k)$ is polynomially decreasing as $\sigma_k=k^{-1/2-\alpha}$ as in \eqref{defsig}, then adaptation holds in the range  $\be\in(0,\alpha]$, that is exactly in the  case the prior `oversmooths' the truth, as expected from formula \eqref{ratep}. For a comparison with other priors and more discussion, we refer to Section \ref{sec:disc}. \\

A difficulty with the proof of Theorem \ref{thm-seq} is that it does not seem possible to use the general approach to posterior convergence rates as in \cite{ggv00, gvbook},  as the latter requires exponential decrease of probabilities of sieve sets (at least with infinite series priors or priors modelling high-frequencies, so excluding sieve priors, for which specific arguments can be used, see e.g. \cite{arbel13, ray13, shenghosal15}), which is essential in being able to discard regions of the parameter spaces, as crucially used in results for Gaussian or $p$--exponential priors (the latter just allow for exponential  decrease when $p=1$). Our proof is based on a detailed analysis of the posterior induced on coefficients, the most delicate part being high-frequencies, for which careful compensations from numerator and denominator in the ratios arising from Bayes' formula are needed. We note that the results in the present section impose a moment condition on the heavy-tailed density $h$; this is mostly for technical convenience:  it is expected that existence of a second moment is required in the theorem above, as its proof goes through controlling the posterior second moment. It is likely that one can remove the moment condition by requiring a control in probability only; such approach would allow to include the Cauchy density, but the proof would likely be more technical, so we refrain from pursuing this goal here; we only note that in this vein results for the Cauchy prior are derived in Section \ref{sec:prmass}.\\ 

{\em Numerical intuition behind the result.} 
Underlying our proofs, is the behaviour of heavy-tailed priors on $\mu\in\RR$ in the model $X\given \mu \sim \cN(\mu, 1/n)$, which we compare here to the behaviour of Gaussian priors. Consider $\mu \sim \sigma \Pi$ where $\sigma$ is a positive scaling and $\Pi$ is either standard normal or say a standard Student distribution with 3 degrees of freedom. Recall that in the Gaussian prior case, the posterior mean is given as $E[\mu\given X]=n\sigma^2X/(1+n\sigma^2)$. Figure~\ref{fig:studentmeans} depicts the posterior mean in the Student prior case as a function of $X$, for decreasing values of $\sigma$ and noise level $1/\sqrt{n}=10^{-3.5}$. We observe that in the Student prior case, for large prior scalings $\sigma$ the posterior mean is given by the observation (this is similar to the Gaussian prior case), while for small $\sigma$ the posterior mean resembles a thresholding estimator, nearly setting to zero small observations and preserving larger observations (this is unlike the Gaussian prior case where observations are shrunk by a constant factor determined by the size of $\sigma$ relative to the noise precision $n$). In particular, Figure \ref{fig:studentmeans} suggests that with the Student prior, good recovery is achieved by the posterior mean independently of the size of the scaling $\sigma$, for $|X|\ge 0.002 \gg 1/\sqrt{n}\approx 0.0003$. Contrast this to the Gaussian prior case, where small $\sigma$ leads to poor recovery of large observations. It thus appears, that an oversmoothing heavy tailed prior may still have good nonparametric behaviour, despite the scaling being mismatched. 

\begin{remark}[Optimal rate and log factors] \label{rem:log} The rate in Theorem \ref{thm-seq} is optimal up to a logarithmic factor (it can be checked for instance that one can take a power $d=1$ in $\cL_n$ for the OT prior), which we did not try to optimise. A main reason is that work by Tony Cai \cite{cai08} shows that any method that is smoothness-adaptive in $L^2$ and separable, in the sense that it makes coordinates independent, must pay a logarithmic price in its convergence rate. Moreover, the squared-rate cannot be better than $(\log{n}/n)^{2\be/(2\be+1)}$ for some $\be$, a rate (nearly) achieved for tempered posteriors in Theorem \ref{thm-besov-rho} below, see Remark \ref{rem:log2} and \eqref{rateveps-2}.  
\end{remark}

\begin{remark}[Smoothness and order of basis] \label{rem:basis} Theorem \ref{thm-seq} and results below hold for any smoothness parameter $\be>0$, over regularity balls defined by coefficients as before. As usual with estimators defined over bases, if one wishes results over classical H\"older spaces or Besov spaces defined via moduli of continuity, one needs to assume a basis of large enough order/regularity (which means adaptation holds in that case over $\be\le \be_{max}$, where $\be_{max}$ can be made as large as desired by choosing the order of the basis large enough). 
\end{remark}


\begin{figure}
    \centering
    \includegraphics[width=1\textwidth]{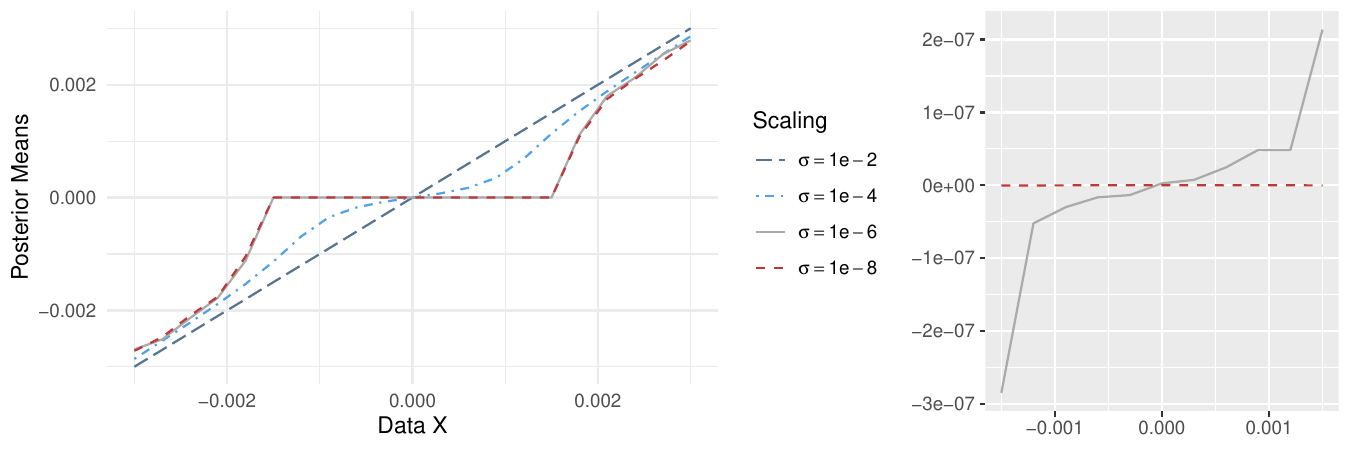}

    \caption{Left: posterior means for the univariate model $X\given \mu \sim N(\mu,10^{-7}), \,\mu \sim\sigma\Pi$, plotted against the observed data $X$, for $\Pi$ a standard Student distribution with 3 degrees of freedom and for 4 values of the scaling $\sigma$. Right: detailed view of the center region of the plot on the left.}
    \label{fig:studentmeans}
\end{figure}

\subsection{Linear inverse problems, Sobolev smoothness}\label{ssec:ip}

A synthetic prototypical model in linear inverse problems arises when projecting onto the SVD of the forward operator: the observation model is, for some $\frak{\nu}\ge 0$, independently for $k\ge 1$,
\begin{equation} \label{modip}
X_{k}|f_k\sim \cN(\kappa_k f_k, 1/n),\qquad \kappa_k\asymp k^{-\frak{\nu}}.
\end{equation}
This generalises the former signal--in--white--noise setting with the `inverse' nature of the problem represented by the sequence $(\kappa_k)$. This model has been much studied in terms of minimaxity and adaptation over Sobolev smoothness in the frequentist literature \cite{cavalier11}. Following a Bayesian approach with Gaussian priors, \cite{kvv11} derived posterior contraction rates in the non-adaptive case of fixed regularity; again in \cite{ksvv16}, the authors proved that empirical and hierarchical Bayes approaches could be used to derive adaptation for the previous class of Gaussian priors. Results for sieve priors were obtained in \cite{ray13} (see also Section \ref{sec:disc} for more on this). 

\begin{theorem} \label{thminv}
In the inverse regression model \eqref{modip} with degree of ill--posedness of the forward operator $\nu\ge 0$, consider the heavy-tailed series prior \eqref{prior} with parameters specified by \eqref{defsigr} and \eqref{conds}--\eqref{condti}  as well as \eqref{qmom} with $q=2$. Suppose $f_0\in \Sbl$ for some $\be, L>0$. Then,  as $n\to\infty$,
\[ E_{f_0}\Pi[ \{f:\ \|f-f_0\|_2 > \cL_n  n^{-\be/(2\be+2\nu + 1)} \} \given X^{(n)}] \to 0, \]
where $\cL_n=(\log{n})^{d}$ for some $d>0$. Further, the same result holds, with a possibly different $d$, for the choice of $(\sigma_k)$ as in \eqref{defsig}, provided $\alpha\ge \be$. 
\end{theorem} 
 
Theorem \ref{thminv} generalises Theorem \ref{thm-seq} (the case $\nu=0$). With the technique and estimates of the proof of Theorem \ref{thm-seq} in hand, its proof via coordinate-estimates is relatively simple; this is in contrast to existing empirical or hierarchical Bayes proofs in this setting for infinite series Gaussian priors \cite{ksvv16}, for which the study of the marginal maximum likelihood estimates (or of the posterior on the smoothness parameter) requires non-trivial work. 

\subsection{$L^\infty$--loss in white noise regression} \label{sec:linf}
 
Let us now consider estimation under the $L^\infty$-loss, which is a particularly desirable cost function in curve estimation, as a bound in supremum norm guarantees visual closeness between curves. In this case, we expand in a wavelet orthonormal basis, and the projection of model \eqref{gwn} becomes a normal sequence model
\begin{equation}\label{normseq2}
X_{lk}|f_{lk}\sim \cN(f_{lk}, 1/n),
\end{equation}
independently over relevant indices $l,k$. Again, we denote by $X^{(n)}$ the observation sequence. Minimaxity and adaptation in supremum loss for H\"older smoothness have been studied extensively in the frequentist literature \cite{ih80, lepski91, ginenicklbook}. Deriving results for Bayesian procedures (or more generally for likelihood--based procedures, see e.g. the notes of Chapter 7 of \cite{ginenicklbook}) for this loss is in general quite delicate. 
Generic results in this direction include \cite{gn11, ic14}, 
but most existing results are concerned with specific models and priors. In white noise regression, supremum norm adaptation has been derived so far for spike--and--slab priors \cite{hrs15}, and tree priors {\em \`a la} Bayesian CART \cite{cr21} (up to an unavoidable log factor). 

\begin{theorem} \label{thm-seq-sup}
In the regression model, let us consider the prior on $f$ induced by the heavy-tailed wavelet series prior \eqref{priord} on coefficients $f_{lk}$ in \eqref{normseq2}, with parameters specified by \eqref{defsigr} and \eqref{conds}--\eqref{condti} as well as \eqref{qmom} with $q\ge1$. Suppose $f_0\in \Hbl$ for some $\be, L>0$. Then,  as $n\to\infty$,
\[ E_{f_0}\Pi[ \{f:\ \|f-f_0\|_\infty > \cL_n  (\log{n}/n)^{\be/(2\be+1)} \} \given X^{(n)}] \to 0, \]
where $\cL_n=(\log{n})^{d}$ for some $d>0$. Further, the same result holds, with a possibly different $d$, for the choice of $(s_l)$ as in \eqref{defsig} and under \eqref{qmom} with any $q\ge1$, provided  $\alpha\ge \beta+1/q$.\\
\end{theorem}
Theorem \ref{thm-seq-sup} shows that the OT prior attains the adaptive minimax supremum-norm rate up to a logarithmic term. This is the first result of this kind for a prior distribution that does not have a `spike' part that sets coefficients to $0$ (as opposed to spike--and--slab and tree priors, that select a subset of coefficients and set all others to $0$). For the HT$(\al)$ prior, the condition $\alpha\ge\beta+1/q$ is for technical reasons and may be suboptimal (note though that as $q\to \infty$ the condition becomes milder and approaches the one in Theorem \ref{thm-seq} for the $L^2$--loss).



\subsection{Besov classes}
We now consider the case of unknown functions $f_0$ that can be spatially inhomogeneous in the projected white noise model \eqref{normseq2}.  In particular, we study adaptation of heavy tailed priors for underlying true functions with Besov smoothness $\Bblr$ for $1\leq r <2$. Minimaxity and adaptation over spatially inhomogeneous Besov spaces in this model have been studied in the frequentist setting in \cite{DJ98}. A distinctive feature is that linear estimators are provably suboptimal by a polynomial factor. More recently, rates of contraction in the non-adaptive case of fixed regularity were established using undersmoothing and appropriately rescaled $p$-exponential priors with $p\leq r$ in \cite{adh21}. Adaptation with $p$-exponential priors, $p\leq r$, was achieved in \cite[Theorem 2.5]{as22} using either a hierarchical or an empirical (marginal maximum likelihood) Bayes choice of both the regularity and scaling parameters (simultaneously). Importantly, it was established in \cite{aw21} that Gaussian priors suffer from the same suboptimality as linear estimators in the frequentist setting. The next result establishes adaptation with heavy-tailed priors in this setting as well. 

\begin{theorem} \label{thm-besov}
In the white noise regression model, let  $\Pi$ be a heavy tailed prior generated by \eqref{priord} with parameters specified by \eqref{defsigr} and \eqref{conds}--\eqref{condti}   as well as \eqref{qmom} with $q=2$.  
 Suppose $f_0\in \Bblr$ for some $\be, L>0$, $r\in[1,2]$  and $\be>1/r-1/2$.  Then, as $n\to\infty$,
\[ E_{f_0}\Pi[ \{f:\ \|f-f_0\|_2 > \cL_n  n^{-\be/(2\be+1)} \} \given X^{(n)}] \to 0, \]
where $\cL_n=(\log{n})^{d}$ for some $d>0$. Further, the same result holds, with a possibly different $d$, for the choice of $(\sigma_k)$ as in \eqref{defsig}, provided $\alpha\ge \be$. 
\end{theorem}

The assumption $\beta>1/r-1/2$ is sharp, in the sense that it is the weakest assumption on $\beta$ ensuring $B^\beta_{rr}\subset L^2$, \cite[Theorem 3.3.1]{T83},  a necessary condition since we consider contraction in $L^2$--loss. This is in contrast to \cite[Theorem 2.5]{as22} which has a more stringent assumption on $\beta$ (arising due to the necessity of controlling the prior mass uniformly with respect to the scaling prior-parameter, see Remark 2.6(a) in that source). Another advantage of the above result is that it provides a very simple and practical framework for achieving adaptation for spatially inhomogeneous functions which, as discussed in Section \ref{besovpriormass} below, goes beyond the regression setting. In particular, from the practical point of view and  unlike \cite[Theorem 2.5]{as22}, approximating posteriors in this framework does not require the use of a Gibbs sampler, which may mix poorly in high dimensions when the data are informative \cite{abps14}, neither does it require implementing marginal maximum likelihood estimators of prior parameters which can be technically difficult, especially in more involved models.


\subsection{Adaptive nonparametric Bernstein--von Mises theorem} 
Beyond convergence rates, one may be interested in the limiting shape of the posterior distribution, as formalised in nonparametric settings in \cite{cn13, cn14}, from which one can deduce, for instance, uncertainty quantification results on certain functionals of $f$, as in the application below Theorem \ref{adnpbvm}.

For monotone increasing weighting sequences $w=(w_l)_{l \ge 0}$, $w_l \ge 1,$ we define multi-scale sequence spaces
\begin{equation} \label{def-multi}
\mathcal M \equiv \mathcal M(w) \equiv \left\{x=\{x_{lk}\}: \ \ \|x\|_{\mathcal M(w)} \equiv \sup_{l}\frac{\max_{k}|x_{lk}|}{w_l} <\infty \right\}.
\end{equation}
The space $\mathcal M(w)$ is a non-separable Banach space (it is isomorphic to $\ell_\infty$). However, the following space $\mathcal M_0$ forms a separable closed subspace for the same norm
\begin{equation} \label{em0}
\mathcal M_0=\mathcal M_0(w)= \left\{x \in \mathcal M(w): \ \ 
\lim_{l \to \infty} \max_k \frac{|x_{lk}|}{w_l}=0\right\}.
\end{equation}
The white noise model \eqref{gwn} can be rewritten as the sequence model $X^{(n)}=f+\mathbb{W}/\rn$, where in slight abuse of notation $f$ is identified with the sequence of its coefficients $f=(f_{lk})$ and $\mathbb{W}=(\int \psi_{lk} dW(t))_{l,k}$  has the distribution of an iid sequence of $\cN(0,1)$ variables. It is not hard to see that $\mathbb{W}$ almost surely belongs to $\mathcal{M}_0$ (and $X^{(n)}$ as well for $f\in L^2$) under the condition that $w_l$ diverges faster than $\sqrt{l}$, see \cite{cn14}. Denote $\tau: f\to \sqrt{n}(f-X^{(n)})$. Then $\Pi[\cdot\given X^{(n)}]\circ \tau^{-1}$ denotes the induced posterior distribution on $\cM_0$, shifted and rescaled by~$\tau$.   
   

For $S$ a given metric space, let $\be_S(P,Q)$ denote the bounded-Lipschitz metric over probability distributions $P,Q$ on $S$. It is well-known that $\be_S$ metrises weak convergence on $S$.  
\begin{theorem} \label{adnpbvm}  
In the regression model, let us consider the heavy-tailed wavelet series prior \eqref{priord} on coefficients $f_{lk}$ in \eqref{normseq2}, with parameters  specified by \eqref{defsigr} and \eqref{conds}--\eqref{condti}. Consider the multiscale space $\cM_0$ as in \eqref{em0} with $w_l=l^{1+\kappa+\veps}$ for some $\veps>0$.   
Suppose $f_0\in \Hbl$ for some $\be, L>0$. Then
\[ E_{f_0} \be_{\cM_0}( \Pi[\cdot\given X^{(n)}]\circ\ta^{-1}, \cL(\mathbb{W})) \to 0\]
as $n\to\infty$, where $\cL(\mathbb{W})$ denotes the law of a Gaussian white noise in $\cM_0$.  
Further, the same result holds for $(s_l)$ as in \eqref{defsig} with $w_l=l^{(1+\kappa+\veps)/2}$, for some $\veps>0$ (and any $\al>0$).
\end{theorem}   

The heavy-tailed priors we consider thus automatically satisfy an adaptive nonparametric Bernstein--von Mises theorem \cite{ray17}. As opposed to results of this type obtained in the literature so far, note that we do not need to modify the prior to impose a {\em flat initialisation}; \cite{ray17} proves that this is necessary for spike--and--slab priors: for these one needs to remove the spikes from a slowly increasing number of coordinates to allow for Gaussian finite-dimensional distributions in the limit (otherwise small signals can erroneously be classified into the spike part by the posterior; the same holds for Bayesian CART, see \cite{cr21}). Here the heavy tailed prior is continuous, so asymptotic normality holds even for arbitrarily low frequencies: this is because the prior induced on the first coordinates has a continuous and positive density on the whole real line, so the conditions of (a version of)  the parametric Bernstein--von Mises theorem are satisfied, see the proof of Theorem \ref{adnpbvm} in the Supplement.\\

{\em Application.} An implication of Theorem \ref{adnpbvm} is the following (using  \cite{cn14}, Theorem 4): a Donsker-type theorem holds for the posterior distribution with heavy-tailed priors when estimating the primitive $F(\cdot)=\int_0^{\cdot} f(u)du$ of $f$: for $B$ standard Brownian motion and $\cL(\|B\|_\infty)$ the distribution of its supremum on $[0,1]$, in $P_0$--probability,
\[ \be_{\RR}\left( \cL(\sqrt{n}\|F(\cdot)-X^{(n)}(\cdot)\|_\infty \given X^{(n)}) , \cL(\|B\|_\infty) \right) \to 0, \]
where $\cL(F(\cdot)\given X^{(n)})$ denotes the posterior distribution on $F$ induced from the posterior on $f$ through the primitive map $f\to F$. 
From this one immediately deduces that supremum-norm quantile credible bands for $F$ are asymptotically optimal (efficient) confidence bands for  $F_0$, see \cite{cn14} for details and discussion.

\section{Prior mass bounds and $\rho$--posterior convergence}\label{sec:prmass}

In this section we first derive lower bounds for the prior mass that heavy-tailed priors put on $L^2$-- and $L^\infty$--balls, around Sobolev and H\"older functions. These bounds next enable us to obtain contraction rates for tempered posteriors ($\rho$-posteriors) in a variety of nonparametric settings: as examples, we consider density estimation, binary classification and regression {(the latter under Besov regularity of the truth)}. 
As a slight variant to  the moment assumption \eqref{qmom}, here we require the tail condition: for some $c_2>0$,
\begin{equation}
\overline{H}(x):=\int_x^{\infty} h(u)du \le c_2/x^2,\qquad x\ge 1. \label{condts}
\end{equation}
Condition \eqref{condts} allows for most Student distributions;  Cauchy tails can also be accommodated, see Remark \ref{rem:cauchy} below and Remark \ref{rem:cauchyinfty} in the Supplement. 


\subsection{Generic prior mass results}
 
\begin{theorem}[Generic prior mass condition in $L^2$] \label{thmpriorm}
For $\Pi$ a prior generated by \eqref{prior},  
\begin{itemize}
\item
suppose $(\sigma_k)$ is as in \eqref{defsig} for $\alpha>1/2$ and assume \eqref{conds}--\eqref{condti}--\eqref{condts}.  For $\be>0$, let 
\begin{equation} \label{rateveps-1}
 \veps_n =  (\log{n})^{\frac{1+(1+\kappa)\be}{2\be+1}} n^{-\frac{\be}{2\be+1}}.
\end{equation}
Then for any $\be\le \alpha$, $L>0$ and $f_0\in \Sbl$, it holds that for any $d_2>0$ there exists $d_1>0$ sufficiently large such that 
\[ \Pi[\|f-f_0\|_2<  d_1\veps_n ]  
\ge e^{ -d_2n\veps_n^2 }.\]
\item suppose $(\sigma_k)$ is defined, for some $a,\delta>0$, by
\begin{equation} \label{fastsig}
\sigma_k = e^{-a (\log{k})^{1+\delta}},
\end{equation}
and assume \eqref{conds}--\eqref{condti}--\eqref{condts}. For $\be>0$, let 
\begin{equation} \label{rateveps-2}
 \veps_n = (\log{n})^{\frac{(1+\kappa)(1+\delta)\be}{2\be+1}} n^{-\frac{\be}{2\be+1}}.
 \end{equation}
Then for any $\be> 0$, $L>0$ and $f_0\in \Sbl$, it holds that for any $d_2>0$ there exists $d_1>0$  sufficiently large such that
\[ \Pi[\|f-f_0\|_2<d_1 \veps_n ]  
\ge e^{ -d_2n\veps_n^2 }.\]
\end{itemize}
\end{theorem}

\begin{remark}[Logarithmic factor] \label{rem:log2}
For the OT prior with $\kappa=0$, the logarithmic term is nearly the best possible one (see Remark \ref{rem:log}) up to a power $C \cdot \delta$ (for some constant $C=C(\be)$ depending on $\be$ only) that can be made  arbitrarily small by taking a small $\delta$. 
\end{remark}  

 
\begin{theorem}[Generic prior mass condition in $L^\infty$] \label{thmpriormlinf}
For $\Pi$ a prior generated by \eqref{priord},  
\begin{itemize}
\item
suppose $(s_l)$ is as in \eqref{defsig} for $\alpha>1/2$ and assume \eqref{conds}--\eqref{condti}--\eqref{condts}.  For $\be>0$, let 
\begin{equation} \label{rateveps-inf-1}
 \veps_n = (\log\log{n})^{\frac{2}{1+2\be}} (\log{n})^{\frac{1+(1+\kappa)\be}{1+2\be}} n^{-\frac{\be}{2\be+1}}.
\end{equation}
Then for any $\be\le \alpha$, $L>0$ 
and $f_0\in \Hbl$, it holds that for any $d_2>0$ there exists $d_1>0$ sufficiently large such that, for large enough $n$,
\[ \Pi[\|f-f_0\|_\infty< d_1\veps_n ]  
\ge e^{ -d_2n\veps_n^2 }.\]
\item
suppose $(s_l)$ is as in \eqref{defsigr} and assume \eqref{conds}--\eqref{condti}--\eqref{condts}.  For $\be>0$, let 
\begin{equation} \label{rateveps-inf-2}
 \veps_n =  (\log{n})^{\frac{(2+2\kappa)\be}{1+2\be}} n^{-\frac{\be}{2\be+1}}.
\end{equation}
Then for any $\be>0$, $L>0$  and $f_0\in \Hbl$, it holds that for any $d_2>0$ there exists $d_1>0$ sufficiently large such that, for large enough $n$,
\[ \Pi[\|f-f_0\|_\infty< d_1\veps_n ]  
\ge e^{ -d_2n\veps_n^2 }.\]
\end{itemize}
\end{theorem}

Prior mass results as obtained in Theorem \ref{thmpriorm}--\ref{thmpriormlinf} are a key preliminary step for obtaining posterior contraction rates. Yet, as noted above, the general theory in \cite{ggv00, gvbook} also typically requires exponentially decreasing prior masses for certain portions of the parameter space (which can then be `tested out'). A major difficulty with (non-truncated) heavy tailed series priors is that prior masses that are exponentially decreasing correspond to extremely small (or `far-away') sets, so usual approaches via entropy control of sieve sets seem out of reach. While we were able in Section \ref{sec:reg} to derive all the results for classical posteriors, here we use instead tempered posteriors: these are defined as, for $0<\rho\le 1$, by
\[ \Pi_\rho[B\given  X] = \frac{\int_B \left(p_f^{(n)}(X)\right)^\rho d\Pi(f)}{\int
\left(p_f^{(n)}(X)\right)^\rho d\Pi(f)}, \]
for measurable $B$. 
The usual posterior corresponds to $\rho=1$ while $\rho<1$ `tempers' the influence of the likelihood. Tempered posteriors require only a prior mass condition to converge \cite{walkerhjort01, tongz06}. Inference in terms of uncertainty quantification can also be conducted with these, see \cite{ltcr23}. We refer 
to the Supplement for more context (Section A therein)  and a precise statement (Section \ref{sec:rhopost}). We also note that the results to follow are obtained for any $\rho<1$ but not for $\rho=1$ (except in white noise regression where results hold for both). Deriving results for $\rho=1$ in general models is beyond the scope of the present work but is an interesting avenue of future research. A detailed discussion on possible approaches to this can be found in the Supplement, Section \ref{sup:seca}. 
 
\begin{remark}
The condition $\al>1/2$ for the HT$(\al)$ prior is a technical condition; it can be checked using similar bounds as in the proof of  Theorem \ref{thmpriormlinf} that, under that condition, a draw $f$ from the prior \eqref{priord} is bounded $\Pi$--almost surely, which in particular ensures that $e^f$ is integrable, a fact used for inducing a density in \eqref{priordens} below.
\end{remark} 

\subsection{Density estimation}\label{ssec:de}

From a prior defined  by \eqref{priord} and \eqref{defsigr}, a prior on densities on $[0,1]$ is easily defined by exponentiation and renormalisation: for $f$ bounded and measurable, let
\begin{equation}\label{priordens}
 g(x) = g_f(x) = \frac{e^{f(x)}}{\int_0^1e^{f(u)}du}.
\end{equation}

\begin{theorem}[density estimation] \label{thmdensity}
Consider data $X=(X_1,\ldots,X_n)$ sampled independently from a density $g_0$ on $[0,1]$ that is bounded away from $0$ and suppose $f_0:=\log{g_0}\in \Hbl$ for some $\be>0$.  Let $\Pi$ be a prior on densities $g$ generated by \eqref{priordens}, with parameters on the prior on $f$ as in the statement of Theorem \ref{thmpriormlinf} (e.g.  $\al>1/2$ for the HT$(\al)$ prior) with corresponding rates $\veps_n$ as in \eqref{rateveps-inf-1} or \eqref{rateveps-inf-2}. For any  $\rho<1$, there exists $M=M(\rho)>0$ with
\[ E_{g_0} \Pi_\rho[ \|g-g_0\|_1 > M\veps_n\given X] \to 0 \]
as $n\to\infty$, and where $P_0=P_{g_0}$.
\end{theorem}

Theorem \ref{thmdensity} derives adaptive posterior contraction for the $\rho$--posterior at minimax rate (up to a logarithmic term) in density estimation. The term $(1-\rho)^{-1}$ is expected since it is known that usual posteriors may not converge without entropy and sieve conditions. This result can be seen as a counterpart for heavy-tailed priors to results for usual posteriors in density estimation for Gaussian priors \cite{vvvz09}, recently obtained also for Laplace priors in \cite{giordano22}. Sampling from tempered posteriors is generally of comparable difficulty compared to classical posteriors, and since there is no hyper-posterior on the smoothness parameter to sample from with the considered heavy tailed priors, sampling in density estimation is relatively easy and carried out in Section \ref{sec:sim} -- and this even though the model does not tensorise over coordinates.


\subsection{Classification} \label{sec:classif}

Consider independent observations $(X_1,Y_1),\ldots,(X_n,Y_n)$ from a given  distribution of a random variable $(X,Y)$, where $Y\in\{0,1\}$ is binary and $X$ takes values in $\cX=[0,1]^d$ for $d\ge 1$. The interest is in estimating the binary regression function $h_0(x)=P(Y=1\given X=x)$. 
Consider the logistic link function $\La(u)=1/(1+e^{-u})$ and denote its inverse by $\La^{-1}$. From a function $f$ sampled from \eqref{priord}--\eqref{defsig} (or \eqref{priord}--\eqref{defsigr}), setting
\begin{equation} \label{priorclassif}
h_f(x) = \La( f(x) ) 
\end{equation}
induces a prior distribution $\Pi$ on binary regression functions. The density of the data $(X,Y)$ given $f$ equals $p_f(x,y)=h_f(x)^y(1-h_f(x))^{1-y}g(x)$, where $g(x)$ denotes the marginal density of $X$. Denote by $\|\cdot\|_{G,1}$ the $L^1(G)$ norm on $\cX$ and by $P_0$ the true data generating distribution with regression $f_0$ and marginal $g$ (note that Bayesian modelling of $g$ is not needed for inference on $f$ as it factorises from the likelihood).

\begin{theorem}[classification] \label{thmclassif}
Consider data $(X,Y)$ from the binary classification model. Suppose $f_0=\La^{-1}h_0$ belongs to the Sobolev ball $\Sbl$ for $\be, L>0$. Let $\Pi$ be a prior generated by \eqref{priorclassif}, with parameters on the prior on $f$ as in the statement of Theorem \ref{thmpriorm} with corresponding rates $\veps_n$ as in \eqref{rateveps-1} or \eqref{rateveps-2}. Then for any given $\rho<1$, there exists $M=M(\rho)$ such that, as $n\to\infty$,
\[ E_{P_0} \Pi_\rho[ \|p_f-p_{f_0}\|_{G,1} > M \veps_n\given X_1,Y_1,\ldots,X_n,Y_n] \to 0. \]
\end{theorem}

Theorem \ref{thmclassif} derives adaptation for binary classification. Once again, simulation from (an approximation of) the $\rho$--posterior can be carried out using a direct MCMC method without the need of hyperparameter sampling, see Subsection \ref{bin:sim} 
in the Supplement for details.


%
%

\subsection{Besov classes}\label{besovpriormass}

We now provide results for possibly spatially inhomogeneous functions and $\rho$--posteriors. We restrict for simplicity to white noise regression and to variances as in \eqref{defsigr}. 
The following theorem shows that Theorem \ref{thm-besov} also holds for $\rho$--posteriors, $\rho<1$. 


\begin{theorem} \label{thm-besov-rho}
In the white noise regression model, for  $\Pi$ a prior generated by \eqref{priord} and \eqref{defsigr}, assume \eqref{conds}--\eqref{condti}--\eqref{condts} hold. 
 Suppose $f_0\in \Bblr$ for some $\be, L>0$, $r\in[1,2]$ and $\be>1/r-1/2$.  Then, for any given $\rho<1$,  as $n\to\infty$,
\[ E_{f_0}\Pi_\rho\left[ \{f:\ \|f-f_0\|_2 > \cL_n  n^{-\frac{\be}{2\be+1}} \} \given X^{(n)}\right] \to 0, \]
where $\cL_n= (\log{n})^{d}$ for some $d>0$.
\end{theorem}
We observe excellent empirical behaviour in simulations of the corresponding posterior distributions in terms of adaptation and signal fit on a variety of inhomogeneous test signals, see Section \ref{sim:spin}
 in the Supplement. 
A proof of Theorem \ref{thm-besov-rho} could be given following similar arguments as for Theorem  \ref{thm-besov} (i.e. relying on the approach of Theorems \ref{thm-seq}--\ref{thm-seq-sup}). The proof we provide in the Supplement relies on prior mass arguments. Indeed, the latter are easier to generalise to more other settings (such as density estimation as above) modulo slight adaptation of the conditions. A more systematic study of convergence in Besov spaces in different models and for different losses is beyond the scope of the present work and is the object of forthcoming work.

\section{A simulation study}
\label{sec:sim}

We consider the following four simulation settings
\begin{enumerate}[a)]
\item inverse regression with Sobolev/spatially homogeneous truth,
\item spatially inhomogeneous truth in white noise regression,
\item density estimation with H\"older--regular truth, 
\item binary classification with Sobolev--regular truth. 
\end{enumerate}
Here we present the setting a) and a simulation for c) for illustration, and refer to the Supplement, Section \ref{sec:adsim} 
for more details. 

{\em Inverse regression.} We consider the model studied in \cite[Section 3]{ksvv16} and \cite[Section 4]{svv15}, where one observes the process
\[X_t=\int_0^t\int_0^s f(u)duds+\frac1{\sqrt{n}}B_t, \quad t\in[0,1],\]
for $B_t$ a standard Brownian motion and $f\in L^2[0,1]$ the unknown function. This is a linear inverse problem with the Volterra integral operator $Kf(t)=\int_0^tf(u)du$ as forward operator, which has eigenfunctions $e_k(t)=\sqrt{2}\cos(\pi(k-1/2)t)$ and corresponding eigenvalues $\kappa_k=\pi/(k-1/2)$, for $k\ge 1$. Equivalently, we study the normal sequence model 
\[X_k|f_k \stackrel{ind}{\sim} \cN(\kappa_k f_{k}, 1/n), \quad k\ge 1,\]
where $f_k$ are the coefficients of the unknown with respect to the orthonormal system formed by the eigenfunctions $(e_k)$, so that we are in the setting of Subsection \ref{ssec:ip}. As underlying truth we use a function with coefficients with respect to $(e_k)$ given by $f_{0,k}=k^{-3/2}\sin(k)$. In particular, the truth can be thought of as having Sobolev regularity (almost) $\be=1$. 

We consider priors on the coefficients of the unknown of the form $f_k=\sigma_k\zeta_k$ for i.i.d. $\zeta_k$, for three different choices of the standard deviations $\sigma_k$ and/or the distribution of $\zeta_1$:
\begin{itemize}
\item Gaussian hierarchical prior: $\sigma_k=k^{-1/2-\alpha}$ with $\alpha\sim {\rm Exp}(1)$, $\zeta_1$ standard normal;
\item HT$(\al)$ prior: $\sigma_k=k^{-1/2-\alpha}$ with $\alpha=5$, $\zeta_1$ Student distribution with 3 degrees of freedom;
\item OT prior: $\sigma_k=e^{-a(\log k)^{1+\delta}},$ with $a=1, \delta=0.5$ and $\zeta_1$ again a Student $t_3$ distribution.
\end{itemize}
Note that, on purpose, in order to test the robustness of the method, we take a very `unfavourable' $\al$ for the HT$(\al)$ prior, with $\al=5$ much larger than the true smoothness $\be=1$ here. Furthermore, for the OT prior we use $\delta=0.5$ instead of  $\delta=1$ used in our analysis. As noted in Section \ref{sec:intro}, the contraction rates are identical for any $\delta>0$, however we found `the finite' $n$ behaviour to be slightly better for $\delta=0.5$ compared to $\delta=1$ (although the difference seemed relatively small in all conducted experiments), so we kept this choice through the simulations.

To sample the posterior arising from the Gaussian hierarchical prior we employ a Metropolis-within-Gibbs sampler which alternates between updating the $\alpha|f,X$ and $f|\alpha, X$ (with an appropriate parametrization, centered or non-centered depending on the size of the noise, to optimize the mixing of the $\alpha$-chain, see \cite{abps14}). For the two Student priors, due to independence, the posterior decomposes into an infinite product of univariate posteriors. We use Stan, with random initialization uniformly on the interval $(-2,2)$, to sample each of the univariate posteriors \cite{stan} 
(it is also possible to code this manually e.g. via a  Random Walk Metropolis). In all three cases, we truncate at $K=200$, which, for the considered regularities of the truth and the priors, suffices for the truncation error to be of lower order compared to the estimation error. 

In Figure \ref{fig-postSob}, we present posterior sample means as well as 95\% credible regions for various noise levels, computed by taking the 95\% out of the 4000 draws (after burn-in/warm up) which are closest to the mean in $L^2$-sense. {The OT prior appears to {perform at least as well as} the Gaussian hierarchical prior at all noise levels both in terms of the posterior sample mean as well as uncertainty quantification (this is expected from the theory, since the OT posterior is guaranteed to converge (near)-optimally for both quadratic and supremum norm).} {For the HT$(\al)$ prior, although $\al$ is very far off the true smoothness, we see that as $n$ increases the posterior is still able to approximately match the unknown truth. }
Compared to the OT prior, in this setting the HT$(\al)$ prior appears to be overconfident in all but the lowest noise levels. This is a `finite $n$' phenomenon, which can be explained by the behaviour of the univariate Student prior in the model $X\sim \cN(f,1/n)$ as detailed in Section \ref{sec:reg} (but adapted to accommodate $\kappa_k$): since $\kappa_k\sigma_k=k^{-6.5}$ becomes very small already for small $k$, among coefficients with small signal-to-noise ratio, very few get a significant value under the posterior
and hence there is very little variance in the posterior. Although asymptotically for $k\to\infty$ the scalings of the OT prior decay even faster, $\kappa_k\sigma_k=k^{-1}e^{-(\log k)^{3/2}}$ remains large (relatively to $1/\sqrt{n})$) for more frequencies, {hence more frequencies get a significant value under the posterior and the posterior on the function $f$ exhibits more variability.} 

In Section \ref{sec:adsim}
of the Supplement we additionally study $\rho$-posteriors in this setting for the two Student priors, with similar conclusions as for the classical posteriors. We also compare the behaviour of the HT$(\al)$ prior and a corresponding Gaussian process prior both with $\al=5$, as an illustration of the `tail--adaptation' property which takes place for the HT prior but, as expected, not for the Gaussian prior.

\begin{figure}
    \centering
     \includegraphics[width=0.97\textwidth]{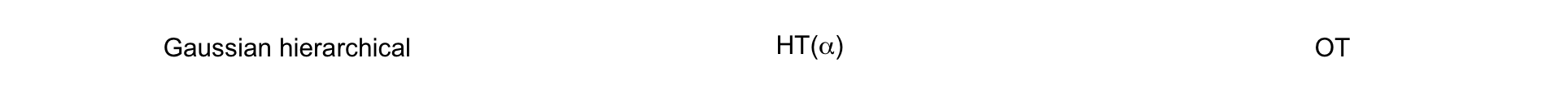}
    \includegraphics[width=0.97\textwidth]{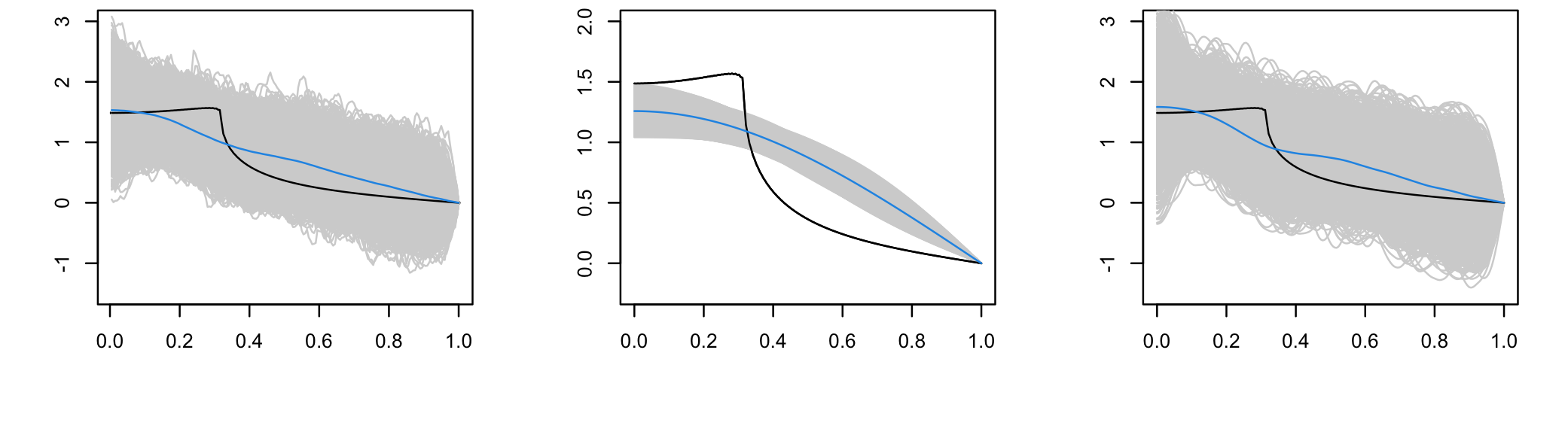}
    \includegraphics[width=0.97\textwidth]{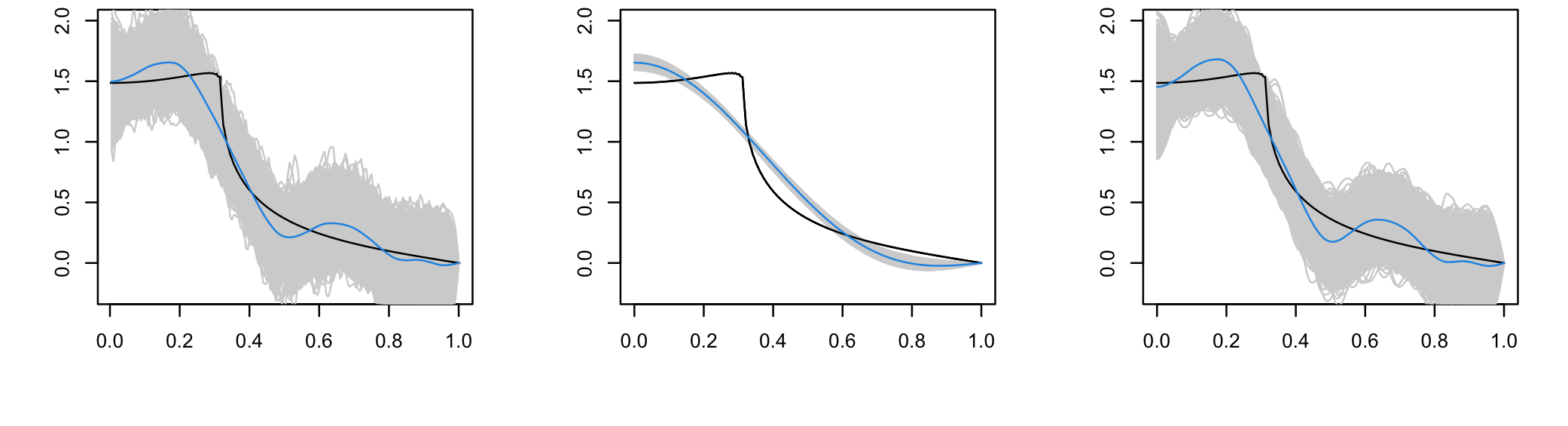}
    \includegraphics[width=0.97\textwidth]{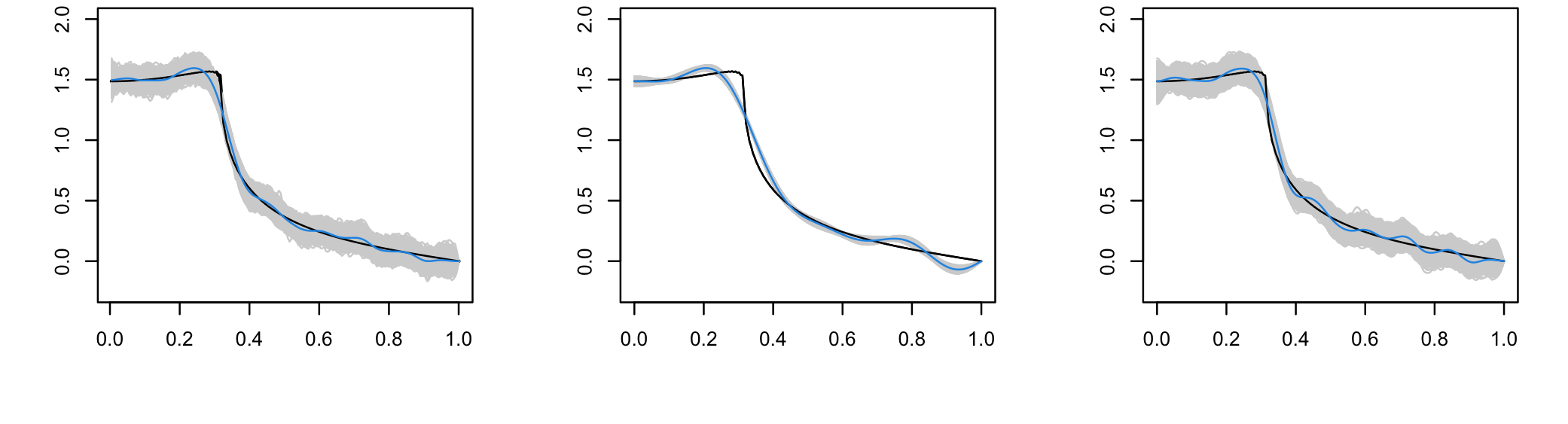}
    \includegraphics[width=0.97\textwidth]{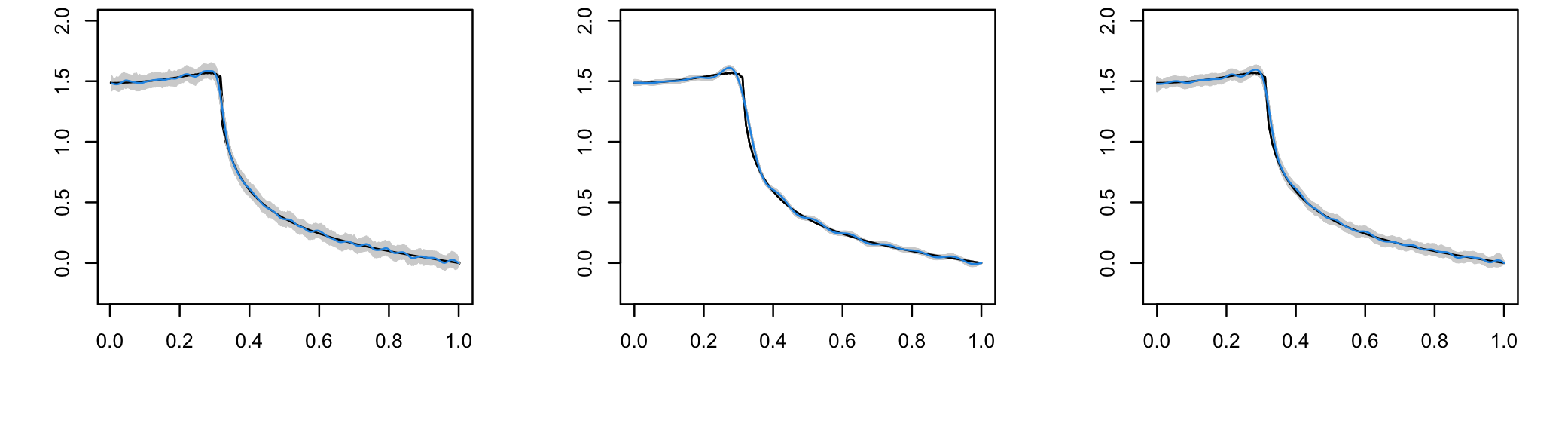}
    \includegraphics[width=0.97\textwidth]{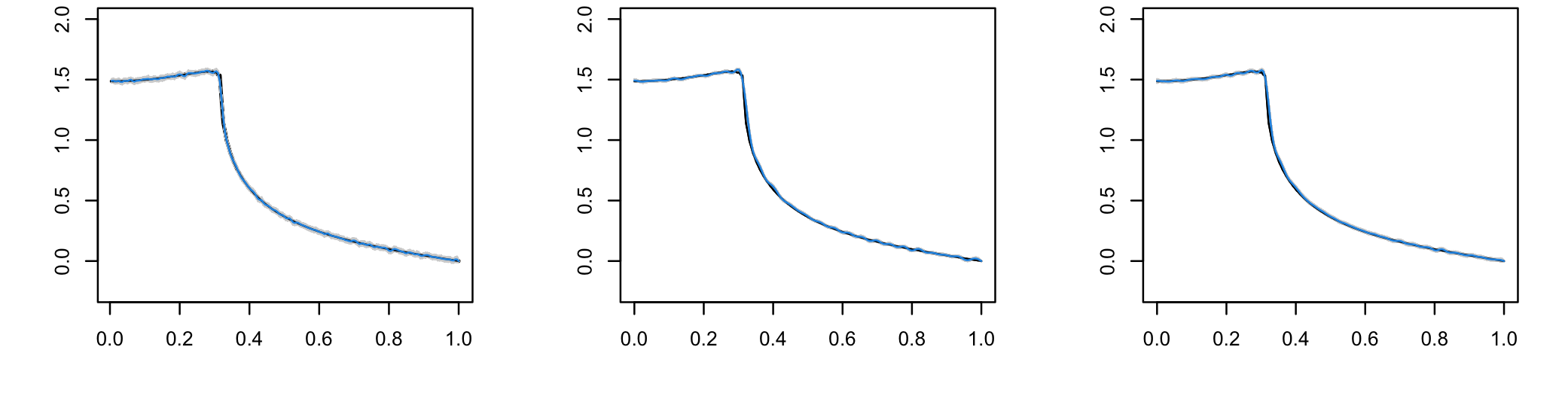}
    \caption{White noise model: true function (black), posterior mean (blue), 95\% credible regions (grey), for $n=10^3, 10^5, 10^7, 10^9, 10^{11}$ top to bottom and for the three considered priors left to right.}
    \label{fig-postSob}
\end{figure}

{\em Density Estimation.} We consider the density estimation setting of Subsection \ref{ssec:de}, for a true density $g_0$ defined via \eqref{priordens} with $f_0$ a 2-H\"older smooth function defined via its coefficients in a certain wavelet basis. We consider an $\alpha$-smooth Gaussian prior, and Cauchy HT$(\al)$ and OT priors, where $\alpha=5$.  In Figure \ref{fig-postSob-de} we present the corresponding posteriors. Even though we are not in a product space, hence we have to use a function space MCMC algorithm, Cauchy priors show excellent performance without the requirement of a Gibbs Sampler for sampling a smoothness hyper-parameter. For more details, additional experiments and a comparison to a standard frequentist estimator, see Section \ref{sec:adsim} 
in the Supplement. The code for all our experiments with heavy-tailed priors is available at \url{https://www.mas.ucy.ac.cy/sagapi01/assets/code/code-HT-BNP-adapt.zip}.

\begin{figure}
    \centering
    \includegraphics[width=0.92\textwidth]{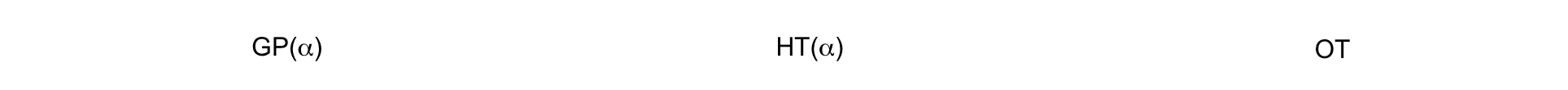} \quad

    \includegraphics[width=0.3\textwidth]{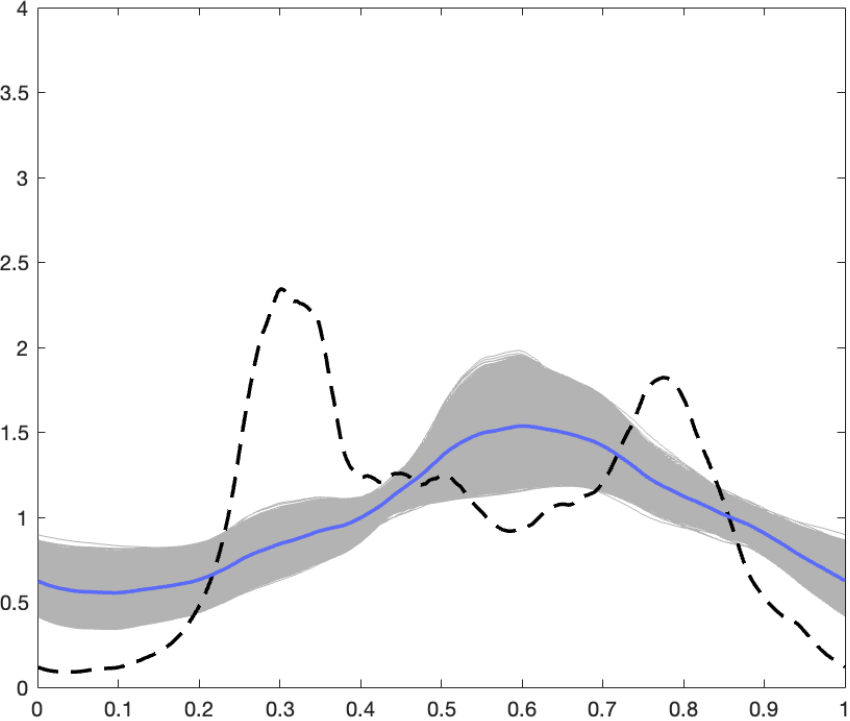}\;
        \includegraphics[width=0.3\textwidth]{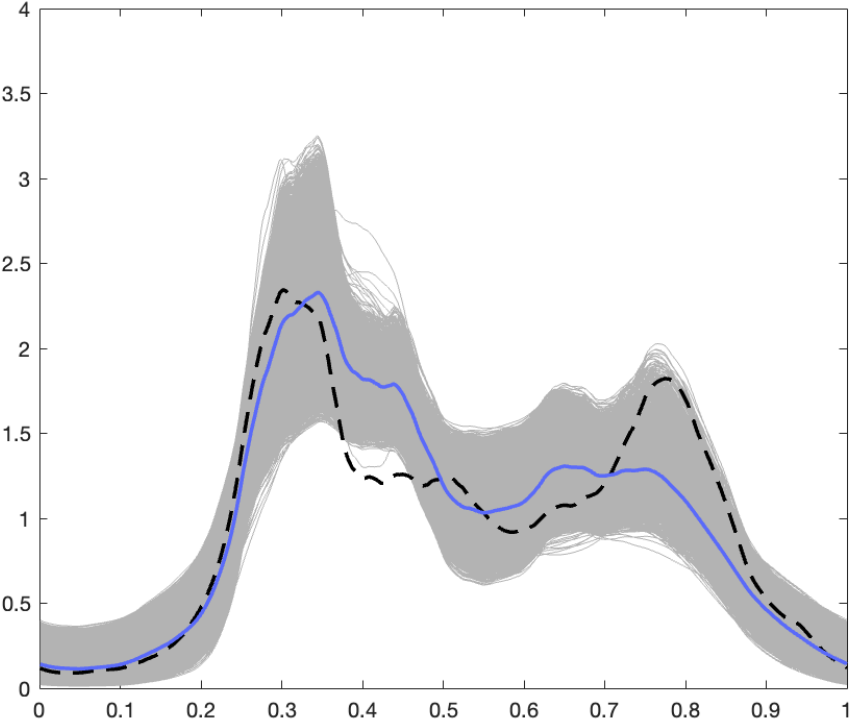}\;
    \includegraphics[width=0.3\textwidth]{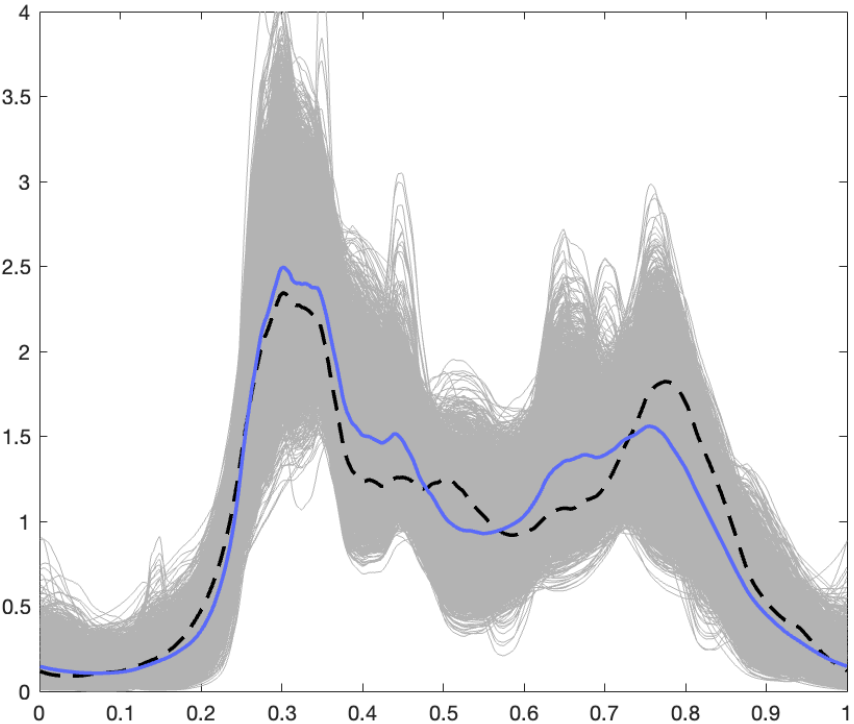}\vspace{0.4cm}

    \includegraphics[width=0.3\textwidth]{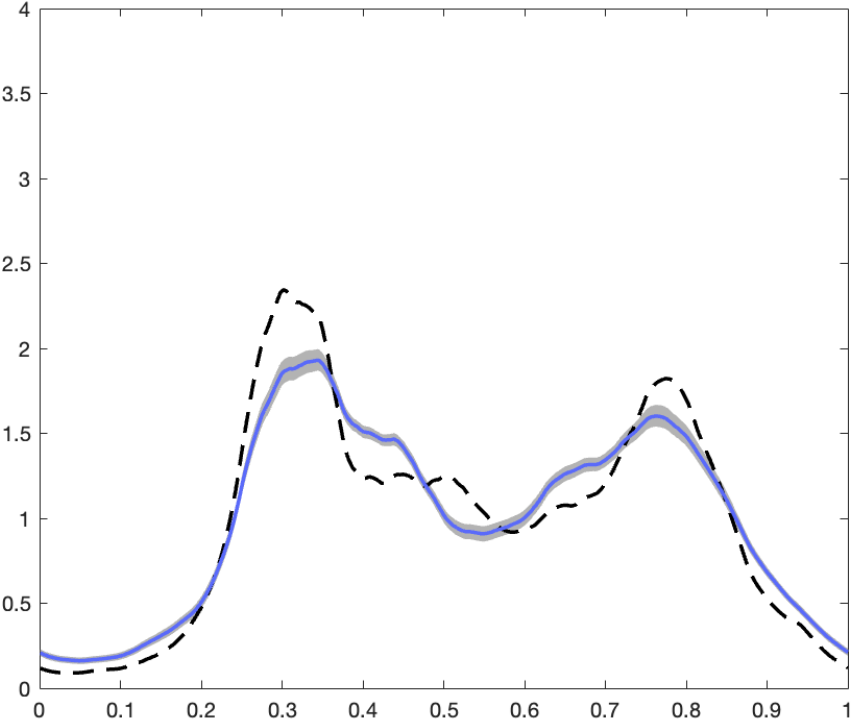}\;
        \includegraphics[width=0.3\textwidth]{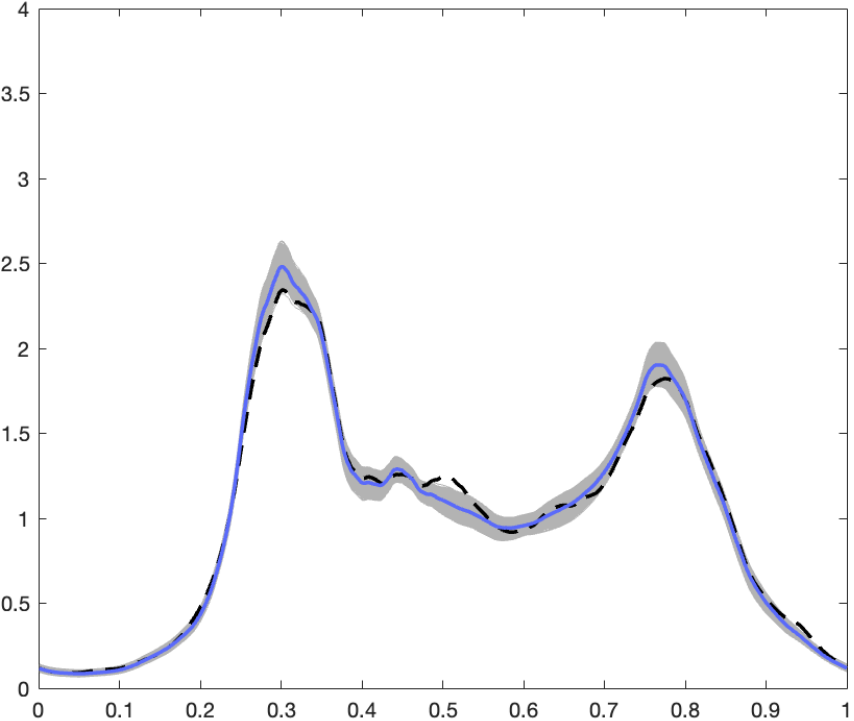}\;
    \includegraphics[width=0.3\textwidth]{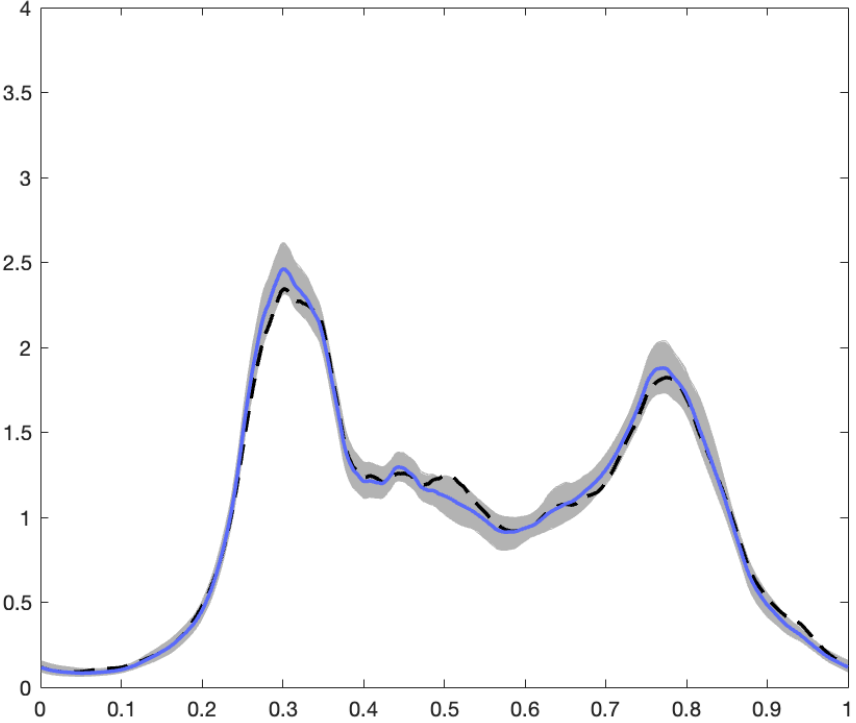}\vspace{0.4cm}
    
        \includegraphics[width=0.3\textwidth]{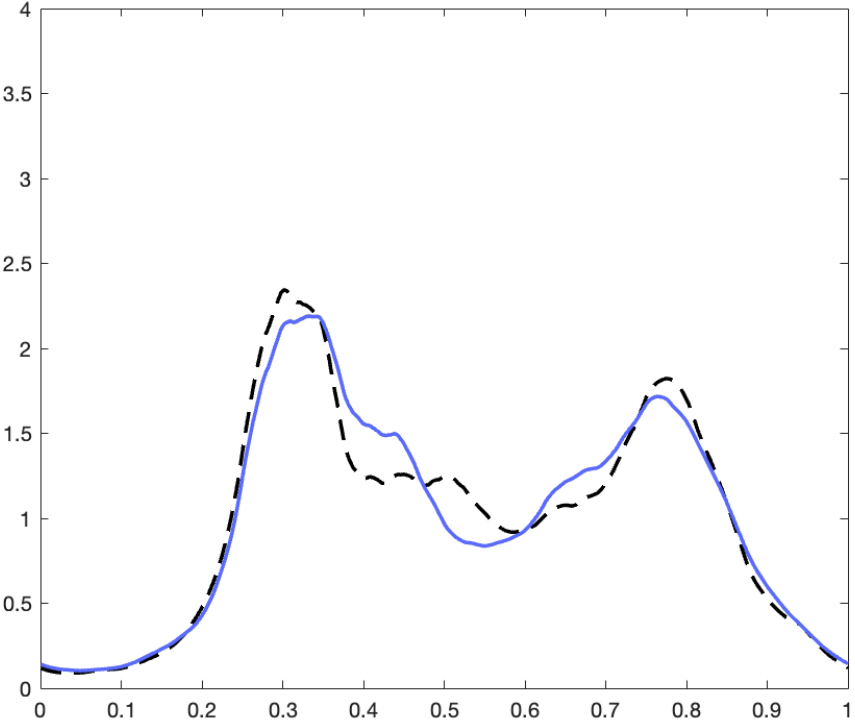}\;
        \includegraphics[width=0.3\textwidth]{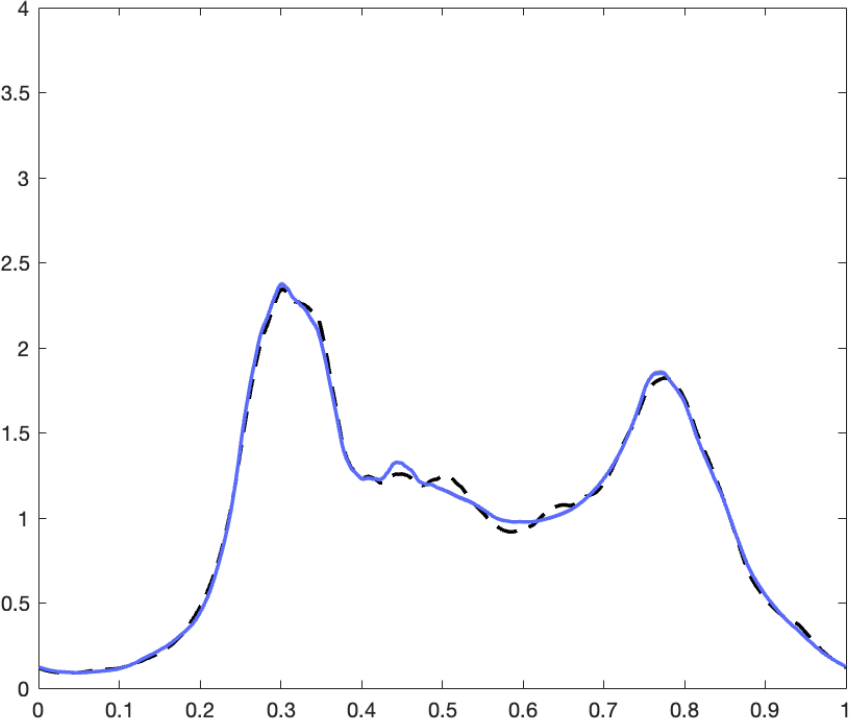}\;
    \includegraphics[width=0.3\textwidth]{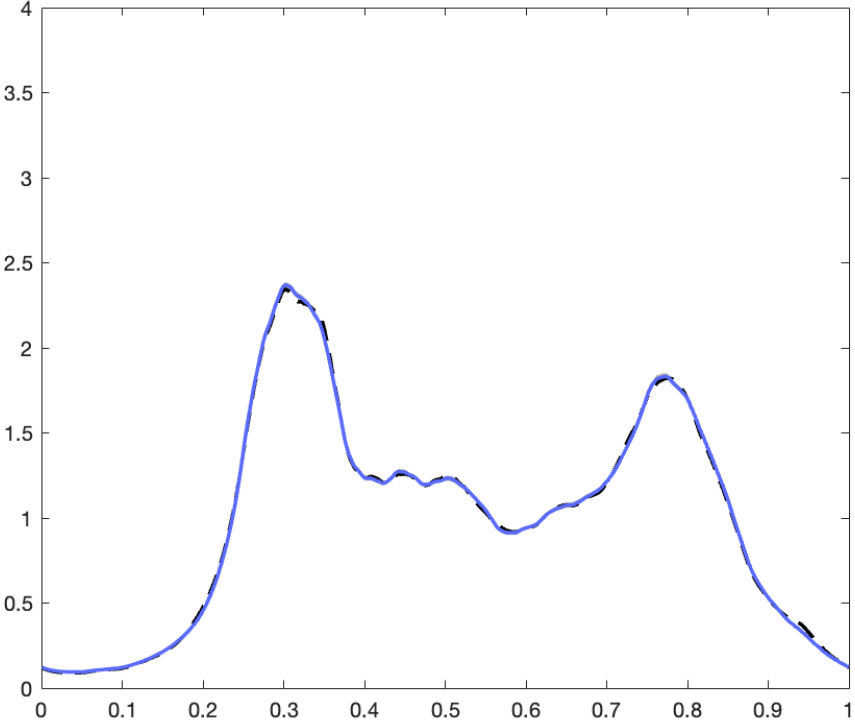}

    \caption{Density estimation: true density (black dashed), posterior mean (blue), 95\% credible regions (grey), for $n=10^2, 10^4, 10^6$  top to bottom and for the three considered priors left to right.}
    \label{fig-postSob-de}
\end{figure}

\section{Discussion} \label{sec:disc}

We have introduced a new prior, the {\em oversmoothed heavy-tailed} (OT) prior, which we show leads to Bayesian nonparametric adaptation to smoothness in a wide array of settings. One main appeal is that it is only defined from the basis coefficients one wishes to model, without extra need of hyper-parameters to derive adaptation.  

While some prior classes can be seen as Bayesian analogues of nonparametric adaptation methods (e.g. sieve priors$\leftrightarrow$model selection, spike--and--slab$\leftrightarrow$thresholding), the OT prior achieves adaptation through the prior distribution's (heavy) {\em tails}, so in a sense is distinctively Bayesian in spirit.

{\em Open directions.} While this work proposes a novel class of adaptive priors and closes a gap in the literature by showing that the phase transition from light (e.g. Gaussian) and moderate (e.g. exponential) tails in infinite series priors to heavy-tails comes with obtaining automatic adaptation, it opens a number of questions for future work. First, the fact that the prior performs a `soft' model selection -- by this we mean that many coefficients are close to zero but not exactly zero under the posterior -- is of interest for complex models, where performing a `hard' model selection (as is the case for spike and slab priors that set coefficients exactly to $0$) can be computationally intensive. We are currently working on adaptation on deep ReLU neural networks using such heavy-tailed priors; another interesting direction is that of sparse settings.  Among computationally less demanding alternatives to hard model selection, let us also mention the possibility to use a variational Bayes approach, see e.g. \cite{zg20, ar20, ypb20}, although then approximating a different object than the posterior or tempered posterior itself. Second, we conjecture that the results of Section \ref{sec:prmass} also hold for classical posteriors. This is in particular supported by numerical evidence from Section \ref{sec:sim}. Although beyond the scope of the present contribution, this suggests that one could possibly extend the general posterior rates theory \cite{ggv00, gvbook} beyond cases where exponential decrease of sieve set probabilities is available. 
Third, it would also be particularly interesting to study the computational complexity of heavy-tailed priors in the spirit of the recent works   \cite{nicklwang22, BMNW23}.

\section{Proof of the main results} \label{sec:proof}

Here we prove Theorems \ref{thm-seq} and \ref{thmpriorm}, and related technical lemmas. Proofs of the remaining Theorems can be found in the Supplement. 

\begin{proof}[Proof of Theorem \ref{thm-seq}]
Let $\veps_n=n^{-\be/(2\be+1)}$ and  $K_n$ be the (closest integer) solution to 
\begin{equation} \label{defkn}
K_n=n^{1/(2\be+1)}.
\end{equation}
For 
 $\cA_n$ a suitable event to be defined below, using Markov's inequality,
\[ E_{f_0}\Pi[\|f-f_0\|_2>v_n \given X^{(n)}] \le P_{f_0}(\cA_n^c) +
 v_n^{-2}E_{f_0}\left[ \int \|f-f_0\|_2^2 d\Pi(f\given X^{(n)})\1_{\cA_n}\right],   \]
where we choose $v_n=\cL_n \veps_n$. For $f\in L^2$, let  $f^{[K_n]}$ denote its orthogonal projection onto the linear span of the first $K_n$ basis vectors and set $f^{[K_n^c]}:=f-f^{[K_n]}$. Then
\[ \|f-f_0\|_2^2 = \|f^{[K_n]}-f_0^{[K_n]}\|_2^2 + \|f^{[K_n^c]}-f_0^{[K_n^c]}\|_2^2 
\le \|f^{[K_n]}-f_0^{[K_n]}\|_2^2+2\|f^{[K_n^c]}\|_2^2+2\|f_0^{[K_n^c]}\|_2^2.\]
By definition of $\Sbl$, we have $\|f_0^{[K_n^c]}\|_2^2\leqa K_n^{-2\be}\leqa \veps_n^2$. Next 
\[\int  \|f^{[K_n]}-f_0^{[K_n]}\|_2^2 d\Pi(f\given X^{(n)})\le2\sum_{k\le K_n} \int (f_k-X_k)^2 d\Pi(f\given X^{(n)})+ 2\sum_{k\le K_n}(X_k-f_{0,k})^2.\] 
 Under $E_{f_0}$, the second sum on the right hand side is bounded by $K_n/n$. For the first term,
Lemma \ref{laptr} used on coordinate $k$ combined with Lemma \ref{lad} gives, for any $t>0$,
\[ nE_{f_0}\int (f_k-X_k)^2 d\Pi(f\given X^{(n)}) \leqa t^{-2}\left\{1+t^4+\log^{2+2\kappa}
\left(1+\frac{L+1/\sqrt{n}}{\sigma_k} \right)\right\}.\]
 This is optimised in $t$ by  taking $t^4$ of the order of the log term in the last display, leading to 
\begin{equation}\label{eq:logbnd}  nE_{f_0}\int (f_k-f_{0,k})^2 d\Pi(f\given X^{(n)})\leqa 
\log^{1+\kappa}
\left(1+\frac{L+1/\sqrt{n}}{\sigma_k} \right). \end{equation}
Using $\sigma_k^{-1}\le \sigma_{K_n}^{-1}$, the last term is at most of logarithmic order in $n$ for both choices of $\sigma$, so 
\[ E_{f_0} \int \|f^{[K_n]}-f_0^{[K_n]}\|_2^2 d\Pi(f\given X^{(n)}) \leqa (\log{n})^d \sum_{k=1}^{K_n} \frac{1}{n} \leqa (\log{n})^d K_n/n \leqa (\log{n})^d n^{-\frac{2\be}{2\be+1}}. \]

Let us now turn to bounding the term $\int \|f^{[K_n^c]}\|_2^2 d\Pi(f\given X)$. 
We first claim that it is enough to focus on the set  of indices $k$ for which $|f_{0,k}|\le 1/\sqrt{n}$. Indeed, if $N_n$ is the cardinality of 
\begin{equation}\label{enn}
\cN_n:=\left\{k:\ |f_{0,k}|> 1/\sqrt{n} \right\},
\end{equation}  
we have $L^{2}\ge \sum_{k\in\cN_n} k^{2\be}f_{0,k}^2 \ge n^{-1}\sum_{k=1}^{N_n}k^{2\be}\asymp N_n^{2\be+1}/n$, so that $N_n\leqa K_n$. Further, if $k\in \cN_n$ then $k\le (L^{2}n)^{1/(2\be)}$. This means that for any index $k\in \cN_n$, one can use bound \eqref{eq:logbnd} (which holds for any $k$) similarly to the case $k\le K_n$, just using $\sigma_k^{-1}\le \sigma_{(L^{2}n)^{1/(2\be)}}^{-1}$ this time, giving a bound in \eqref{eq:logbnd} which is logarithmic in $n$. In other words, for some $d'>0$,
\[
 E_{f_0} \int \sum_{k\in \cN_n} (f_k-f_{0,k})^2 d\Pi(f\given X^{(n)}) \leqa (\log{n})^{d'}  \sum_{k\in \cN_n} \frac{1}{n} \leqa (\log{n})^{d'} \frac{K_n}{n} \leqa (\log{n})^{d'} n^{-\frac{2\be}{2\be+1}}. 
 \]
Further note that for both choices of $\sigma$'s \eqref{defsig} and \eqref{defsigr},  for any $k>K_n$, one has $\sigma_k \leqa 1/\sqrt{n}$: this results from the definitions, using $\alpha\ge \be$ for the choice \eqref{defsig}.
 
We now focus on indices $k\in
\{k>K_n :\ |f_{0,k}|\le 1/\sqrt{n} \}$ and  bound $E_{f_0} \int f_k^2 d\Pi(f\given X)$. Using Bayes' formula, it is enough to bound the following terms individually, where $\phi$ denotes the standard normal density,
\[ \int f_k^2 d\Pi(f\given X^{(n)}) = \frac{ \int \te^2 \phi(\sqrt{n}(X_k-\te))h(\te/\sigma_k)d\te }{ \int  \phi(\sqrt{n}(X_k-\te))h(\te/\sigma_k)d\te } =: \frac{N}{D}.  \]
To bound the numerator $N$, we use $|\phi|\le \|\phi\|_\infty$ and $\int \te^2 h(\te/\sigma_k)d\te = \sigma_k^3\int u^2h(u)du\leqa \sigma_k^3$, using \eqref{qmom} with $q=2$, 
  so that $N\leqa \sigma_k^3$ regardless of $X_k$. 

The denominator is bounded as follows. By symmetry of both $\phi$ and $h$,  it is enough to focus on the case $X_k\ge 0$ (denoting $D=D(X_k)$, we have $D=D(-X_k)$). 

We first deal with the  case of super-light variances \eqref{defsigr}. Let $(x_k)$ be a deterministic nonnegative sequence to be chosen. By restricting the denominator to the set $[X_k-x_k,X_k+x_k]$, 
 \[ D\ge \phi(\sqrt{n}x_k) \int_{X_k-x_k}^{X_k+x_k} h(\te/\sigma_k)d\te. \]
 Assuming $X_k\le x_k$, the integral in the last display can be further bounded from below by $\int_0^{x_k} h(\te/\sigma_k)d\te=\sigma_k \int_0^{x_k/\sigma_k} h(u)du$, recalling $X_k\ge 0$. Further assuming that $\sigma_k\leqa x_k$, the latter integral is further bounded below,  for some $c>0$, by $\int_0^c h(u)du\geqa 1$. Putting everything together and using symmetry, one gets, if $\sigma_k\leqa x_k$,
 \[  \frac{N}{D}\1_{|X_k| \le x_k} \le \sigma_k^2 \frac{\1_{|X_k|\le x_k}}{\phi(\sqrt{n}x_k)}. \]
 Define the events, for $j\ge 0$, $k\ge 1$ and with $a\vee b=\max(a,b)$,
 \begin{align}
 \cA_{k,j} & := \left\{ \, |X_k| \le \sqrt{\frac{4\log\{n(j^2\vee 1)\}}{n}}\, \right\},  \label{defevkj}\\
 \cA_n(\cN_n) & := \bigcap_{K_n<k\le n\,,\, k\in\cN_n^c} \,\cA_{k,0}\
  \cap\ \bigcap_{j\ge 1}\ \bigcap_{jn < k \le (j+1)n\,,\, k\in\cN_n^c} \,\cA_{k,j}.\label{defev}
 \end{align}
Let us  set $x_k=\left(4n^{-1}\log\{n(j^2\vee 1)\}\right)^{1/2}$ whenever $jn<k\le (j+1)n$. The  constraint $\sigma_k\leqa x_k$ is trivially satisfied when $k>K_n$ for this choice of $x_k$ and large enough $n$ since then $\sigma_k\leqa 1/\sqrt{n}$  as noted above. 
 Then, with $\cA_n=\cA_n(\cN_n)$, using that $(\sigma_k)$ is decreasing,
\begin{align*} 
E_{f_0} & \sum_{k>K_n\,,\, k\in \cN_n^c} \int f_k^2 d\Pi(f\given X^{(n)})\1_{\cA_n}\\
& \le \sum_{K_n<k\le n\,,\,k\in\cN_n^c} \frac{\sigma_k^2}{\phi(\sqrt{n}x_0)}+ \sum_{j\ge 1} \sum_{jn < k \le (j+1)n\,,\, k\in\cN_n^c} \frac{\sigma_{k}^2}{\phi(\sqrt{n}x_k)}\\
& \le n\sigma_{K_n}^2/\phi(\sqrt{n}x_0)+\sum_{j\ge 1} \sum_{jn < k\le (j+1)n\,,\, k\in\cN_n^c} \sigma_{jn}^2/\phi(\sqrt{n}x_{(j+1)n}) \\
& \leqa n^3\sigma_{K_n}^2 + \sum_{j\ge 1}n\sigma_{jn}^2 (n j^2)^2 
 \leqa n^3
\Big(e^{-2\{\log K_n\}^2}+ \sum_{j\ge 1} e^{-2\{\log(nj)\}^2}j^4\Big).
\end{align*}
The latter bound is $o(n^{-M})$ for arbitrary $M>0$, which, combined with Lemma \ref{lemev}, concludes the proof for $\sigma_k$ as in \eqref{defsigr}. 

Let us now turn to the case of variances as in \eqref{defsig}. In this case, one slightly updates the definition of $\cN_n$ by choosing 
\begin{equation} \label{emn}
\cM_n=\{k:\, |f_{0,k}|> \delta_n/\sqrt{n}\},
\end{equation}
 with $\delta_n:=1/\sqrt{\log{n}}$. In slight abuse of notation we still denote $\cA_n=\cA_n(\cM_n)$ below. Note that as above one can first deal with indices $k\in \cM_n$ using the same bounds as before; reasoning as above, there are at most $m_n\le (n\log{n})^{1/(2\be+1)}\leqa K_n (\log{n})^{1/(2\be+1)}$ such indices, so their overall contribution to the quadratic risk is within a logarithmic factor of $K_n/n$.

Thus it is enough to focus on the indices $k>K_n$ and $k\notin\cM_n$. If $\sqrt{n}|X_k|\le 1$, then the above bounds for $N, D$ can be used with $1/\sqrt{n}$ in place of $x_k$, leading to, for $k>K_n$, 
\[  \frac{N}{D}\1_{\sqrt{n}|X_k| \le 1} \leqa \sigma_k^2 \phi(1)^{-1}. \]
One splits, recalling the definition of $x_k=\left(4n^{-1}\log\{n(j^2\vee 1)\}\right)^{1/2}$,
\begin{align*}
\frac{N}{D}\1_{\sqrt{n}|X_k| > 1}\1_{|X_k| \le x_k} & = \sum_{p\ge 1} \frac{N}{D}\1_{\sqrt{p}< \sqrt{n}|X_k| \le \sqrt{p+1}}  \1_{|X_k| \le x_k} \\
& \le 
\sum_{p=1}^{4\log(n(j^2\vee 1))} \frac{N}{D}\1_{\sqrt{p}< \sqrt{n}|X_k| \le \sqrt{p+1}}.  
\end{align*}
We bound from below $D$, on the event $\{\sqrt{p}< \sqrt{n}|X_k| \le \sqrt{p+1} \}$, as follows, for $X_k\ge 0$,
\begin{align*}
D\ge \phi(\rn X_k) \int_0^{X_k} h(\te/\sigma_k)d\te\ge \sigma_k \phi(\sqrt{p+1}) \int_0^{X_k/\sigma_k} h(u)du\geqa  \sigma_k  \phi(\sqrt{p+1}),
\end{align*}
by restricting the denominator to $[0,X_k]$ and 
using $X_k/\sigma_k\ge 1/(\sqrt{n}\sigma_k)\geqa 1$ for $k>K_n$. By symmetry, the same bound also holds in case $X_k\le 0$. So, for given $k$,
\[  \sum_{p=1}^{4\log(n(j^2\vee 1))} \frac{N}{D}\1_{\sqrt{p}< \sqrt{n}|X_k| \le \sqrt{p+1}}  
\leqa \sigma_k^2 \sum_{p=1}^{4\log(n(j^2\vee 1))} \frac{\1_{\sqrt{p}< \sqrt{n}|X_k| \le \sqrt{p+1}}}{\phi(\sqrt{p+1})},  \]
where we use the universal bound $N\leqa \sigma_k^3$ obtained above. Then,
\begin{align*} 
E_{f_0} & \sum_{k>K_n\,,\, k\in \cM_n^c} \int f_k^2 d\Pi(f\given X)\1_{\cA_n}
\leqa E_{f_0}\sum_{K_n<k\le n\,,\,k\in\cM_n^c} \sigma_k^2
\left\{1+
\sum_{p=1}^{4\log{n}} \frac{\1_{\sqrt{p}< \sqrt{n}|X_k| \le \sqrt{p+1}}}{\phi(\sqrt{p+1})} 
\right\} \\
& + E_{f_0} \sum_{j\ge 1} \sum_{jn < k\le(j+1)n\,,\, k\in\cM_n^c} 
\sigma_k^2\left\{1+ \sum_{p=1}^{4\log(nj^2)} \frac{\1_{\sqrt{p}< \sqrt{n}|X_k| \le \sqrt{p+1}}}{\phi(\sqrt{p+1})}\right\}.
\end{align*}
 We have, for $k\in \cM_n^c$, denoting $\bar\Phi(u)=\int_u^{+\infty} \phi(t)dt$,
\begin{align*}
 E_{f_0} \1_{\sqrt{p}< \sqrt{n}|X_k| \le \sqrt{p+1}}
& =P\left(\left|\cN(0,1)+f_{0,k}\rn \right|\in[\sqrt{p},\sqrt{p+1}]\right) \\
& \le 2\left\{\bar{\Phi}\left(\sqrt{p}-\delta_n\right)-\bar{\Phi}\left(\sqrt{p+1}+\delta_n\right)\right\} \le 2\phi(\sqrt{p}-\delta_n)/(\sqrt{p}-\delta_n),
\end{align*}
by removing the negative term and using the standard bound $\bar{\Phi}(x)\le \phi(x)/x$ for $x>0$. Now
\[ \frac{1}{\sqrt{p}-\delta_n}\frac{\phi(\sqrt{p}-\delta_n)}{\phi(\sqrt{p+1})}
= \frac{e^{1/2}}{\sqrt{p}-\delta_n}e^{-\delta_n^2/2+\sqrt{p}\delta_n}.
 \]
First dealing with the term $k\le n$, one deduces, recalling $\delta_n=(\log{n})^{-1/2}$,
\begin{align*}
\sum_{K_n<k\le n\,,\,k\in\cM_n^c} & \sigma_k^2
\sum_{p=1}^{4\log{n}} \frac{P_{f_0}(\sqrt{p}< \sqrt{n}|X_k| \le \sqrt{p+1})}{\phi(\sqrt{p+1})}   \leqa 
\sum_{K_n<k\le n\,,\,k\in\cM_n^c} \sigma_k^2
\sum_{p=1}^{4\log{n}} \frac{1}{\sqrt{p}}e^{\sqrt{p}\delta_n}\\
& \leqa \sum_{K_n<k\le n\,,\,k\in\cM_n^c} \sigma_k^2 \sum_{p=1}^{4\log{n}}\frac{1}{\sqrt{p}}\leqa \sum _{K_n<k\le n\,,\,k\in\cM_n^c} \sigma_k^2 \sqrt{\log{n}},
 \end{align*}
where one uses the previous bounds. Similarly,
\begin{align*}
  \sum_{j\ge 1} & \sum_{jn < k\le(j+1)n\,,\, k\in\cM_n^c}  \sigma_k^2\sum_{p=1}^{4\log(nj^2)} \frac{P_{f_0}(\sqrt{p}< \sqrt{n}|X_k| \le \sqrt{p+1})}{\phi(\sqrt{p+1})}\\
  & \le \sum_{j\ge 1} \sum_{jn < k\le (j+1)n} \sigma_k^2 \sum_{p=1}^{4\log(nj^2)}
\frac{1}{\sqrt{p}} e^{\sqrt{p}\delta_n}
 \leqa  \sum_{j\ge 1} \sum_{jn < k\le (j+1)n} \sigma_k^2 \sqrt{\log(nj^2)}e^{2\sqrt{\log(nj^2)}\delta_n}.
\end{align*}
Using that $nj^2\le (nj)^2\le k^2$ for $k > (nj)$ one gets that the last display is bounded up to a constant multiplicative factor by
\[ \sum_{j\ge 1} \sum_{jn < k \le (j+1)n}\sigma_k^2 \sqrt{\log{k}}e^{4\sqrt{\log{k}}\delta_n}=\sum_{k > n}\sigma_k^2 \sqrt{\log{k}}e^{4\sqrt{\log{k}}\delta_n}. \] 
Since $\sqrt{\log{k}}e^{4\sqrt{\log{k}}\delta_n}\le e^{\eta \log{k}}$ for any $k\ge n$ for $\eta>0$ fixed as small as desired for large enough $n$, the last display is bounded by $\sum_{k\ge n} \sigma_k^2 k^{\eta}\leqa n^{-2\alpha+\eta}=o(n^{-2\alpha/(2\alpha+1)})$ for small enough $\eta$. Putting the previous bounds together one gets
\[ E_{f_0}  \sum_{k>K_n\,,\, k\in \cM_n^c} \int f_k^2 d\Pi(f\given X^{(n)})\1_{\cA_n}
\leqa \sum_{k=K_n}^n \sigma_k^2(1+\sqrt{\log{n}}) + o(n^{-\frac{2\alpha}{2\alpha+1}}) 
\leqa \frac{\sqrt{\log{n}}}{n^{2\alpha/(2\alpha+1)}},\]
using  \eqref{defsig}. This bound is $O((\sqrt{\log{n}})n^{-2\be/(2\be+1)})$ if $\alpha\ge \be$, which, combined with Lemma \ref{lemev}, concludes the proof for $\sigma_k$ as in \eqref{defsig}.
\end{proof}

\begin{proof}[Proof of Theorem \ref{thmpriorm}]
Let $K\ge 2$ be an integer, and for a function $f$ in $L^2$, let as before $f^{[K]}$ denote its projection onto the linear span of $\vphi_1,\ldots,\vphi_K$ and $f^{[K^c]}=f-f^{[K]}$. Then 
\begin{align*}
\Pi&[\|f-f_0\|_2 < \veps] \ge 
\Pi\left[ \|f^{[K]}-f_0^{[K]}\|_2 < \veps/2 \,,\, \|f^{[K^c]}-f_0^{[K^c]}\|_2 < \veps/2\right] \\
& \ge \Pi\left[\forall\, k\le K,\ \  |f_k - f_{0,k} | \le \frac{\veps}{2\sqrt{K}}\ ;\ \forall\, k>K,\ \  |f_k| \le \frac{\veps}{D\sqrt{k}\log{k}} \right] \1_{\|f_0^{[K^c]}\|_2<\veps/4}\\
& = \prod_{k=1}^{K} \Pi\left[ |f_k - f_{0,k} | \le \frac{\veps}{2\sqrt{K}} \right] \Pi\left[\forall\, k>K,\ \  |f_k| \le \frac{\veps}{D\sqrt{k}\log{k}} \right] \1_{\|f_0^{[K^c]}\|_2<\veps/4},
\end{align*}
where $D$ is a large enough constant, and where we have used independence of the coefficients under the prior and the fact that $k^{-1/2}/\log(k)$ is a square-summable sequence. 

Suppose the indicator in the last display equals one, which imposes $\|f_0^{[K^c]}\|_2<\veps/4$, for which a sufficient condition is,  if $f_0$ is in $\Sbl$,
\begin{equation}
K^{-2\be} L^2 <\veps^2/16.
\end{equation}

Let us now bound each individual term $p_k:=\Pi[ |f_k-f_{0,k}|\le \veps/(2\sqrt{K})]$. By symmetry, for any $k\le K$, one can assume $f_{0,k}\ge 0$ and
\begin{align*}
 p_k  \ge \int_{f_{0,k}}^{f_{0,k}+\veps/(2\sqrt{K})} \sigma_k^{-1}h(x/\sigma_k) dx  \ge \frac{\veps}{2\sqrt{K}} h(C/\sigma_K), 
       \end{align*} 
where we have used that $(\sigma_k)$ is decreasing as well as $x\to h(x)$ on $[0,\infty)$ by assumption, and that $f_{0,k}+\veps/(2\sqrt{K})\le C$ since $|f_{0,k}|$ are bounded by $L$ for $f_0\in \Sbl$.  
For either choice of $\sigma_k$ it holds $\log(2\sqrt{K})\le \log(1+C/\sigma_K)$ for large $K$, hence combining with \eqref{condti} we get 
\begin{align*}
 p_k  \ge \veps e^{-C_1\log^{1+\kappa}(1+C/\sigma_K)}.
\end{align*}

One deduces, for a new value of the constant $C_1$,
\[ \cP_1:= \prod_{k=1}^{K} p_k \ge \veps^K \exp\left\{ -C_1K\log^{1+\kappa}(C/\sigma_K) \right\}. \]
Let us deal first with the case of $\sigma_k$ given by \eqref{defsig}.  We now also need to bound
\begin{align*}
\cP_2:=\Pi\left[\forall\, k>K,\ \  |f_k| \le \frac{\veps}{D\sqrt{k}\log{k}} \right]
&=\prod_{k>K} (1-2 \overline{H}(\veps/\{D\sigma_k\sqrt{k}\log{k}\}))\\
& = \prod_{k>K} (1-2 \overline{H}(\veps k^{\alpha}/\{D\log{k}\})).
\end{align*}  
Note that if  
\begin{equation}  \label{defeps}
\veps \ge D'K^{-\be}\log{K},
\end{equation} for some universal constant $D'>0$ to be chosen. 
Then, for $k>K$,
\[  \veps k^\alpha/\{D\log{k}\}  \ge k^\alpha K^{-\be}\frac{\log{K}}{\log{k}}\frac{D'}{D}. \]
Hence, as long as $\alpha\ge \be$,  $D'$ can be chosen sufficiently large so that the last term is at least $1$. Then, by using \eqref{condts},
\[ \overline{H}(\veps k^\alpha/\{D\log{k}\})
\le C_3 K^{2\be}k^{-2\alpha}(\log{k}/\log{K})^2.
 \]
Then, if $\alpha> 1/2$ so that the series $\sum k^{-2\alpha}$ is converging, and possibly enlarging $D'$ in \eqref{defeps} further in order to have that the right hand side of the last display is less than 1/4, using the inequality $\log(1-2x)\geq -4x$ for all $x\in[0,1/4)$, we have
\begin{align*}
 \cP_2 & \ge \exp\{ \sum_{k>K}\log(1-2C_3 K^{2\be}k^{-2\alpha}(\log{k}/\log{K})^2) \} \\
& \ge \exp\{ -4C_3 K^{2\be} \sum_{k>K} k^{-2\alpha} (\log{k}/\log{K})^2) \ge \exp(-C_4 K \cdot K^{2(\be-\alpha)}),
\end{align*}
where we have used, whenever $\alpha>1/2$, that $\sum_{k>K} k^{-2\alpha} (\log{k})^2=O( K^{-2\alpha+1}\log^2{K})$ as $K\to\infty$. The bound of the last display is at least $\exp(-C_4K)$ assuming $\alpha\ge \be$. 

Putting everything together one gets
\begin{align*}
\Pi[ \|f-f_0\|_2< \veps ] &  \ge 
\veps^K \exp\left\{ -C_1K\log^{1+\kappa}(C/\sigma_K) -C_4K\right\} \\
& \ge \exp \left\{ -K\log(K^\be/(D'\log{K}))-C_1K\log^{1+\kappa}(C/\sigma_K) -C_4K\right\} \\
& \ge \exp \left\{ -C_5 K \log^{1+\kappa}{K}\right\}.
\end{align*} 
{Recall the definition of $\veps_n$ as in \eqref{rateveps-1} and the constants $d_1,d_2$ from the statement of the Theorem. 
Set $K=K_n=(d_1/D')^{-1/\be} n^{1/(1+2\be)} \log^q{n}$, with $q=(1-\kappa)/(1+{2}\be)$ and $d_1$ to be chosen below.   
 Then  \eqref{defeps} holds for $\veps=d_1\veps_n$ and large enough $n$. Also,
\[C_5K\log^{1+\ka}K\le C d_1^{-1/\be} n\veps_n^2, \] where $C$ is a constant independent of $d_1, d_2$. For given $d_2>0$, the right hand side of the last display is bounded from above by $d_2n\veps_n^2$, provided $d_1$ is chosen sufficiently large, which concludes the proof for $\sigma_k$ as in \eqref{defsig}.}

Let us now turn to the case of $\sigma_k$ given by \eqref{fastsig}. The term $\cP_1$ above is bounded below by
\[ 
\prod_{k=1}^{K} p_k \ge \veps^K \exp\left\{ -C_1K\log^{1+\kappa}(C/\sigma_K) \right\} 
\ge \veps^K\exp\{ -C_1K\log^{(1+\kappa)(1+\delta)}{K} \}.
\]
On the other hand, we also have 
\begin{align*}
\cP_2 &=\prod_{k>K} (1-2 \overline{H}(\veps/\{D\sigma_k\sqrt{k}\log{k}\}))\\
& = \prod_{k>K} (1-2 \overline{H}(\veps e^{a\log^{1+\delta}{k}}/\{D\sqrt{k}\log{k}\})).
\end{align*} 
Let $\veps\ge DK^{-\be}$, then for large enough $K$, 
$\veps e^{a\log^{1+\delta}{k}}/\{D\sqrt{k}\log{k}\}\ge 1$ and by \eqref{condts}
\[\overline{H}(\veps e^{a\log^{1+\delta}{k}}/\{D\sqrt{k}\log{k}\})
\le c_2(\veps e^{a\log^{1+\delta}{k}}/\{D\sqrt{k}\log{k}\})^{-2}
\le e^{-a\log^{1+\delta}{k}}
 \] 
so that $\cP_2\ge \exp\{ -C\sum_{k>K} e^{-a \log^{1+\delta}{k}} \}\ge \exp\{-C' e^{-a\log^{1+\delta}{K}}\}$. The latter is bounded from below by a constant, so the final bound obtained for the probability at stake is $\exp\{-C'K\log^{(1+\kappa)(1+\delta)}{K}\}$ (for a new value of the constant $C'$). 
{Let us set $\veps=d_1\veps_n$ and $K:=(d_1\veps_n/D)^{-1/\be}$, for $\veps_n$ as in \eqref{rateveps-2}. Then 
\[ C'K\log^{(1+\kappa)(1+\delta)}{K}\le C d_1^{-1/\be} n\veps_n^2, \] where $C$ is  independent of $d_1, d_2$. The last display is less than $d_2n\veps_n^2$ for large $d_1$, which concludes the proof. 
} 
\end{proof}  

\begin{remark}[Cauchy tails] \label{rem:cauchy}
We note that the case of $H$ equal to the Cauchy distribution, corresponding to $\overline{H}(x)\leq c_2/x$ for $x\ge1$, can be accommodated up to a slight variation on the condition for the prior HT$(\al)$. Suppose in this case that $\alpha>1$ (recall that we have the choice of $\al$, and that in view of the Theorem, the larger $\al$ is, the larger the range for which adaptation occurs, so we can always choose $\al>1$ beforehand). Indeed, in this case
for $\alpha> 1$ one gets, for $\sigma_k$ as in \eqref{defsig},
\begin{align*}
\cP_2 & \ge \exp\{ \sum_{k>K}\log(1-2C_3(K^{\be}/k^{\alpha})(\log{k}/\log{K})) \} \\
& \ge  \exp\{ -C_4 K^{1+\be-\alpha} \} = \exp\{ -C_4 K K^{\be-\alpha} \}, 
\end{align*}
and from there on the proof is identical to that of Theorem \ref{thmpriorm}. A similar remark applies to the case of $\sigma_k$ as in \eqref{defsigr}, this time with no extra condition (the latter choice is free of $\al$).
\end{remark}

{\em Technical Lemmas.} 
\begin{lemma} \label{laptr}
Consider the model $\te\sim\pi$ and $X\given \te\sim \cN(\te,1/n)$. Suppose $\pi$ is the law of $\sigma\cdot\zeta$, for $\sigma>0$ and $\zeta$ a real random variable with density $h$ satisfying \eqref{conds}--\eqref{condti}.
Then for some $C, C_1>0$, it holds, for all $t\in\RR, \theta_0\in\RR, \sigma>0$,
\[ \log E_{\te_0}E^\pi\big[e^{t\sqrt{n}(\te-X)}\,\big|\,X\big] 
\le  t^2/2  + C\log^{1+\kappa}\left(1+\frac{|\te_0|+1/\sqrt{n}}{\sigma}\right)+C_1.\] 
\end{lemma}
\begin{proof}
For any $t\in\RR$, we have  
\begin{align*}
E_{\theta_0}E^\pi\big[e^{t\sqrt{n}(\theta-X)}\,\big|\,X\big]&=
E_{\theta_0}\frac{\int\exp\big(t\sqrt{n}(\theta-X)\big)\varphi\big(\sqrt{n}(X-\theta)\big)h(\theta/\sigma)d\theta}{\int\varphi\big(\sqrt{n}(X-\theta)\big)h(\theta/\sigma)d\theta}\\
&=E_{\xi\sim \cN(0,1)}\frac{\int e^{t(v-\xi)-\frac{(v-\xi)^2}2}h\big(\frac{\theta_0+v/\sqrt{n}}{\sigma}\big)dv}{\int e^{-\frac{(v-\xi)^2}2}h\big(\frac{\theta_0+v/\sqrt{n}}{\sigma}\big)dv}.
\end{align*}
Using that $h$ is bounded, one gets
\begin{align*}\label{denomintegral}
E_{\theta_0}E^\pi\big[e^{t\sqrt{n}(\theta-X)}\,\big|\,X\big]&\lesssim E_{\xi\sim \cN(0,1)}\frac{\int e^{tu-u^2/2}du}{\int e^{-\frac{(v-\xi)^2}2}h\big(\frac{\theta_0+v/\sqrt{n}}{\sigma}\big)dv}\\
&\lesssim e^{t^2/2} E_{\xi\sim \cN(0,1)}\Bigg[\Big(\int e^{-\frac{(v-\xi)^2}2}h\big(\frac{\theta_0+v/\sqrt{n}}{\sigma}\big)dv\Big)^{-1}\Bigg].
\end{align*}
The latter integral can be lower bounded using \eqref{conds} and \eqref{condti},
\[\int e^{-\frac{(v-\xi)^2}2}h\big(\frac{\theta_0+v/\sqrt{n}}{\sigma}\big)dv\gtrsim \int_{-1}^1 e^{-\frac{(v-\xi)^2}{2}}e^{-c_1\log^{1+\kappa}\left(1+\frac{|\te_0|+1/\sqrt{n}}{\sigma}\right)}dv.\]
As a result, we have 
\[E_{\theta_0}E^\pi\big[e^{t\sqrt{n}(\theta-X)}\,\big|\,X\big]\lesssim e^{t^2/2+c_1\log^{1+\kappa}\big(1+\frac{|\te_0|+1/\sqrt{n}}{\sigma}\big)} E_{\xi\sim \cN(0,1)}\Bigg[\Big(\int_{-1}^1e^{-\frac{(v-\xi)^2}{2}}dv\Big)^{-1}\Bigg],\]
and the claim follows since the expectation appearing on the right hand side can be bounded by a universal constant as in \cite{cn13}, pages 2015-2016.
\end{proof}
 
\begin{lemma} \label{lad} 
 Let $Y$ be a real random variable. Then for $t>0$ and $\cL(t)=E[\exp(t|Y|)]$,
 \[ E[Y^2] \le t^{-2}\left\{8+2\log^2\cL(t)\right\}.\]
\end{lemma}
\begin{proof}
Let us write $E[Y^2]=t^{-2}E[\log^2\exp(t|Y|)]$. Using concavity of the map $x\to \log^2(x)$ for $x>e$, one may write
\begin{align*}
 E[\log^2\exp(t|Y|)] & \le E[\log^2(e+\exp(t|Y|))]\le \log^2\{e+ \cL(t)\} \\
& \le (2+\log\cL(t))^2\le 8+2\log^2\cL(t),
\end{align*}
where the last inequality uses $\log(e+b)\le 2+\log(b)$ valid for $b\ge 1$.  
\end{proof}

\begin{lemma} \label{lemev}
Let $\cA_n(N)$ be the event as in \eqref{defev}, where either $N=\cN_n$ or $N=\cM_n$ as in \eqref{enn} and \eqref{emn}. Then for large enough $n$ it holds
\[ P_{f_0}[\cA_n^c] \leqa 1/\sqrt{n}.  \] 
\end{lemma}
\begin{proof} 
First, for any  $k\in {N^c}$  and $j\ge 0$, we have 
\[ P_{f_0}\left[\cA_{k,j}^c\right]\le P\left[ |\cN(0,1)|>\sqrt{3\log{n(j^2\vee 1)}}\right] \le 2(n(j^2\vee 1))^{-3/2},\] 
where one uses that $|f_{0,k}|{\le 1/\sqrt{n}}$ {for $k\in N^c$}. 
From this one deduces
\[P_{f_0}[\cA_n^c] \leqa  
\sum_{j\ge 0} n \frac{1}{\{n (j^2 \vee 1)\}^{3/2}}\leqa 1/\sqrt{n}. \qedhere \]
\end{proof}

\begin{acks}[Acknowledgments]
 The authors would like to thank the Associate Editor and two referees for a number of insightful comments.
 \end{acks}

\begin{funding}
IC's work was partly funded by ANR BACKUP grant ANR-23-CE40-0018-01 and a fellowship from the Institut Universitaire de France.
\end{funding}
\bigskip

\begin{center}\large{\bf SUPPLEMENTARY MATERIAL}\end{center}\smallskip

 We start this supplement with a discussion on $\rho$--posteriors and the possibility to extend some of the results in Section \ref{sec:prmass} to the standard posterior. We then provide the rest of the proofs of our results, some technical lemmas as well as additional simulations corroborating the theory.
\setcounter{section}{0}

\numberwithin{equation}{section}
\numberwithin{theorem}{section}
\numberwithin{lemma}{section}
\numberwithin{remark}{section}
\renewcommand\thesection{\Alph{section}}
\renewcommand\thefigure{\thesection.\arabic{figure}}   
\renewcommand\thetable{\thesection.\arabic{table}}   
\renewcommand\theremark{\thesection.\arabic{remark}}

\section{Discussion:  $\rho$--posteriors ($\rho<1$) and the posterior ($\rho=1)$}\label{sup:seca}

 We now briefly discuss contributions on posteriors and fractional posteriors (that is, $\rho$--posteriors with $\rho<1$), mostly focusing on the nonparametric setting. General results on convergence rates for posterior distributions were obtained in a series of fundamental works in the early 2000's \cite{ggv00, sw01, gvni} among others. These results have shown a very broad applicability, and both convergence rates and, more recently, limiting results, uncertainty quantification, and testing using posterior distributions have been investigated for an increasing number of settings; we refer to the book \cite{gvbook} for an overview of the field.  The approach that has been developed for convergence rates generally assumes three conditions:  i) a prior mass condition around the truth, together with ii) the existence of sieve sets on which iii) certain tests exist (the latter condition is often replaced by entropy conditions that are typically stronger but may be easier to verify on examples). 
On the other hand, $\rho$--posteriors when $\rho<1$ only require a prior mass condition for the $\rho$--posterior to converge in terms of the $\rho$--R\'enyi divergence \cite{tongz06}. In terms of convergence rates, when both approaches can be used, the obtained converge rate is generally the same up to constants (that depend on $\rho$) for the posterior and fractional posterior. We refer to the work \cite{ltcr23} for a discussion and results on the dependence in terms of $\rho$, in particular for uncertainty quantification. 

Since fractional posteriors require a prior mass condition only to verify the convergence rate, they can sometimes been deployed in settings where the verification of the test and sieve conditions seems difficult, see e.g. \cite{BPY, lineroyang}. Sometimes, it is possible by strengthening the conditions on the prior to obtain results for both $\rho<1$ and $\rho=1$, as done for instance in \cite{lineroyang}, where the authors still apply the general approach \cite{ggv00} for $\rho=1$. We are not aware of existing results for settings where using the approach of \cite{ggv00} seems out of reach. This is precisely the case for the heavy tailed priors considered herein, for which it is unclear how to construct appropriate sieve sets, due to the slow decrease of the prior mass of complements of growing balls, implied by the prior's heavy tails.  Nevertheless, we are able (for example) in Theorem \ref{thm-seq} to give a proof of convergence for the standard posterior $\rho=1$, in white noise regression. The corresponding result for $\rho<1$ is obtained from Theorem \ref{thm-besov-rho}; it could also be obtained by following the arguments of the proof of Theorem \ref{thm-seq} with the likelihood raised to the power $\rho$. 

An interesting topic for future research would be to prove, if true, a convergence rate for the standard posterior ($\rho=1$) corresponding to the heavy-tailed priors of the paper, in settings beyond white noise regression, such as density estimation and classification. In simulations in those settings, we have seen barely any difference between the behaviour of the standard posterior ($\rho=1$) and that of a $\rho$--posterior with $\rho$ close to $1$ (e.g. equal to $.9$ say). This strongly suggests that there is a `smooth' transition between $\rho<1$ and $\rho=1$ and that theoretical results as we have obtained for $\rho$--posteriors (e.g. in density estimation or classification) should continue to hold for the standard posterior. We thus conjecture that the results obtained in Theorems \ref{thmdensity} and \ref{thmclassif} continue to hold for the standard posterior. We now discuss possible ways to prove this.

A main technical difficulty with using heavy-tailed priors as considered here is the lack of natural sieve sets to deal with the `high-frequency'--part of the posterior distribution (by this we mean the control of the basis coefficients corresponding to $k\ge K_n\asymp n^{1/(2\beta+1)}$ in single--index notation). Note that priors that explicitly model a  `truncation' point, for example sieve priors with random truncation at a frequency $K$, do not suffer from this difficulty, even if `low-frequencies' $k\le K_n$  are modelled via heavy-tailed priors: this is because the prior distribution on $K$ is able to downweight too complex models, while for low-frequencies one can use local-entropy arguments (see e.g. the contributions \cite{arbel13}--\cite{shenghosal15}). But note that one main advantage of our approach is precisely that we do not model explicitly a latent cut-off variable $K$, which while generally leading to $L^2$--adaptation, typically requires a costly Gibbs sampling to simulate from the posterior distribution on $K$; our method avoids this which is very appealing simulations-wise as demonstrated in the simulations in Section \ref{sec:sim} of the main article and  Section \ref{sec:adsim} below. 
 
Based on the above comments one can outline two promising directions for the case $\rho=1$. 
To fix ideas, let us focus on density estimation. For this model one puts a heavy-tailed prior on the log-density as in Section 3.2 of the paper.
 
A first direction consists in seeking for an extension of the approach of \cite{ggv00}. As discussed above, one main difficulty in applying this approach as such lies in the fact that, often, it requires building a collection of sieve sets on which suitable tests exist; but the sieve sets need to have exponentially vanishing prior probability, which for heavy-tailed priors means that they should be very large, hence making the construction of the aforementioned tests more complicated. Most likely, following this approach in the present setting will require finding an alternative to the usual entropy-sieve argument. A possibility
could be to build test functions directly; an example of application of this strategy can be found e.g. in the paper \cite{gn11} (see also e.g. Theorem 7.3.5 of \cite{ginenicklbook}), where the authors construct tests based on a well-chosen estimator for which sufficiently sharp concentration inequalities provide the appropriate decrease of testing errors.  However the later approach still often requires to work with some form of truncated prior (as it still requires sieves), so `high-frequencies' would have to be dealt with separately. 

A second direction could be based on extending the arguments used for the white noise model, which suggest a possible route to prove such results for the standard posterior $(\rho=1)$ in more general models. It is in particular conceivable that the arguments for `low frequencies' ($k\le K_n$; or $l\le L_n$ in double-index notation, with $K_n, L_n$ the usual nonparametric cut-offs) can be extended to more general models by studying the induced posterior on the coefficients $f_k=\psg f,\varphi_k\psd$ -- or $f_{lk}=\psg f,\psi_{lk}\psd$ in double-index notation -- viewed as linear functionals of $f$, as in the line of work initiated in \cite{cn13, cn14, ic14, n17} --; in particular, the case of the density estimation model has been studied in detail in \cite{cn14, ic14}. The main obstacle we foresee lies in the high-frequency terms, for which up to now there is no known argument to prove that these are suitably small. More precisely, along the proof of Theorem 1 for large $k$'s (case $k>K_n$ and $|f_{0,k}|\le 1/\sqrt{n}$), we are able to relate the posterior second moment $\int f_k^2 d\Pi(f\given X)$ to the prior's variance $\sigma_k^2$. This is done through quite precise calculations made possible by the structure of the white noise model, but is likely to be even more delicate for more complex models. For the OT prior, the fact that the prior's variance decays rapidly with $k$ could help there, although it is likely that the argument would become more complicated for low frequencies. For the HT$(\al)$ prior, one could truncate the posterior at some desired maximal cut-off $K_{max}$. It should be possible (as discussed above) to handle the low frequency terms $k\le K_n$ using a BvM-type argument (\cite{cn14, ic14}); however, a new argument would be required for frequencies $K_n<k\le K_{max}$, for which the prior should take over the likelihood. 

Going beyond this technical bottleneck is an interesting topic for future work. More generally, deriving tools that could avoid the need for verifying tests and/or entropy conditions when proving convergence rates would be particularly interesting for the theoretical validation of the use of Bayesian posterior distributions. For instance, in the setting of deep Gaussian processes, theory for usual posteriors can for now be provided for priors that downweight sufficiently large models \cite{fs23} so that the generic approach of \cite{ggv00}  can be used, which goes through a careful conditioning of Gaussian process paths and building a `strong' model-selection type prior. On the other hand, theoretical guarantees can be obtained for $\rho$-posteriors, $\rho<1$,  for unconstrained deep Gaussian processes close to the ones  implemented in practice \cite{cr24}. Yet, although one can conjecture that, again, results should still carry over to the case $\rho=1$, there is no theory available so far to confirm this.
\\

\section{Proofs}

\subsection{Proofs of remaining results of Section \ref*{sec:reg}}\label{proof:sec2}

\begin{proof}[Proof of Theorem \ref{thminv}]
Let us introduce the cut-off $\cK_n$ given by
\begin{equation} \label{kninv}
\cK_n = n^{1/(2\be+2\nu+1)}.
\end{equation}
We wish to achieve the quadratic rate, up to a logarithmic factor,
\begin{equation} \label{rateinv}
\veps_n^2 = n^{-2\be/(2\be+2\nu+1)}.
\end{equation}
The prior distribution $\Pi$ on the collection of coefficients $(f_k)$ induces a prior distribution, that we denote $\Pi^\#$, on the sequence $(\mu_k)$ where $\mu_k:=\kappa_k f_k$. By definition of $\Pi$, the distribution $\Pi^\#$ draws coefficients $\mu_k$ independently with respective laws $\tau_k \zeta_k$, with $\tau_k:=\kappa_k\sigma_k$. 

Since the sequence defines a signal-in-white noise model $X_k\sim\cN(\mu_k,1/n)$, the posterior distribution $\Pi^\#[\cdot\given X]$ induced on $(\mu_k)$ has then the same expression as the posterior on coefficients in the proof of Theorem \ref{thm-seq}, with a `true' $\mu_0$ that has coefficients $\mu_{0,k}:= \kappa_k f_{0,k}$. 

By the corresponding argument in the proof of Theorem \ref{thm-seq}, Eq.  \eqref{eq:logbnd}  becomes, for $k\ge 1$, 
\[ nE_{\mu_0}\int (\mu_k-\mu_{0,k})^2 d\Pi^\#(\mu\given X^{(n)})\leqa 
\log^{1+\kappa}
\left(1+\frac{L+1/\sqrt{n}}{\tau_k} \right).\]
From this one deduces that, in model \eqref{modip}, it holds for any $k\ge 1$,
\[ nE_{f_0}\int (f_k-f_{0,k})^2 d\Pi(f\given X^{(n)})\leqa \kappa_k^{-2} 
\log^{1+\kappa}
\left(1+\frac{L+1/\sqrt{n}}{\tau_k} \right). \]
Using this bound for indices $k\le \cK_n$ or $k\in\cN_n:=\{k>\cK_n:\ \kappa_k |f_{0,k}|>1/\sqrt{n}\}$
 and noting  
 \begin{align*}
 \sum_{k\le \cK_n} \ka_k^{-2} & \leqa \cK_n^{2\nu+1}\leqa n\veps_n^2,\\
 \sum_{k\in \cN_n} \ka_k^{-2} & \le n \sum_{k\in\cN_n} f_{0,k}^2 \le n \cK_{n}^{-2\be} \sum_{k\ge 1} k^{2\be} f_{0,k}^2 \leqa  n\veps_n^2,
 \end{align*}
where we have used the definition of $\cK_n, \cN_n$ and the Sobolev regularity of $f_0$, one deduces 
\[ \sum_{k\le \cK_n \text{ or }k\in\cN_n} E_{f_0}\int (f_k-f_{0,k})^2 d\Pi(f\given X^{(n)})\leqa (\log n)^d \veps_n^2,\]
where we have used that the logarithmic term above is bounded by $(\log{n})^d$ for some $d>0$ for $k\in \cN_n$ (since for $k\in\cN_n$ one must have $k^{2\be+2\nu}\leqa n$). 

Now focusing on the case $k\notin\cN_n$ and $k>\cK_n$, working again first with $\mu_k$, we follow  the same reasoning as in the proof of Theorem \ref{thm-seq} (with the same definition of $x_k$ and $\cA_n$, with $\cK_n$ in place of $K_n$)  to obtain, for $\sigma_k$ as in \eqref{defsigr} and $\ta_k=\kappa_k\sigma_k$,
\[ \int  \mu_k^2 d\Pi^\#(\mu \given X^{(n)}) \1_{|X_k|\le x_k}
\le \ta_k^2 \frac{\1_{|X_k|\le x_k}}{\phi(\sqrt{n}x_k)}, \]
noting that the constraint is now $\ta_k\leqa x_k$ and is satisfied: indeed, it follows from the definition of $\sigma_k$ that $\tau_k\leqa 1/\sqrt{n}$ for any $k>\cK_n$. Dividing both sides of the last display by $\ka_k^2$, one gets
\[ \int  f_k^2 d\Pi(\mu \given X^{(n)}) \1_{|X_k|\le x_k}
\le \si_k^2 \frac{\1_{|X_k|\le x_k}}{\phi(\sqrt{n}x_k)}. \]
This is the same bound as in the proof of Theorem \ref{thm-seq}, so one concludes similarly that
\[
 \sum_{k\notin\cN_n,\ k>\cK_n} \int f_k^2d\Pi(f\given X^{(n)})\1_{\cA_n}=o(n^{-r}),
\]
where $r$ can be chosen to be an arbitrarily large integer. This concludes the proof for $\sigma_k$ as in \eqref{defsigr}. The proof for the case of $\sigma_k$ as in \eqref{defsig}, which follows  in a similar way, is omitted.
\end{proof}

\begin{proof}[Proof of Theorem \ref{thm-seq-sup}]
Below we repeatedly use the following (e.g. \cite{ginenicklbook}, Chapter 4): for $\{\psi_{lk}\}$ a (boundary-corrected) wavelet basis of $[0,1]$ as chosen in the main paper, and for $g$ a continuous function on $[0,1]$ with wavelet coefficients $g_{lk}$, we have
\[ g(x) = \sum_{l\ge 0} \sum_{k} g_{lk}\psi_{lk}(x) \]
in $L^\infty$ and the following upper-bounds on the supremum norm in terms of coefficients
\begin{equation} \label{bsn}
 \|g\|_\infty \le  \sum_{l\ge 0} \max_{k} |g_{lk}| \sup_{x}\sum_{k} |\psi_{lk}|(x) 
\leqa  \sum_{l\ge 0} 2^{l/2} \max_{k} |g_{lk}|,
\end{equation}
using $\|\sum_k |\psi_{lk}|\|_\infty\leqa 2^{l/2}$ (implied by $\sum_{k\in\mathbb Z}\psi_{00}(\cdot-k)\in L^\infty(\mathbb R)$ which holds by construction of boundary-corrected wavelet bases, see e.g. \cite[Section 4.3.5 and Theorem 4.2.10]{ginenicklbook}). 
Recalling that we assume $f_0\in\cH^\beta(L)$, denote by $\ell_n, \lambda_n$ the (closest integer) solutions to
\[
2^{\ell_n}=(n/\log{n})^{1/(1+2\be)} \quad\text{and}\quad 2^{\lambda_n}=(n \log{n})^{1/(1+2\be)},
\]
respectively. Furthermore, denote by $\veps_n=(\log{n}/n)^{\be/(2\be+1)}\asymp\sqrt{\ell_n2^{\ell_n}/n}$. Note that $\ell_n$ is the standard truncation point when approximating functions in $\cH^\beta$, while $\lambda_n$ is slightly larger. We will use the former for $s_l$ as in (\ref{defsigr}) and the latter for $s_l$ as in (\ref{defsig}).
For 
 $\cA_n$ a suitable event to be defined below, using Markov's inequality,
\[ E_{f_0}\Pi[\|f-f_0\|_\infty>v_n \given X^{(n)}] \le P_{f_0}(\cA_n^c) +
 v_n^{-1}E_{f_0}\left[ \int \|f-f_0\|_\infty d\Pi(f\given X^{(n)})\1_{\cA_n}\right],   \]
where $v_n=\cL_n \veps_n$ for $\cL_n=(\log{n})^d$ with $d>0$ sufficiently large.
Given an integer $l_n$, for a function $f\in L^2$ denote by $f^{[l_n]}$ and $f^{[l_n^c]}$ its projection onto the linear span of the wavelets $\{\psi_{lk}\}_{l\leq l_n, k\in\cK_l}$ and $\{\psi_{lk}\}_{l> l_n, k\in\cK_l}$, respectively. We can then decompose the norm as follows
\[
\|f-f_0\|_\infty\leq \|f^{[l_n]}-f_0^{[l_n]}\|_\infty+\|f^{[l_n^c]}\|_\infty+\|f_0^{[l_n^c]}\|_\infty.
\]
We denote the three terms on the right hand side by (I), (II), (III), respectively and below we study them one by one.\\

For (III), using \eqref{bsn} and that $f_0$ belongs to $\Hbl$, for either choice $l_n=\ell_n$ or $l_n=\lambda_n$, we have
\[ \|f_0^{[l_n^c]}\|_\infty\leqa 
\sum_{l>l_n} 2^{l/2} \max_{k} |f_{0,lk}|
\leqa  \sum_{l>l_n} 2^{l/2} L 2^{-l(1/2+\be)}\leqa  2^{-l_n\be}=o(v_n).
\]

For (I), one can write
\begin{align*}
E_{f_0}E[{\rm (I)}|X^{(n)}]&\leq \sum_{l\leq l_n}2^{l/2}E^n_{f_0}E^\Pi\big[\max_{k\in\cK_l}|f_{lk}-f_{0,lk}|\,\big|\,X^{(n)}\big]\\
&\leq \sum_{l\leq l_n}2^{l/2}E^n_{f_0}E^\Pi\big[\max_{k\in\cK_l}|f_{lk}-X_{lk}|\,\big|\,X^{(n)}\big]+\sum_{l\leq l_n}2^{l/2}E^n_{f_0}\max_{k\in\cK_l}|X_{lk}-f_{0,lk}|.
\end{align*}
The second term on the right hand side satisfies 
\[
\sum_{l\leq l_n}2^{l/2}E^n_{f_0}\max_{k\in\cK_l}|X_{lk}-f_{0,lk}|=\sum_{l\leq l_n}\frac{2^{l/2}}{\sqrt{n}}E\max_{k\in\cK_l}|\xi_{lk}|\leq c \sum_{l\leq l_n} \frac{2^{l/2}\sqrt{l}}{\sqrt{n}},
\]
where $\xi_{lk}$ are independent standard normal random variables. Here, we have first used the observational model and then that the maximum of $C2^l$ standard normal variables has expectation of order $\sqrt{l}$. 

To bound the first term on the right hand side we use similar techniques to \cite[Proposition~1 and Lemma 1]{ic14}. First, we observe that for any $t>0$
\begin{align*}
tE^n_{f_0}E^\Pi\big[\max_{k\in\cK_l}\sqrt{n}|f_{lk}-X_{lk}|\,\big|\,X^{(n)}\big]
&\leq\log\sum_{k\in\cK_l}E_{f_0}E^\Pi\big[e^{t\sqrt{n}(f_{lk}-X_{lk})}+ e^{-t\sqrt{n}(f_{lk}-X_{lk})}\,\big|\,X^{(n)}\big], 
\end{align*}
where we have used Jensen's inequality and bounded the maximum by the sum. To bound the Laplace transforms on the right hand side, we employ Lemma \ref{laptr} and distinguish the two choices of $s_l$ in (\ref{defsig}) and (\ref{defsigr}), with the corresponding choices $l_n=\lambda_n$ and $l_n=\ell_n$, respectively. For $s_l$ as in (\ref{defsig}), for $l\le \lambda_n$ and $\alpha\ge\be$ it holds
\begin{align*}
\frac{|f_{0,lk}|+1/\sqrt{n}}{s_l}\lesssim 2^{l(\alpha-\be)}+\frac{1}{\sqrt{n}}2^{l(1/2+\alpha)}\leqa 2^{l(\alpha-\be)}\sqrt{\log n},
\end{align*}
hence Lemma \ref{laptr} shows that there exist $C_1, C_2, C_3>0$, such that for any $t>0$ and $l\leq \lambda_n$
\begin{align*}tE_{f_0}E^\Pi\big[\max_{k\in\cK_l}\sqrt{n}|f_{lk}-X_{lk}|\,\big|\,X^{(n)}\big]&\lesssim \log\sum_{k\in\cK_l}C_1e^{t^2/2}e^{C_2l^{1+\kappa}}e^{C_3(\log\log n)^{1+\kappa}}\\
&=\log\Big(C_12^le^{t^2/2}e^{C_2l^{1+\kappa}}e^{C_3(\log\log n)^{1+\kappa}}\Big).
\end{align*} 
On the other hand, for $s_l$ as in (\ref{defsigr}), it holds 
\begin{align*}
\frac{|f_{0,lk}|+1/\sqrt{n}}{s_l}\lesssim 2^{l(l-1/2-\be)}+\frac{1}{\sqrt{n}}2^{l^2}\lesssim 2^{l^2},
\end{align*}
hence using Lemma \ref{laptr} we obtain that there exist $C_1, C_2>0$, such that 
for any $t>0$ and $l\leq \ell_n$
\begin{align*}tE_{f_0}E^\Pi\big[\max_{k\in\cK_l}\sqrt{n}|f_{lk}-X_{lk}|\,\big|\,X^{(n)}\big]&\lesssim \log\sum_{k\in\cK_l}C_1e^{t^2/2}e^{C_2l^{2+2\kappa}}\\
&=\log\Big(C_12^le^{t^2/2}e^{C_2l^{2+2\kappa}}\Big).
\end{align*} 
Combining, and setting $r=1$ for $s_l$ as in (\ref{defsig}) and $r=2$ for $s_l$ as in (\ref{defsigr}), we get that for $l\le \lambda_n$ and $l\le \ell_n$, respectively,
\begin{align*}E_{f_0}E^\Pi\big[\max_{k\in\cK_l}\sqrt{n}|f_{lk}-X_{lk}|\,\big|\,X^{(n)}\big]&\lesssim\frac{\log(C_12^l)}{t}+\frac{t}2+\frac{l^{r(1+\kappa)}}t+\frac{\log\log{n}}t\\
&\lesssim t+\frac{l^{r(1+\kappa)}}t+\frac{\log\log{n}}t.
\end{align*}
We can optimize by choosing $t=\sqrt{l^{r(1+\kappa)}+\log\log{n}}$, whence we get 
\[E_{f_0}E^\Pi\big[\max_{k\in\cK_l}\sqrt{n}|f_{lk}-X_{lk}|\,\big|\,X^{(n)}\big]\lesssim l^{\frac{r(1+\kappa)}2}+\sqrt{\log\log{n}}.\]
Putting things together, and since both choices of $l_n$ satisfy $l_n\lesssim \log n$, we get that for either $s_l$ as in (\ref{defsigr}) with $l_n=\ell_n$ or (\ref{defsig}) with $l_n=\lambda_n$, it holds
\begin{align}\label{boundI}
E_{f_0}E[{\rm (I)}|X^{(n)}]&\lesssim \sum_{l\leq l_n}\frac{2^{l/2}}{\sqrt{n}}\Big(\sqrt{l}+l^{\frac{r(1+\kappa)}2}+\sqrt{\log\log{n}}\Big)\nonumber\\
&\lesssim\sum_{l\leq l_n}\frac{l^{\frac{r(1+\kappa)}2}+\sqrt{\log\log{n}}}{\sqrt{n}}2^{l_n/2}=o(v_n).
\end{align}

We next study term (II). We start by noting that for any $q\ge 1$, by Jensen's inequality, 
\begin{align*}
E_{f_0}\left[E^{\Pi}[{\rm (II)}|X^{(n)}]\1_{\cA_n}\right]&\leqa  \sum_{l> l_n}2^{l/2}E^n_{f_0}\left[E^\Pi\big[\max_{k\in\cK_l}|f_{lk}|\,\big|\,X^{(n)}\big]\1_{\cA_n}\right]\\
&\leq \sum_{l>l_n}2^{l/2}E^n_{f_0}\left[E^\Pi\big[\big(\sum_{k\in\cK_l}|f_{lk}|^q\big)^{1/q}\,\big|\,X^{(n)}\big]\1_{\cA_n}\right]\\
&\leq \sum_{l>l_n}2^{l/2}E^n_{f_0}\left[\left(E^\Pi\big[\sum_{k\in\cK_l}|f_{lk}|^q\,\big|\,X^{(n)}\big]\right)^{1/q}\1_{\cA_n}\right]\\
&\leq \sum_{l>l_n}2^{l/2}\left(\sum_{k\in\cK_l}E^n_{f_0}\left[E^\Pi\big[|f_{lk}|^q\,\big|\,X^{(n)}\big]\1_{\cA_n}\right]\right)^{1/q}.
\end{align*}
Using Bayes' formula, it is enough to bound the following terms individually
\[ \int |f_{lk}|^q d\Pi(f\given X) = \frac{ \int \te^q \phi(\sqrt{n}(X_{lk}-\te))h(\te/s_l)d\te }{ \int  \phi(\sqrt{n}(X_{lk}-\te))h(\te/s_l)d\te } =: \frac{N_q}{D}.  \] To bound the numerator $N_q$, we use $|\phi|\le \|\phi\|_\infty$ and $\int \te^q h(\te/s_l)d\te = s_l^{q+1}\int u^q h(u)du\leqa s_l^{q+1}$, using (\ref{qmom}), that is, that $h$ has $q$ moments, so that $N_q\leqa s_l^{q+1}$ regardless of $X_{lk}$. 

The denominator is bounded in the same way as in the proof of Theorem \ref{thm-seq}. In particular, for $(x_{lk})$ a deterministic nonnegative sequence such that $s_l\lesssim x_{lk}$, it holds that 
 \[  \frac{N_q}{D}\1_{|X_{lk}| \le x_{lk}} \le s_l^q \frac{\1_{|X_{lk}|\le x_{lk}}}{\phi(\sqrt{n}x_{lk})}. \]
 
 Let us first deal with the case of super-light variances (\ref{defsigr}) and $l_n=\ell_n$. 
Define the events, for $j\ge 0$, $l\ge 1$, $k\in\cK_l$ and $t=(L+\sqrt{3})^2$ (recalling that $f_0\in \cH(L)$),
 \begin{align}
 \cA_{l, k,j} & := \left\{ \, |X_{lk}| \le \sqrt{\frac{t(j+1)\log{(n)}}{n}}\, \right\},  \label{defevkj-inf}\\
 \cA_n & := \bigcap_{l_n<l\le \log_2(n)}\ \bigcap_{k\in\cK_l} \,\cA_{l,k,0}\
  \cap\ \bigcap_{j\ge 1}\ \bigcap_{j\log_2(n) < l \le (j+1)\log_2(n)} \ \bigcap_{k\in \cK_l} \,\cA_{l,k,j}.\label{defev-2}
 \end{align} 
Let us set $x_{lk}=\sqrt{\frac{t(j+1)\log{n}}{n}}$ whenever $j\log_2(n)<l\le (j+1)\log_2(n)$. The  constraint $s_l\leqa~x_{lk}$ is trivially satisfied when $l>\ell_n$ for this choice of $x_{lk}$ and large enough $n$ since then $s_l\leqa 1/\sqrt{n}$ for super-light variances. 
 Then, using that $(s_l)$ is decreasing and that $\phi(\sqrt{n}x_{lk})^{-1}\leqa n^{t(j+1)/2}$ for $j\ge 0$, we have the bounds
 \begin{align*}
&E_{f_0}\left[E^{\Pi}[{\rm (II)}|X^{(n)}]\1_{\cA_n}\right]\leq
\sum_{l>\ell_n}2^{l/2}\left(\sum_{k\in\cK_l}E^n_{f_0}\left[E^\Pi\big[|f_{lk}|^q\,\big|\,X^{(n)}\big]\1_{\cA_n}\right]\right)^{1/q}\\
&\leqa \sum_{\ell_n<l\le\log_2{(n)}}2^{l/2}\left(\sum_{k\in\cK_l} s_l^q n^{t/2}\right)^{1/q}
+\,\sum_{j\ge 1}\,\sum_{j\log_2{(n)}<l\le(j+1)\log_2{(n)}}2^{l/2}\left(\sum_{k\in\cK_l} s_l^q n^{t(j+1)/2}\right)^{1/q}\\
&\leqa n^{\frac12+\frac{2+t}{2q}}s_{\ell_n}+\sum_{j\ge1}n^{(j+1)(\frac12+\frac{t+2}{2q})}s_{j\log_2{n}}\leqa n^{\frac12+\frac{2+t}{2q}}2^{-\ell_n^2}+\sum_{j\ge1}n^{(j+1)(\frac12+\frac{t+2}{2q})}2^{-j^2(\log_2{n})^2}.
\end{align*}
The latter bound is $o(n^{-M})$ for arbitrary $M>0$. Combining the bounds for the terms ${\rm(I)}-{\rm (III)}$ which are all $o(v_n)$ for sufficiently large $d$, together with Lemma \ref{lemev-2}, completes the proof for $(s_l)$ as in (\ref{defsigr}).

Finally we turn to the case of variances as in (\ref{defsig}) and $l_n=\lambda_n$. We again use the event $\cA_n$ from $\eqref{defev-2}$, albeit with $l_n=\lambda_n$. First one notes that if $\sqrt{n}|X_{lk}|\le 1$ and since for $\alpha\ge~\be, l>~\lambda_n$ it holds $s_l\leq s_{\lambda_n}=2^{-(1/2+\alpha)\lambda_n}\leqa 1/\sqrt{n}$, the above bounds for $N_q, D$ can be used with $x_{lk}=1/\sqrt{n}$. In particular, we obtain that, for $l>\lambda_n$, 
\[  \frac{N_q}{D}\1_{\sqrt{n}|X_{lk}| \le 1} \leqa s_l^q \phi(1)^{-1}. \]
For $\sqrt{n}|X_{lk}|>1$ we split further, recalling the definition of $x_{lk}=\sqrt{\frac{t(j+1)\log{n}}{n}}$, $j\ge 0$,
\begin{align*}
\frac{N_q}{D}\1_{\sqrt{n}|X_{lk}| > 1}\1_{|X_{lk}| \le x_{lk}} & = \sum_{p\ge 1} \frac{N_q}{D}\1_{\sqrt{p}< \sqrt{n}|X_{lk}| \le \sqrt{p+1}}  \1_{|X_{lk}| \le x_{lk}} \\
& \leq 
\sum_{p=1}^{t(j+1)\log{n}} \frac{N_q}{D}\1_{\sqrt{p}< \sqrt{n}|X_{lk}| \le \sqrt{p+1}}.  
\end{align*}
As in the proof of Theorem \ref{thm-seq}, given $l>\lambda_n$ and $k\in \cK_l$, we have the bound
\[  \sum_{p=1}^{t(j+1)\log{n}} \frac{N_q}{D}\1_{\sqrt{p}< \sqrt{n}|X_{lk}| \le \sqrt{p+1}}  
\leqa s_l^q \sum_{p=1}^{t(j+1)\log{n}} \frac{\1_{\sqrt{p}< \sqrt{n}|X_{lk}| \le \sqrt{p+1}}}{\phi(\sqrt{p+1})}, \]
where we have used (\ref{qmom}), that is the assumption that $h$ has $q$ moments, together with the facts that, since $l>\lambda_n$, for $\alpha\ge\be$ a) if $\sqrt{n}|X_{lk}|>\sqrt{p}$ then $|X_{lk}|/s_l\geqa1$  and b) it holds $s_l\leqa x_{lk}$.
Then,
\begin{align*}
&E_{f_0}\left[E^{\Pi}[{\rm (II)}|X^{(n)}]\1_{\cA_n}\right]\leq
\sum_{l>\lambda_n}2^{l/2}\left(\sum_{k\in\cK_l}E^n_{f_0}\left[E^\Pi\big[|f_{lk}|^q\,\big|\,X^{(n)}\big]\1_{\cA_n}\right]\right)^{1/q}\\
&\leqa \sum_{\lambda_n<l\le \log_2{(n)}} 2^{l/2}\left(\sum_{k\in\cK_l}s_l^q\{1+
\sum_{p=1}^{t\log{n}} \frac{E_{f_0}\1_{\sqrt{p}< \sqrt{n}|X_{lk}| \le \sqrt{p+1}}}{\phi(\sqrt{p+1})} \}\right)^{1/q} \\
& +  \sum_{j\ge 1} \,\sum_{j\log_2{(n)} < l\le(j+1)\log_2{(n)}} 
2^{l/2}\left(\sum_{k\in \cK_l} s_l^q\{1+ \sum_{p=1}^{t(j+1)\log{n}} \frac{E_{f_0}\1_{\sqrt{p}< \sqrt{n}|X_{lk}| \le \sqrt{p+1}}}{\phi(\sqrt{p+1})}\}\right)^{1/q}.
\end{align*}
Noting that, for $l>\lambda_n$, $k\in \cK_l$, for $\alpha\ge\be$ it holds $\sqrt{n}|f_{0,lk}|\leq 1/\sqrt{\log n}$, and letting $\delta_n=~1/\sqrt{\log n}$, similarly to the proof of Theorem \ref{thm-seq}, we have
\begin{align*}
 E_{f_0} \1_{\sqrt{p}< \sqrt{n}|X_{lk}| \le \sqrt{p+1}}
 \le 2\phi(\sqrt{p}-\delta_n)/(\sqrt{p}-\delta_n),
\end{align*}
and 
\[ \frac{1}{\sqrt{p}-\delta_n}\frac{\phi(\sqrt{p}-\delta_n)}{\phi(\sqrt{p+1})}
\lesssim \frac{e^{\sqrt{p}\delta_n}}{\sqrt{p}}.
 \]

First dealing with the term $l\le \log_2{n}$, one deduces
\begin{align*}
&\sum_{\lambda_n<l\le \log_2{(n)}} 2^{l/2}\left(\sum_{k\in\cK_l}s_l^q\{1+
\sum_{p=1}^{t\log{n}} \frac{P_{f_0}(\sqrt{p}< \sqrt{n}|X_{lk}| \le \sqrt{p+1})}{\phi(\sqrt{p+1})} \}\right)^{1/q}\\
&\leqa \sum_{\lambda_n<l\le \log_2{(n)}} 2^{l(1/2+1/q)}s_l\left(1+
\sum_{p=1}^{t\log{n}} \frac{e^{\sqrt{p}\delta_n}}{\sqrt{p}}\right)^{1/q}\\
&\leqa \sum_{\lambda_n<l\le \log_2{(n)}} 2^{l(1/2+1/q)}s_l\left(1+e^{\delta_n\sqrt{t\log{n}}}
\sum_{p=1}^{t\log{n}} \frac{1}{\sqrt{p}}\right)^{1/q}\\
&\leqa (\log n)^{\frac1{2q}}\sum_{\lambda_n<l\le \log_2{n}}2^{l(1/2+1/q)}s_l \\
&\leqa (\log n)^{\frac1{2q}}\sum_{\lambda_n<l\le \log_2{n}}2^{l(1/q-\alpha)}\leqa (\log n)^{\frac1{2q}}n^{\frac{1/q-\alpha}{1+2\beta}},
 \end{align*}
where one uses the previous bounds and the assumption on $\alpha$ which in particular implies that $1/q-\alpha<0$. The last bound is $o(v_n)$ under the assumption $\alpha\ge\beta+1/q$. Similarly,
\begin{align*}
&\sum_{j\ge 1} \,\sum_{j\log_2{(n)} < l\le(j+1)\log_2{(n)}} 
2^{l/2}\left(\sum_{k\in \cK_l} s_l^q\{1+ \sum_{p=1}^{t(j+1)\log{n}} \frac{P_{f_0}(\sqrt{p}< \sqrt{n}|X_{lk}| \le \sqrt{p+1})}{\phi(\sqrt{p+1})}\}\right)^{1/q}\\
  & \leqa \sum_{j\ge 1}\, \sum_{j\log_2{(n)} < l\le (j+1)\log_2{(n)}} 2^{l(1/2+1/q)}s_l\left(1+ \sum_{p=1}^{t(j+1)\log{n}}\frac{e^{\sqrt{p}\delta_n}}{\sqrt{p}}\right)^{1/q}\\
  & \leqa \sum_{j\ge 1}\, \sum_{j\log_2{(n)} < l\le (j+1)\log_2{(n)}} 2^{l(1/2+1/q)}s_l\left(1+ e^{\delta_n\sqrt{t(j+1)\log{n}}}\sum_{p=1}^{t(j+1)\log{n}}\frac{1}{\sqrt{p}}\right)^{1/q}\\
  & \leqa \sum_{j\ge 1}\, \sum_{j\log_2{(n)} < l\le (j+1)\log_2{(n)}} 2^{l(1/2+1/q)}s_l\{t(j+1)\log{n}\}^{\frac1{2q}}e^{\frac{\sqrt{t(j+1)}}{q}}\\
  &\leqa \sum_{j\ge 1}\, \sum_{j\log_2{(n)} < l\le (j+1)\log_2{(n)}} 2^{l(1/2+1/q)}s_ll^{\frac{1}{2q}} e^{\frac{\sqrt{2tl}}{q\sqrt{\log_2 n}}}\\
  &\leqa \sum_{j\ge 1}\, \sum_{j\log_2{(n)} < l\le (j+1)\log_2{(n)}} 2^{l(1/q+\eta-\alpha)}
  \leqa\sum_{l>\log_2{n}}2^{l(1/q+\eta-\alpha)}\leqa n^{1/q+\eta-\alpha},
\end{align*}
for $\eta>0$ arbitrarily small and where we used that $t(j+1)\log_2{n}\leq 2tj\log_2{n}\leq 2tl$ since $j\ge1$ (and again the assumption on $\alpha$). The latter bound is $o(v_n)$ provided $\eta\le \alpha-\frac{\be}{1+2\be}-1/q$, where such a choice of $\eta>0$ is possible due to the assumption $\alpha\ge \beta+1/q$. 

Putting the previous bounds together one gets for $s_l$ as in (\ref{defsig}) and $l_n=\lambda_n$, that
\[E_{f_0}\left[E^{\Pi}[{\rm (II)}|X^{(n)}]\1_{\cA_n}\right]
\leqa o(v_n),\]
provided $\alpha\ge\beta+1/q$. 
Combining the bounds for the terms ${\rm(I)}-{\rm (III)}$ which are all $o(v_n)$ for sufficiently large $d$, together with Lemma \ref{lemev-2}, completes the proof.
\end{proof}

\begin{proof}[Proof of Theorem \ref{thm-besov}]
Let $J_0, J_1$ be the (closest integer) solutions to 
\begin{align}
 2^{J_0} & =n^{1/(2\be+1)} \label{defj0} \\
 2^{J_1} & = n^{\frac{\be}{2\be+1}\frac{1}{\be'}},\qquad 
\be'=\be-\left(\frac{1}{r}-\frac12\right)_{+}.  \label{defj1}
\end{align}

We first deal with super-light variances  (\ref{defsigr}).  
Let us recall the definition of the set of indices $\cN_n$ from the proof of Theorem \ref{thm-seq} 
\[ \cN_n := \left\{(l,k):\ \ |f_{0,lk}|>1/\rn  \right\}. \]
It follows from the definition of the class $\mathcal{B}_{rr}^\be(L)$, that $|f_{0,lk}|\le L 2^{-l(\be+1/2-1/r)}$. Since $B:=\be+1/2-1/r>0$, one obtains 
\begin{equation}\label{trcn}
\cN_n\subset \left\{(l,k):\ 2^l\le (\sqrt{n}L)^{1/B}\right\}.
\end{equation}
To prove the theorem, one needs to bound\smallskip
\begin{align} E_{f_0} \int \|f-f_0\|_2^2 d\Pi(f\given X^{(n)})
& = E_{f_0} \sum_{(l,k):\, l\le J_0 \text{ or } (l,k)\in\cN_n}
\int (f_{lk}-f_{0,lk})^2 d\Pi(f\given X^{(n)}) \nonumber\\
& + E_{f_0} \sum_{(l,k)\notin\cN_n\,,\, l>J_0} \int (f_{lk}-f_{0,lk})^2 d\Pi(f\given X^{(n)})=:(A)+(B). \label{trbe}
\end{align}
We now transpose the estimates obtained in the proof of Theorem \ref{thm-seq} using the double--index notation. 
For any $l\ge1$ and $k\in\cK_l$, the inequality (\ref{eq:logbnd}) writes
\begin{equation} \label{eq:logbnd2}
 nE_{f_0}\int (f_{lk}-f_{0,lk})^2 d\Pi(f\given X^{(n)})\leqa 
\log^{1+\kappa}
\left(1+\frac{L+1/\sqrt{n}}{s_l} \right). 
\end{equation}
We use this to bound the first expected sum (A) in \eqref{trbe}. We have, using \eqref{trcn}, that if $l\le J_0$ or $(l,k)\in\cN_n$, then $l\leqa \log{n}$. For super-light variances  (\ref{defsigr}), this leads to $\log(s_l^{-1})~\leqa~l^2\leqa~(\log{n})^2$,  
so that for such $l$'s
\[ \log^{1+\ka}\left(1+\frac{L+1/\sqrt{n}}{s_l} \right)\le
\left(C+l^2\right)^{1+\ka}\leqa \log^{2+2\ka}(n).
\]
This implies that the contribution of the term (A) in \eqref{trbe} is 
\[(A) \leqa \log^{2+2\ka}(n) ( 2^{J_0} + 
|\cN_n| )/n\leqa \log^{2+2\ka}(n) 2^{J_0}/n,\]
using $|\{(l,k):\ l\le J_0\}|\leqa 2^{J_0}$ and Lemma \ref{lembesov} with the choice $\delta_n=1$ 
to bound the cardinality $|\cN_n|$ by a constant times $2^{J_0}$. This implies $(A)\leqa (\log{n})^d n^{-2\be/(2\be+1)}$ as desired.

It now remains to deal with the set of indices  $\cJ:=\{(l,k):\ l>J_0,\ |f_{0,lk}|\le 1/\sqrt{n}\}$. These indices are a subset of 
$\{(l,k):\ l\le J_1,\ |f_{0,lk}|\le 1/\sqrt{n}\}\cup\{(l,k):\ l>J_1\}$. By Lemma~\ref{lembesov}, observing that the latter set of indices coincides with the indices in the sum bounded in the second part of the lemma, combining with $(a+b)^2\le 2a^2+2b^2$, we obtain
\begin{align*} 
\sum_{(l,k)\in \cJ} \int (f_{lk}-f_{0,lk})^2 d\Pi(f\given X^{(n)}) & \le
2 \sum_{(l,k)\in \cJ} f_{0,lk}^2  
+ 2 \sum_{(l,k)\in \cJ} \int f_{lk}^2 d\Pi(f\given X^{(n)})\\
& \leqa n^{-\frac{2\be}{2\be+1}} 
+ \sum_{(l,k)\in \cJ} \int f_{lk}^2 d\Pi(f\given X^{(n)}). 
\end{align*}
To deal with the last term, note that this term is exactly, up to the use of the double index notation, the same as the last term bounded in the proof of Theorem \ref{thm-seq} (which in single-index notation concerns the indices $\{k>K_n:\ |f_{0,k}|\le 1/\sqrt{n}\}$). Note that therein one only uses the bound $|f_{0,k}|\le 1/\sqrt{n}$ (and not any other regularity assumption on $f_0$). In particular, this term can be dealt with in exactly the same way as in that proof, up to transposing the event $\cA_n$ therein in double-index notation: the event $\cA_{k,j}$ now becomes, in a similar way as in the proof of Theorem \ref{thm-seq-sup} above, 
\[\cA_{lk,j}:= \left\{ \, |X_{lk}| \le \sqrt{\frac{4(j+1)\log{n}}{n}}\, \right\}, \]
and one transposes the definition of $\cA_n=\cA_n(\cN_n)$ to double indices as 
\[ \cA_n  := \bigcap_{J_0<l \le \log{n}\,,\, (l,k)\in \cJ} \,\cA_{lk,0}\
  \cap\ \bigcap_{j\ge 1}\ \bigcap_{j\log{n} < l \le (j+1)\log{n}\,,\, (l,k)\in \cJ} \,\cA_{lk,j}. 
  \]
 In a similar way as in the proof of Theorem \ref{thm-seq}, this implies 
\begin{equation}\label{highfreqbe}
 E_{f_0}\sum_{(l,k)\in \cJ} \int f_{lk}^2 d\Pi(f\given X^{(n)})\1_{\cA_n}=o(n^{-M})  
\end{equation} 
for arbitrary $M>0$. Let us write the details here for completeness. First, one notes that $P_{f_0}[\cA_n^c]=o(1)$ by the same proof as that as Lemma \ref{lemev-2}, since we work with indices in $\cJ$ for which $|f_{0,lk}|\le 1/\sqrt{n}$ by definition so the proof goes through. For $j\ge 0$ in the sums below, let us set $x_{lk}=\sqrt{(4(j+1)\log{n})/n}$ whenever $j\log_2(n)<l\le (j+1)\log_2(n)$. The  constraint $s_l\leqa x_{lk}$ is trivially satisfied when $l>J_0$ for this choice of $x_{lk}$ and large enough $n$ since then $s_l\leqa 1/\sqrt{n}$. This enables one to use the bound, for $j\log_2(n)<l\le (j+1)\log_2(n)$,
\[ \int f_{lk}^2 d\Pi(f\given X) 
\1_{|X_{lk}| \le x_{lk}} \le s_l^2 \frac{\1_{|X_{lk}|\le x_{lk}}}{\phi(\sqrt{n}x_{lk})}
 \le s_l^2 n^{2(j+1)}.
 \]
Then, using that $(s_l)$ is decreasing, we have that
 \begin{align*}
E_{f_0} & \sum_{(l,k)\in \cJ} \int f_{lk}^2 d\Pi(f\given X^{(n)})\1_{\cA_n}\\
&\leqa \sum_{J_0<l\le\log_2{(n)}} \sum_{k\in\cK_l} s_l^2 n^2
+\,\sum_{j\ge 1}\,\sum_{j\log_2{(n)}<l\le(j+1)\log_2{(n)}}\sum_{k\in\cK_l} s_l^2 n^{2(j+1)}\\
&\leqa n^3 s_{J_0}+\sum_{j\ge1}n^{3(j+1)}s_{j\log_2{n}}^2
\leqa n^3 2^{-\log_2^2{n}/(2\be+1)^2}+\sum_{j\ge1}n^{3(j+1)}2^{-j^2(\log_2{n})^2},
\end{align*} 
which decreases  to $0$ faster than any polynomial in $n$, leading to \eqref{highfreqbe}. 
Combining \eqref{highfreqbe} with the previously obtained bounds leads to an overall upper bound for $(A)+(B)$ in \eqref{trbe} of $C(\log{n})^d n^{-2\be/(2\be+1)}$ as desired, which concludes the proof for super-light variances.

For the case of variances as in (\ref{defsig}), one proceeds again as in the proof of Theorem \ref{thm-seq} by replacing $\cN_n$ by
\[ \cM_n=\{(l,k):\ |f_{0,lk}|>\delta_n/\sqrt{n}\},\]
with $\delta_n=1/\sqrt{\log{n}}$. Now \eqref{trcn} becomes
\begin{equation}\label{trmn}
\cM_n\subset \left\{(l,k):\ 2^l\le (\sqrt{n}L/\delta_n)^{1/B}\right\},
\end{equation}
and if $l\leq J_0$ or $(l,k)\in \cM_n$ we still have $l\lesssim \log{n}$.
In order to bound the term (A), we have, this time,
\[ \log^{1+\ka}\left(1+\frac{L+1/\sqrt{n}}{s_l} \right)\le
\left(Cl\right)^{1+\ka}\leqa \log^{1+\ka}(n).
\]
This implies that the contribution of the term (A) in \eqref{trbe} is 
\[ (A) \leqa \log^{1+\ka}(n) ( 2^{J_0} + 
|\cM_n| )/n\leqa \log^{2+\ka}(n) 2^{J_0}/n,\]
using $|\{(l,k):\ l\le J_0\}|\leqa 2^{J_0}$ and Lemma \ref{lembesov} applied with $\delta_n=1/\sqrt{\log{n}}$ 
to bound the cardinality $|\cM_n|$ by  $C\delta_{n}^{-r} 2^{J_0}$. This implies $(A)\lesssim (\log{n})^d n^{-2\be/(2\be+1)}$ as desired.

It now remains to deal with the set of indices  $\cJ':=\{(l,k):\ l>J_0,\ |f_{0,lk}|\le \delta_n/\sqrt{n}\}$. These indices are a subset of 
$\{(l,k):\ l\le J_1,\ |f_{0,lk}|\le \delta_n/\sqrt{n}\}\cup\{(l,k):\ l>J_1\}$, so as before this leads to
\begin{align*} 
\sum_{(l,k)\in \cJ'} \int (f_{lk}-f_{0,lk})^2 d\Pi(f\given X^{(n)}) & \le
2 \sum_{(l,k)\in \cJ'} f_{0,lk}^2  
+2 \sum_{(l,k)\in \cJ'} \int f_{lk}^2 d\Pi(f\given X^{(n)})\\
& \leqa n^{-\frac{2\be}{2\be+1}} + \sum_{(l,k)\in \cJ'} \int f_{lk}^2 d\Pi(f\given X^{(n)}). 
\end{align*}
To deal with the last term, note that this term is again exactly, up to the use of the double index notation, the same as the last term bounded in the proof of Theorem \ref{thm-seq} for the variances as in (\ref{defsig}).  In a similar way as in the proof of Theorem \ref{thm-seq}, adapting the argument with double-index notation as for the case of super-light variances above, this implies, for some $d>0$, 
\[ E_{f_0}\sum_{(l,k)\in \cJ'} \int f_{lk}^2 d\Pi(f\given X^{(n)})\1_{\cA_n}=O\left((\log{n})^d n^{-2\be/(2\be+1)}\right).\]
This concludes the proof for the variances as in (\ref{defsig}). 
\end{proof}

\begin{proof}[Proof of Theorem \ref{adnpbvm}]
We recall from \cite{cn14}, Proposition 6, two sufficient conditions for the nonparametric BvM to hold for a prior $\Pi$: first, a tightness requirement in $\cM_0(\bar{w})$, for some sequence $(\bar{w}_l)$ with $\bar{w}_l=o(w_l)$ as $l\to\infty$:
\begin{equation} \label{tightness} 
E[ \|f-T_n\|_{\cM_0(\bar{w})} \given X^{(n)}] = O_{P_0}(1/\sqrt{n}),
\end{equation}
for some centering $T_n$ which here can be taken to be the observation sequence $T_n=X^{(n)}=(X_{lk})_{l,k}$. Second, one should verify that finite-dimensional distributions converge, which here means that we should have that for any $J$, if $\pi_J$ denotes the projection onto the space spanned by the first $J$ coordinates of the basis, the vector $\rn(\pi_J(f-X^{(n)}))$ converges in distribution to $\cN(0,1)^{\otimes J}$.  

For tightness, we borrow the bounds on maxima of coordinates over a given level $l$ obtained in the proof of Theorem \ref{thm-seq-sup}: 
\[ E_{f_0}E^\Pi\big[\max_{k\in\cK_l}\sqrt{n}|f_{lk}-f_{0,lk}|\,\big|\,X^{(n)}\big]\leqa l^{(1+\kappa)r/2}, \]
where $r=1$ for $s_l$ as in (\ref{defsig}) and $r=2$ for $s_l$ as in (\ref{defsigr}). 
In the last display, note that the centering $f_{0,lk}$ can be replaced by $X_{lk}$, since $\sqrt{n}(X_{lk}-f_{0,lk})$ verifies the same inequality with $\sqrt{l}$ on the right-hand side instead of the (larger) $l^{(1+\kappa)r/2}$.  
Deduce that the tightness condition \eqref{tightness} holds for $\bar{w}_l=n^{(1+\kappa)r/2}$. 

For finite dimensional convergence, first observe that the posterior makes coordinates independent, so that it is enough to check convergence for individual coordinates, i.e. we want to show that $\rn(f_{lk}-X_{lk})$ converges in distribution (in $P_0$--probability) towards a $\cN(0,1)$ law, for any {\em fixed} indices $l,k$. For any fixed $l$, the prior density on coordinate $f_{lk}$ is  by definition the law of $s_l  \zeta_{lk}$, which by the prior's construction has a continuous and positive density over the whole real line.  For product priors  one can then  use Theorem 1 in \cite{ic12} to derive the desired asymptotic normality (see the paragraph below the statement of Theorem 7 in \cite{cn13} for a detailed argument).
\end{proof}

\subsection{Proofs of remaining results of Section \ref*{sec:prmass}}\label{proof:sec3}

\begin{proof}[Proof of Theorem \ref{thmpriormlinf}]
Let $l_n\ge 2$ be an integer, and for $f$ in $L^2$, recall that $f^{[l_n]}$ denotes its projection onto the linear span of the wavelets $\{\psi_{lk}\}_{l\leq l_n, k\in\cK_l}$ and $f^{[l_n^c]}=f-f^{[l_n]}$. Then, using the triangle inequality and \eqref{bsn}, 
\begin{align*}
&\Pi[\|f-f_0\|_\infty < \veps] \ge 
\Pi\left[ \|f^{[l_n]}-f_0^{[l_n]}\|_\infty < \veps/2 \,,\, \|f^{[l_n^c]}-f_0^{[l_n^c]}\|_\infty < \veps/2\right] \\
& \ge \Pi\left[\forall\, l\le l_n,\ \  2^{l/2}\max_{k\in\cK_l}|f_{lk} - f_{0,lk} | \le \frac{\veps}{cl_n}\; ; \ \forall\, l>l_n,\ \  2^{l/2}\max_{k\in\cK_l}|f_{lk}| \le \frac{\veps}{Dl\log^2{l}} \right] \1_{\|f_0^{[l_n^c]}\|_\infty<\veps/4}\\
& \ge \Pi\left[\forall\, l\le l_n,\,\forall k\in\cK_l, \ |f_{lk} - f_{0,lk} | \le \frac{\veps}{c2^{l/2}l_n}\right] \\&\hspace{5cm}\Pi\left[\ \forall\, l> l_n,\,\forall k\in\cK_l, \  |f_{lk}| \le \frac{\veps}{D2^{l/2}l\log^2{l}} \right] \1_{\|f_0^{[l_n^c]}\|_\infty<\veps/4}\\
& = \prod_{l\le l_n}\prod_{k\in\cK_l}\Pi\left[\ |f_{lk} - f_{0,lk} | \le \frac{\veps} {c2^{l/2}l_n}\right] \\&\hspace{5cm}\prod_{l>l_n}\prod_{k\in\cK_l}\Pi\left[|f_{lk}| \le \frac{\veps}{D2^{l/2}l\log^2{l}} \right] \1_{\|f_0^{[l_n^c]}\|_\infty<\veps/4}
\end{align*}
where we have used independence and the fact that $l^{-1}/\log^2(l)$ is a summable sequence and where $c,D$ are large enough constants.

Suppose the indicator in the last display equals one, which imposes $\|f_0^{[l_n^c]}\|_\infty<\veps/4$, for which a sufficient condition is
\begin{equation}\label{vepscond-inf}
\veps>4L2^{-\be l_n}
\end{equation}
  if $f_0\in \Hbl$ for some $L>0$. 

Let us now bound each individual term $p_{lk}:=\Pi[ |f_{lk}-f_{0,lk}|\le \veps/(c2^{l/2}l_n)]$ for $l\le l_n, k\in\cK_l$. By symmetry, one can assume $f_{0,lk}\ge 0$ and
\begin{align}\label{eq:pk}
 p_{lk} & \ge \int_{f_{0,lk}}^{f_{0,lk}+\veps/(c2^{l/2}l_n)} s_l^{-1}h(x/s_l) dx \ge \frac{\veps}{c2^{l/2}l_n} s_l^{-1}h(C/s_{l_n})\nonumber\\
        &\geq \frac{\veps}{c2^{l/2}l_n} h(C/s_{l_n})
        \geq \frac1c\veps \ e^{-\log(l_n)-(l_n/2)\log{2}-c_1(1+\log^{1+\kappa}(1+C/s_{l_n}))}\\
        &\geq \veps \ e^{-C_1l_n^{1+\kappa}},\nonumber
\end{align} 
where we have used that $(s_l)$ is decreasing and of type (\ref{defsig}) as well as $x\to h(x)$ on $[0,\infty)$ by assumption, that $f_{0,lk}+\veps/(c2^{l/2}l_n)\le C$ since $|f_{0,lk}|$ is bounded for $f_0\in \Hbl$, and assumption (\ref{condti}).   So 
\[ \prod_{l\leq l_n}\prod_{k\in \cK_l} p_{lk} \ge \veps^{2^{l_n+1}} \exp\left\{ -C_1l_n^{1+\kappa}2^{l_n} \right\}, \]
for a new value of the constant $C_1$.

On the other hand, we also need to bound
\begin{align*}
\cP_2:=\Pi\left[\forall\, l>l_n, \,\forall k\in\cK_l, \ \  |f_{lk}| \le \frac{\veps}{D2^{l/2}l\log^2{l}} \right]
&=\prod_{l>l_n}\prod_{k\in\cK_l} (1-2 \overline{H}(\veps/\{Ds_l2^{l/2}l\log^2{l}\}))\\
& = \prod_{l>l_n}\prod_{k\in\cK_l} (1-2 \overline{H}(\veps 2^{l \alpha}/\{Dl\log^2{l}\})).
\end{align*}
Note that if 
\begin{equation}  \label{defeps-inf}
\veps\ge D'2^{-l_n\be}l_n\log^2{l_n},
\end{equation}
then for $l>l_n$
\[  \veps 2^{l\alpha}/\{Dl\log^2{l}\} \ge 2^{l\alpha-l_n\be}\,\frac{l_n}{l}\frac{\log^2{l_n}}{\log^2{l}}\frac{D'}{D} \]
so that, as long as $\alpha\ge \be$ there is a large enough $D'>0$ such that the last term is at least 1 for all $l>l_n$ and \eqref{vepscond-inf} is satisfied.   Then, by using (\ref{condts}), there exists a constant $C_2>0$ (which can be made arbitrarily small by taking $D'$ large) such that
\[ \overline{H}(\veps 2^{l \alpha}/\{Dl\log^2{l}\})
\le C_2 2^{2l_n\be-2l\alpha}\left(\frac{l\log^2{l}}{l_n\log^2{l_n}}\right)^2.
 \]
Then, possibly enlarging $D'$ in \eqref{defeps-inf} further in order to have that the right hand side in the last display is less than $1/4$, using the inequality $\log(1-2x)\geq -4x$ for all $x\in[0,1/4)$, and as long as $\alpha> 1/2$, so that the series $\sum 2^{l(1-2\alpha)}$ is converging,
\begin{align*}
 \cP_2 & \ge \exp\left\{ \sum_{l>l_n}\sum_{k\in\cK_l}\log\left(1-2C_2 2^{2l_n\be-2l\alpha}\left(\frac{l\log^2{l}}{l_n\log^2{l_n}}\right)^2\right) \right\} \\
& \ge \exp\{ -C_3 \frac{2^{2l_n\be}}{l_n^2\log^4{l_n}} \sum_{l>l_n}\sum_{k\in\cK_l}2^{-2l\alpha}l^2\log^4(l) \}\\
& = \exp\{ -C_3 \frac{2^{2l_n\be}}{l_n^2\log^4{l_n}} \sum_{l>l_n}2^{l(1-2\alpha)}l^2\log^4(l) \} \ge \exp(-C_4 2^{l_n} \cdot 2^{2l_n(\be-\alpha)}),
\end{align*}
where we have used, whenever $\alpha>1/2$, that $\sum_{l>l_n} 2^{l(1-2\alpha)}l^2 \log^4{l}=O( 2^{l_n(1-2\alpha)}l_n^2\log^4{l_n})$ as $l_n\to\infty$. The bound of the last display is at least $\exp(-C_42^{l_n})$ assuming $\alpha\ge \be$. 

Putting everything together, for $D'>0$ in \eqref{defeps-inf} sufficiently large so that $D'l_n\log^2{l_n}\ge1$, one gets
\begin{align*}
\Pi[ \|f-f_0\|_\infty< \veps ] &  \ge 
\veps^{2^{l_n+1}} \exp\left\{ -C_1l_n^{1+\kappa}2^{l_n} -C_42^{l_n}\right\} \\
& \ge \exp \left\{ -C_5l_n 2^{l_n+1}-C_1l_n^{1+\kappa}2^{l_n} -C_42^{l_n}\right\} \\
& \ge \exp \left\{ -C_6 l_n^{1+\kappa}2^{l_n}\right\}.
\end{align*} 


Recall the definition of $\veps_n$ as in (\ref{rateveps-inf-1}) and the constants $d_1,d_2$ from the statement of the Theorem. 
Set $l_n$ such that $2^{l_n}=M_1 n^{1/(1+2\be)}(\log n)^{(2-(1+\kappa))/(1+2\be)}(\log\log{n})^{4/(1+2\be)}$, with $M_1>0$ to be chosen below. It is then straightforward to check that there exists a constant $M_2>0$ such that $2^{-l_n\be}l_n\log^2l_n\leq M_1^{-\beta}M_2\veps_n$, hence \eqref{defeps-inf} holds for $\veps=d_1\veps_n$ and large enough $n$, provided $d_1\ge D'M_1^{-\beta}M_2$. Furthermore, \[C_6 l_n^{1+\kappa}2^{l_n}\le d_2n\veps_n^2,\] provided $M_1\leq d_2/(C_6M_2^{1+\kappa})$. Hence for any $d_2>0$, the constants $M_1$ and $d_1$ can be chosen so that the latter two conditions hold simultaneously, 
which concludes the proof for $s_l$ as in (\ref{defsig}).

Let us now turn to the case of $s_l$ given by (\ref{defsigr}). The term with $p_{lk}$'s is now bounded using \eqref{eq:pk}, similarly as in the case of the prior (\ref{defsig}), by 
\[ \prod_{l\leq l_n}\prod_{k\in \cK_l} p_{lk} \ge \veps^{2^{l_n+1}} \exp\left\{ -C_1l_n^{2+2\kappa}2^{l_n} \right\}. \]
On the other hand, we also have 
\begin{align*}
\cP_2:=\prod_{l>l_n}\prod_{k\in\cK_l} (1-2 \overline{H}(\veps/\{Ds_l2^{l/2}l\log^2{l}\})) = \prod_{l>l_n}\prod_{k\in\cK_l} (1-2 \overline{H}(\veps 2^{l^2-l/2}/\{Dl\log^2{l}\})).
\end{align*}
Let $\veps\ge D2^{-\be l_n}$, then for large enough $l_n$,  
$\veps 2^{l^2-l/2}/\{Dl\log^2{l}\}\ge 1$ and by (\ref{condts})
\[\overline{H}(\veps 2^{l^2-l/2}/\{Dl\log^2{l}\}\})
\le c_2(\veps 2^{l^2-l/2}/\{Dl\log^2{l}\})^{-2}
\le 2^{-l^2}
 \]
so that again for sufficiently large $l_n$ it holds \[\cP_2\ge \exp\{ -C\sum_{l>l_n}\sum_{k\in\cK_l} 2^{-l^2}\}\ge \exp\{-C' 2^{l_n-l_n^2}\}.\] The latter is bounded from below by a constant, so the final bound obtained for the probability at stake is $\exp\{-C'l_n^{2+2\kappa}2^{l_n}\}$ (for a new value of the constant $C'$).  Let us set $\veps=d_1\veps_n$ and $l_n$ such that $2^{l_n}=(d_1\veps_n/D)^{-1/\be}$. Noting that there exists a constant $M>0$ such that for large $n$ it holds $l_n\le M\log n$, we have that there exists $C''>0$ independent of $d_1, d_2$ such that 
\[C'l_n^{2+2\kappa}2^{l_n}\leq C''d_1^{-1/\be}n\veps_n^2.\]
The last display is less than $d_2n\veps_n^2$ for large $d_1$, which concludes the proof.
\end{proof}
 
 \begin{remark}[Cauchy tails] \label{rem:cauchyinfty}
Similarly to Remark \ref{rem:cauchy}, we note that the case of $H$ equal to the Cauchy distribution can be accommodated up to a slight variation on the condition for the prior HT$(\al)$. Suppose in this case that $\alpha>1$ (recall again that we have the choice of $\al$, and that in view of the Theorem, the larger $\al$ is, the larger the range for which adaptation occurs, so we can always choose $\al>1$ beforehand). Indeed, in this case
for $\alpha> 1$ one gets, for $s_l$ as in (\ref{defsig}),
\begin{align*}
 \cP_2 & \ge \exp\left\{ \sum_{l>l_n}\sum_{k\in\cK_l}\log\left(1-2C_2 2^{l_n\be-l\alpha}\frac{l\log^2{l}}{l_n\log^2{l_n}}\right) \right\}\\
 & = \exp\{ -C_3 \frac{2^{l_n\be}}{l_n\log^2{l_n}} \sum_{l>l_n}2^{l(1-\alpha)}l\log^2(l) \} \ge \exp(-C_4 2^{l_n} \cdot 2^{l_n(\be-\alpha)}),
\end{align*}
and from there on the proof is identical to that of Theorem \ref{thmpriormlinf}. A similar remark applies to the case of $s_l$ as in (\ref{defsigr}), this time with no extra condition (the latter choice is free of $\al$).
\end{remark}
 
\begin{proof}[Proof of Theorem \ref{thmdensity}]
Let $r_n=M\veps_n$ be a sufficiently large multiple of $\veps_n$, where the required value of $M$ is made explicit below.  
Using Lemma \ref{lemlinkd}, we know that  the KL-neighborhood $B_n(g_0,r_n)$ contains an $L^\infty$--ball $\{f:\ \|f-f_0\|_\infty \le cr_n\}$, for $c>0$ a small enough universal constant (provided $M\veps_n\le 1$ say, which is the case for $n$ large enough).

On the other hand,  using Theorem \ref{thmpriormlinf} with $d_2=\rho$, we know  that there exists $d_1>0$ such that 
\[ \Pi[\|f-f_0\|_\infty \le d_1\veps_n] \ge \exp(-\rho n\veps_n^2),\]
which is larger than $ \exp(-\rho n r_n^2)$ provided $M\ge 1$. Thus provided $cM\ge d_1$ (i.e. $M\ge d_1/c$), 
\[ \Pi[B_n(g_0, r_n)] \ge \exp(-\rho n r_n^2). \] 
By Theorem  \ref{alphapost}, the $\rho$--posterior thus converges at rate $Cr_n^2 \rho/(1-\rho)\asymp \veps_n^2$ in terms of $D_\rho/n$. By the paragraph below the statement of Theorem  \ref{alphapost}, one deduces that  the $\rho$--posterior converges at rate $C'\veps_n$ in terms of the $\|\cdot\|_1$--norm, which concludes the proof. 
\end{proof}

\begin{proof}[Proof of Theorem \ref{thmclassif}]
One proceeds similarly as in the proof of Theorem \ref{thmdensity}. Using Lemma \ref{lemlinkc}, one obtains that the mass of the KL-neighborhood in Theorem \ref{alphapost} is bounded from below by the mass of an $L^\infty$--neighborhood of $f_0=\La^{-1}h_0$ of size constant times $\veps_n$. The result follows by combining Theorems \ref{thmpriormlinf} and \ref{alphapost} (and the paragraph below it), 
using again that the R\'enyi--entropy can be bounded from below in terms of the $L^1$--distance between individual densities in the model, which coincides with $\|p_f-p_{f_0}\|_{G,1}$.
\end{proof}

\begin{proof}[Proof of Theorem \ref{thm-besov-rho}]
We start by recalling some notation from the proof of Theorem \ref{thm-besov}. Let $J_0$ be the (closest integer) solution to $2^{J_0}=n^{1/(2\be+1)}$. The index $J_1$ is defined as the (closest integer) solution to
\[ 2^{J_1} = n^{\frac{\be}{2\be+1}\frac{1}{\be'}},\qquad 
\be'=\be-\left(\frac{1}{r}-\frac12\right)_{+}. \]
We also set $\veps_n^*:=2^{-\be J_0}$. The following set of indices will also be used
\[ \Sigma_0' := \left\{(l,k):\ l\le J_1,\ |f_{0,lk}|>1/\rn  \right\}\cup\{(l,k): l\le J_0\}. \]
By Theorem \ref{alphapost} for tempered posteriors and the paragraph below its statement, in the white noise model the $\rho$--posterior convergence in $\|\cdot\|_2$--loss follows if the  Kullback--Leibler type neighborhood $B_n(f_0,\veps_n)$, which here coincides with the $\veps_n$--ball in $L^2$ around $f_0$, has prior mass bounded below by $e^{-n\rho\veps_n^2}$.  

{\em Step 1.} By Lemma \ref{lembesov}, for some constant $A$,  
\begin{equation}\label{besovimp}
\cB_n:= \sum_{l\le J_1,k} f_{0,lk}^2 \1_{|f_{0,lk}|\le 1/\sqrt{n}}
 +  \sum_{l> J_1,k} f_{0,lk}^2
 \le  A {\veps_n^*}^2
\end{equation} 
and  the cardinality $|\Sigma_0'|$ of $\Sigma_0'$ verifies
\begin{equation}\label{besovcard}
|\Sigma_0'|\leqa 2^{J_0}.
\end{equation}

{\em Step 2.} 
Using the same notation as in the proof of Theorem \ref{thmpriormlinf}, letting in the next lines $\veps> C_1\veps_n^*$ where $C_1>0$ is a sufficiently large constant and $\cB_n$ is as in \eqref{besovimp}, and for $D$ a large enough constant,
\begin{align*}
\Pi&[ \|f-f_0\|_2 < \veps] \ge 
\Pi\left[ \|f^{[J_1]}-f_0^{[J_1]}\|_2 < \veps/2 \,,\, \|f^{[J_1^c]}-f_0^{[J_1^c]}\|_2 < \veps/2 \right] \\
& \ge\Pi\left[ \forall\, (l,k)\in \Sigma_0',\ \  |f_{lk}-f_{0,lk}| \le \frac{\veps}{D\sqrt{|\Sigma_0'|}} \right]  \cdot \\
&\qquad \Pi\left[ \forall (l,k)\notin\Sigma_0':\  l\le J_1,\ \ | f_{lk}| \le \frac{\veps}{D2^{J_1/2}} \right]  \cdot\\
&\qquad \Pi\left[ \forall\, (l,k):\ l>J_1,\ \ |f_{lk}| \le \frac{\veps}{Dl2^{l/2}} \right] \cdot
\1_{\sqrt{\cB_n}<\veps/8}.
\end{align*}
We bound each term separately from below. First note that since $\beta>1/r-1/2$ by assumption, the indicator function takes the value 1 for the considered values of $\veps$. For the first term, which involves indices with $l\le J_1$ and $(l,k)\in\Sigma_0'$, by the same argument as for the terms $l\le l_n, k\in\cK_l$ in the proof of Theorem \ref{thmpriormlinf}, one obtains
\[\Pi\left[ \forall\, (l,k)\in \Sigma_0',\ \  |f_{lk}-f_{0,lk}| \le \frac{\veps}{D\sqrt{|\Sigma_0'|}} \right] \ge \left(\veps e^{-C_2 J_1^{2+2\kappa}} \right)^{|\Sigma_0'|}, \]
for $C_2>0$ a large enough constant. For the two other terms, by a similar argument as for the terms $l>l_n$ in the proof of Theorem \ref{thmpriormlinf}, one obtains that the probabilities at stake are bounded from below by a constant. Indeed, this follows from the fact that for large enough $n$ and for $\veps>D2^{-\beta J_0}$, for $J_0<l\le J_1$ and for $l>J_1$, respectively, it holds
\[\frac{\veps}{Ds_l2^{J_1/2}}\ge2^{l^2-J_1/2-\beta J_0}\ge 1, \quad\quad \frac{\veps}{Ds_ll2^{l/2}}=\frac{2^{l^2-l/2-\beta J_0}}{l}\ge 1,\] which by condition (\ref{condts}) result in the respective bounds
\[\overline{H}(\frac{\veps}{Ds_l2^{J_1/2}})\leq 2^{-l^2},\quad\quad\overline{H}(\frac{\veps}{Ds_ll2^{l/2}})\leq 2^{-l^2}.\] 
Combining, we deduce that there exists a sufficiently large constant $C_3>0$ such that
\[ 
\Pi[ \|f-f_0\|_2 < \veps] \ge e^{-C_3 (\log{n})^{2+2\kappa} 2^{J_0}}.
\] 
The result follows by setting $\veps=\veps_n\asymp (\log{n})^{1+\kappa} 2^{-\beta J_0}~\asymp~(\log{n})^{1+\kappa} n^{-\be/(2\be+1)}$. Note that by choosing the constant in the definition of $\veps_n$ sufficiently large, the right hand side in the obtained lower bound in the last display becomes larger than $e^{-n\rho \veps_n^2}$, as required. 
\end{proof}

\section{Technical lemmas}
\subsection{Lemma for the regression setting}

\begin{lemma} \label{lemev-2}
Let $\cA_n$ be the event as in \eqref{defev-2}, with either $l_n=\ell_n$ or $l_n=\lambda_n$, and assume $f_0\in \Hbl$ for some $\beta, L>0$. Then for large enough $n$ it holds
\[ P_{f_0}[\cA_n^c] \leqa 1/\sqrt{n}.  \]
\end{lemma}
\begin{proof}
First, for either considered choice of $l_n$, for any $l>l_n, k\in\cK_l$ and $j\ge 0$, we have $P_{f_0}[\cA_{l,k,j}^c]\le P[ |\cN(0,1)|>\sqrt{3(j+1)\log{n}}] \leqa n^{-\frac32(j+1)}$, where one uses that for $l>l_n$ we have $|f_{0,lk}|\leq L\sqrt{\log{n}}/\sqrt{n}$ for large enough $n$.  From this one deduces
\[P_{f_0}[\cA_n^c] \lesssim  \sum_{j\ge 0} 2^{(j+1)\log_2{n}} n^{-\frac32(j+1)}\leqa (\sum_{j\ge0} n^{-j/2})/\sqrt{n}\leqa1/\sqrt{n}.
\qedhere\]
\end{proof}

\subsection{Relating KL-neighborhoods to $L^\infty$--balls}
\begin{lemma}[Lemma 3.1 in \cite{vvvz08}] \label{lemlinkd}
For measurable functions $v,w$ on $[0,1]^d$ and $p_v=e^v/\int_0^1 e^v$,
 we have (with constants in the inequalities $\leqa$ only depending on an upper-bound on $\|v - w\|_\infty$) 
\begin{align*}
 K(p_v,p_w) & \leqa \|v-w\|_\infty^2(1+\|v-w\|_\infty)e^{\|v-w\|_\infty}, \\
 V(p_v,p_w) & \leqa \|v-w\|_\infty^2(1+\|v-w\|_\infty)^2e^{\|v-w\|_\infty}.
\end{align*} 
\end{lemma}

\begin{lemma}[Lemma 3.2 in \cite{vvvz08} for logistic link function] \label{lemlinkc}
For measurable functions $v,w$ on $[0,1]^d$, $\La$ the logistic link function, $h_f(x)=\La( f(x) )$ and $\|\cdot\|_{2,G}$ the $L^2(G)$ norm,
\begin{align*}
 K(h_v,h_w) & \leqa \|v-w\|_{2,G}^2 \\
 V(h_v,h_w) & \leqa \|v-w\|_{2,G}^2.
\end{align*} 
In particular, both quantities are bounded from above by a constant times $\|v-w\|_\infty^2$.
\end{lemma}

\subsection{An auxiliary lemma relating to Besov smoothness}
Let us recall the notation 
$2^{J_0}=n^{1/(2\be+1)}$ and $2^{J_1} = n^{\frac{\be}{2\be+1}\frac{1}{\be'}}$ with $\be'=\be-\left(\frac{1}{r}-\frac12\right)_{+}$. 

\begin{lemma} \label{lembesov}
Let $f_0\in \Bblr$ for some $1\le r\le 2$,  let $\be, L>0$, and $\be>1/r-1/2$. Let $J_0, J_1$ be defined as above.  Then for a large enough constant $A$,  for any $0<\delta_n\le 1$, the set
$\Sigma_0:= \left\{(l,k):\  |f_{0,lk}|>\delta_n/\rn  \right\}$  has cardinality  $|\Sigma_0|$ bounded by
\[ |\Sigma_0|\le A \delta_n^{-r} 2^{J_0},\]
 and we also have
\[ \sum_{l\le J_1,k} f_{0,lk}^2 \1_{|f_{0,lk}|\le \delta_n/\sqrt{n}}
 +  \sum_{l> J_1,k} f_{0,lk}^2
 \le  A n^{-2\be/(2\be+1)}.  \]
\end{lemma}

\begin{proof}
The proof is closely related to the proof of Proposition 10.3 in \cite{wasa},  with a slightly different choice of threshold (we take it of order $\delta_n/\sqrt{n}$ here). One first proves the cardinality claim. It is enough to prove it for the smaller set $\Sigma_0'=\Sigma_0 \cap \{(l,k):\ l>J_0\}$ since $|\{(l,k):\ l\le J_0\}|\leqa 2^{J_0}$. By definition of $\Sigma_0'$,
\[ \delta_n^r n^{-r/2}|\Sigma_0'| \le \sum_{l\ge J_0,\, k} |f_{0,lk}|^r, \]
which in turn is bounded from above by, using that $\be+1/2-1/r>0$ by assumption,
\[ 2^{-J_0r(\be+1/2-1/r)}\sum_{l\ge J_0,\, k} 2^{lr(\be+1/2-1/r)} |f_{0,lk}|^r
\le L^r 2^{-J_0r(\be+1/2-1/r)}=L^r 2^{-J_0(\be+1/2)r/2}2^{J_0}, 
\] 
which is bounded from above by a constant times $n^{-r/2} 2^{J_0}$, so that   $|\Sigma_0'|\leqa \delta_n^{-r} 2^{J_0}$ as required. For the second inequality one bounds each term separately. For $r\le 2$,
\[  \sum_{l\le J_1,k} f_{0,lk}^2 \1_{|f_{0,lk}|\le \delta_n/\sqrt{n}} \le 
\sum_{l\le J_0,k} f_{0,lk}^2 \1_{|f_{0,lk}| \le \delta_n/\sqrt{n}}
+ (\delta_n/\sqrt{n})^{2-r}
 \sum_{l>J_0\,,\,k} |f_{0,lk}|^r. \] 
The last sum has been bounded just above while the first is at most $C\delta_n^2 2^{J_0}/n$. We deduce that the left hand side of the last display is bounded from above by
\[ C\delta_n^2 2^{J_0}/n + C(\delta_n/\sqrt{n})^{2-r} n^{-r/2}2^{J_0}\leqa 2^{J_0}/n,\]
which is bounded by a constant times $n^{-2\be/(2\be+1)}$ as required. For the indices with $l>J_1$, we use the embeddings $\cB_{rr}^{\be}\subset \cB_{2r}^{\be'}\subset \cB_{22}^{\be'}$  (e.g. \cite{ginenicklbook}, Proposition 4.3.6) which imply
\[  \sum_{l> J_1,k} f_{0,lk}^2 \leqa 2^{-2J_1\be'} \sum_{l>J_1,k} 
2^{2j\be'}f_{0,lk}^2\leqa 2^{-2J_1\be'}\leqa n^{-2\be/(2\be+1)},\]
which concludes the proof. 
\end{proof}

\section{Convergence for $\rho$--posteriors}\label{sec:rhopost}

Let us give sufficient conditions for nonparametric contraction of $\rho$--posterior distributions. Unlike for the full Bayesian posterior, testing or entropy conditions are not needed to obtain contraction rates in the R\'enyi-divergence for the fractional posterior when $\rho_n<1$ \cite{tongz06}, see also \cite{kruijervdv13,BPY}.  We state it in the form of a recent result from \cite{ltcr23} that gives a  precise dependence on both $n$ and $\rho$ (since here we consider only fixed $\rho\in(0,1)$, one could also use e.g. the result in \cite{BPY};  the result below gives a sharper dependence in $\rho$, which is relevant if one wishes to allow $\rho$ to go to $0$).

For $0<\rho<1$,  the $\rho$--R\'enyi divergence  between two densities $f$ and $g$ with respect to $\mu$ is 
	\begin{align*}
		D_\rho(f, g) = -\frac{1}{1-\rho}\log \left( \int f^\rho g^{1-\rho}d\mu \right).
	\end{align*}
Consider a statistical model $\{P_\eta^{(n)},\ \eta\in \mathcal{H}\}$ with observations $Y^n$ and suppose it is dominated, i.e. all $P_\eta^{(n)}$ are absolutely continuous with respect to a common measure $\mu$, with density $p_\eta^{(n)}$. 
For two densities $f, g$ dominated by $\mu$, one denotes $K(f, g)=\int f \log(f/g) d\mu$, $V(f,g)= \int f \left(\log(f/g) - K(f,g)\right)^2 d\mu$ and one defines a neighborhood of the true $\eta_0$ by, for any $\veps>0$,
		\begin{align*}
		B_n(\eta_0,\eps) &=\left\{\eta : \ K(p_{\eta_0}^{(n)}, p_{\eta}^{(n)}) \leq n \eps^2,\: V(p_{\eta_0}^{(n)},  p_{\eta}^{(n)}) \leq n \eps^2  \right\}.
	\end{align*}	 
In the next lines $E_{\eta_0}$ refers to the expectation with respect to $P_{\eta_0}^{(n)}$.
	
	\begin{theorem}[Theorem 17 in \cite{ltcr23}] \label{alphapost} Suppose, for $\eps_n>0$, $\rho_n\in(0,1)$ and $n\rho_n\eps_n^2 \rightarrow \infty$, 
		\begin{align}\label{equation_thm_1}
			\Pi(B_n(\eta_0, \eps_n)) \geq e^{-n\rho_n\eps_n^2}.
		\end{align}
Then	 there exists $C>0$ such that as $n\to\infty$,
		\begin{align*}
			E_{\eta_0}\Pi_{\rho_n}\left( \eta: \: \frac{1}{n}D_{\rho_n}(p_{\eta}^{(n)}, p_{\eta_0}^{(n)}) \geq C \frac{\rho_n\eps_n^2}{1-\rho_n}  \given  Y^n\right) \to 0.
		\end{align*}	
	\end{theorem}

 In the Gaussian white noise model, one can directly compute $D_{\rho_n}(f,f_0) = n\rho_n\|f-~f_0\|_2^2/2$, so that the conclusion above becomes, with now $\eta_0=f_0$,
		\begin{align*}
			E_{f_0}\Pi_{\rho_n}\left( f: \: \|f-f_0\|_2 \geq C \frac{\eps_n}{\sqrt{1-\rho_n}}  
			 \given Y^n\right) \to 0.
		\end{align*}	 
 A similar conclusion holds in density estimation with $L^1$--loss, where $P_f^{(n)}=P_f^{\otimes n}$ for $P_f$ the law of density $f$, and one has $D_{\rho_n}(P_f^{\otimes n},P_{f_0}^{\otimes n}) \ge n\rho_n \|f-f_0\|_1^2/2$ (\cite{van_Erven_2014}, Theorem 31)
  thereby giving the same rates as for $L^2$--loss in Gaussian white noise as just above.

\section{Additional Simulations}\label{sec:adsim}
\subsection{Inverse regression with Sobolev/spatially homogeneous truth (continued from Section \ref{sec:sim})}\label{ip:sim}
In Figures \ref{fig-rhopostSob-d1} and \ref{fig-rhopostSob-d2} we present $\rho$-posteriors for $\rho=1, 0.6, 0.2$ and for various noise levels, for the two considered Student priors. As before HT$(\al)$ is overconfident for moderate noise levels. The performance of the considered $\rho$-posteriors is comparable to the one of the classical posterior, albeit with more variability for smaller $\rho$'s.

\begin{figure}[hbt!]
    \centering
    \includegraphics[width=0.97\textwidth]{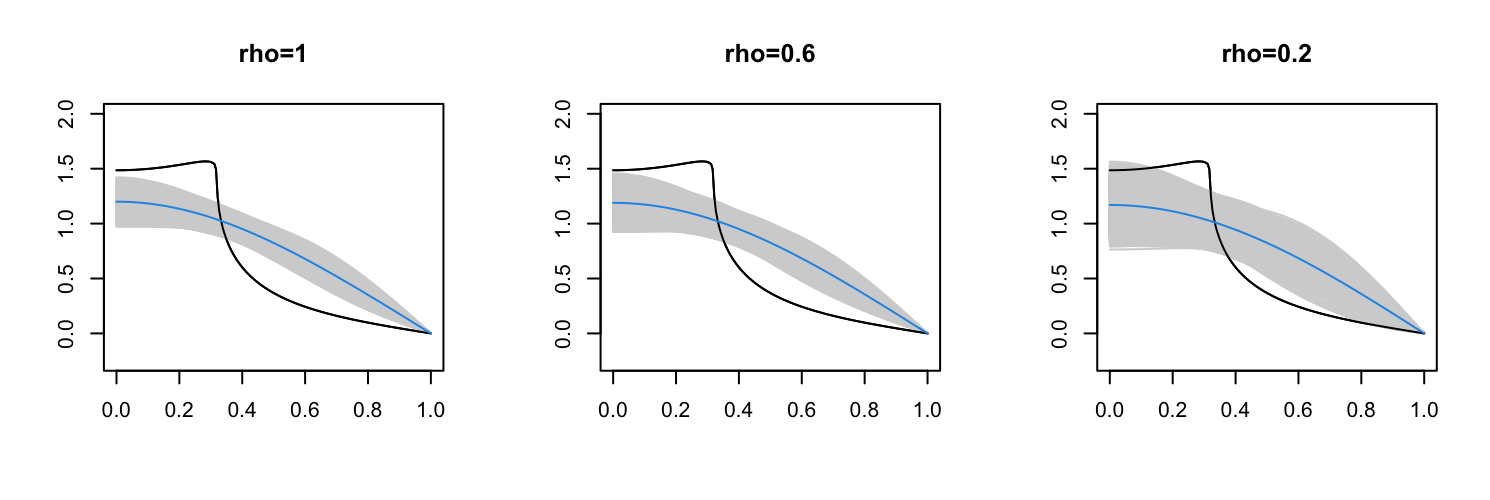}
    \includegraphics[width=0.97\textwidth]{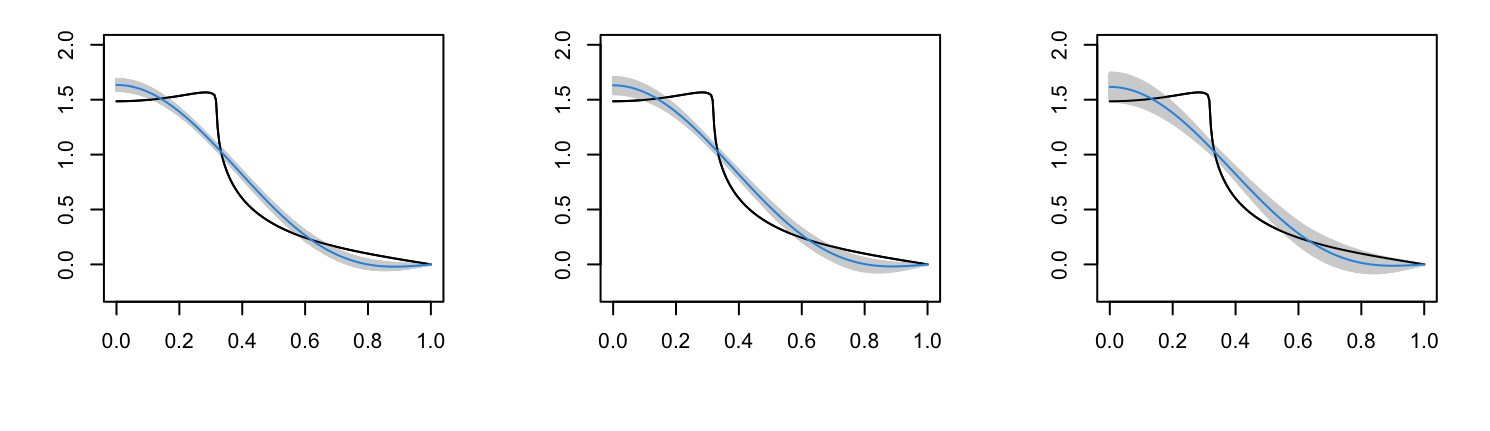}
    \includegraphics[width=0.97\textwidth]{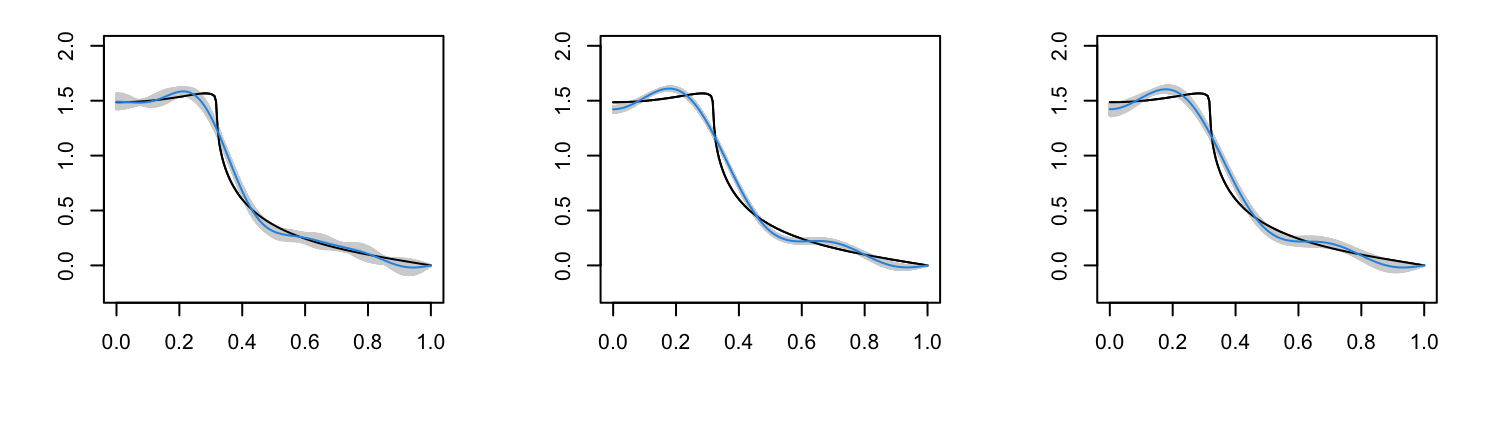}
    \includegraphics[width=0.97\textwidth]{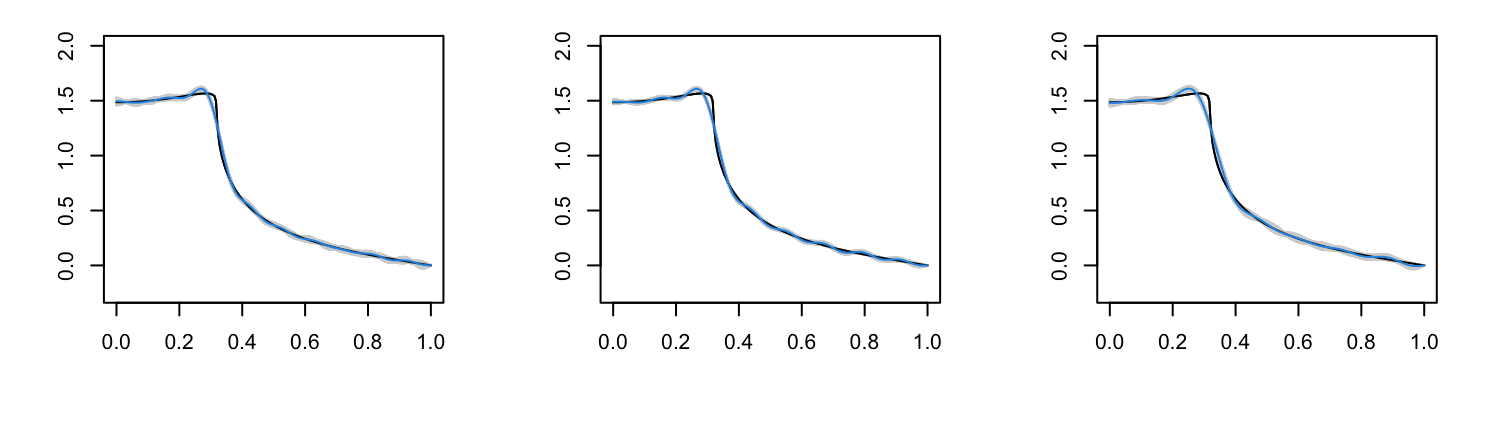}
    \includegraphics[width=0.97\textwidth]{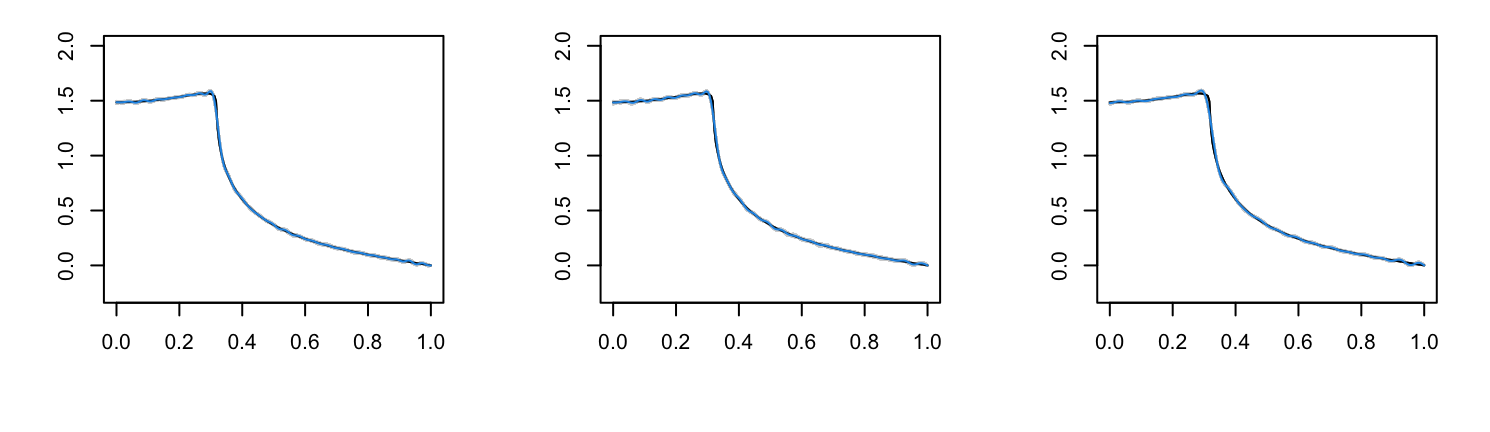}
    \caption{White noise model: true function (black), mean (blue), 95\% credible regions (grey) for the $\rho$-posterior arising for the HT$(\alpha)$ prior, for $\rho=1, 0.6, 0.2$ left to right and for $n=10^3, 10^5, 10^7, 10^9, 10^{11}$ top to bottom.}
    \label{fig-rhopostSob-d1}
\end{figure}

\begin{figure}[hbt!]
    \centering
    \includegraphics[width=0.97\textwidth]{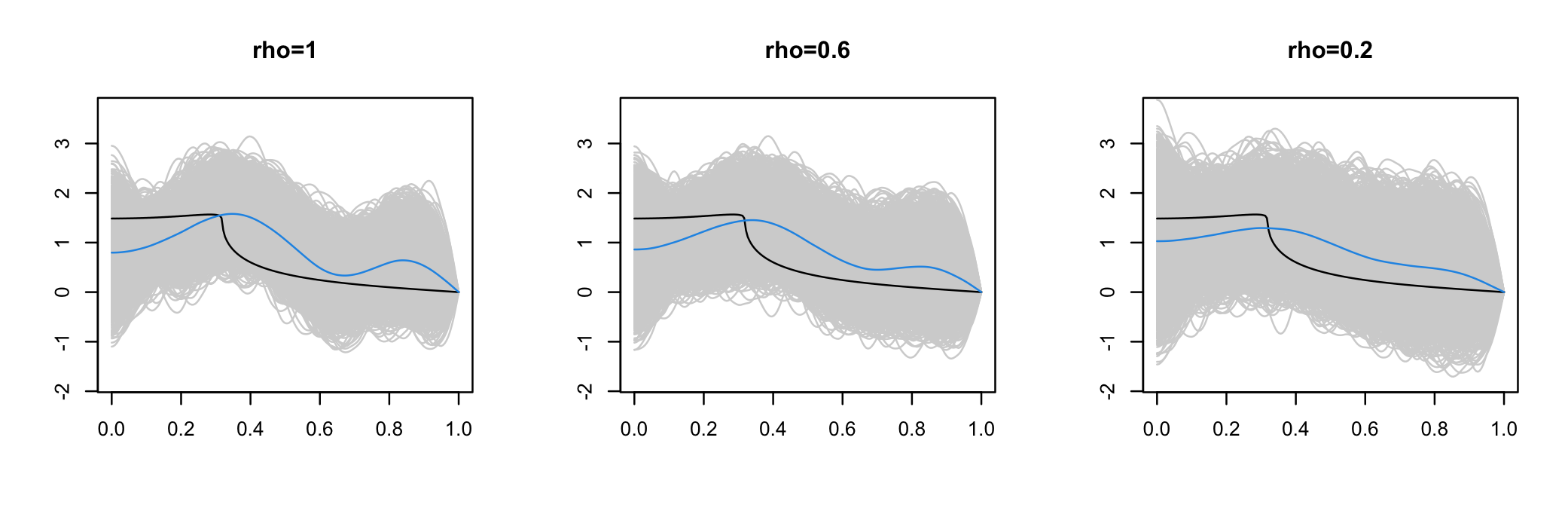}
    \includegraphics[width=0.97\textwidth]{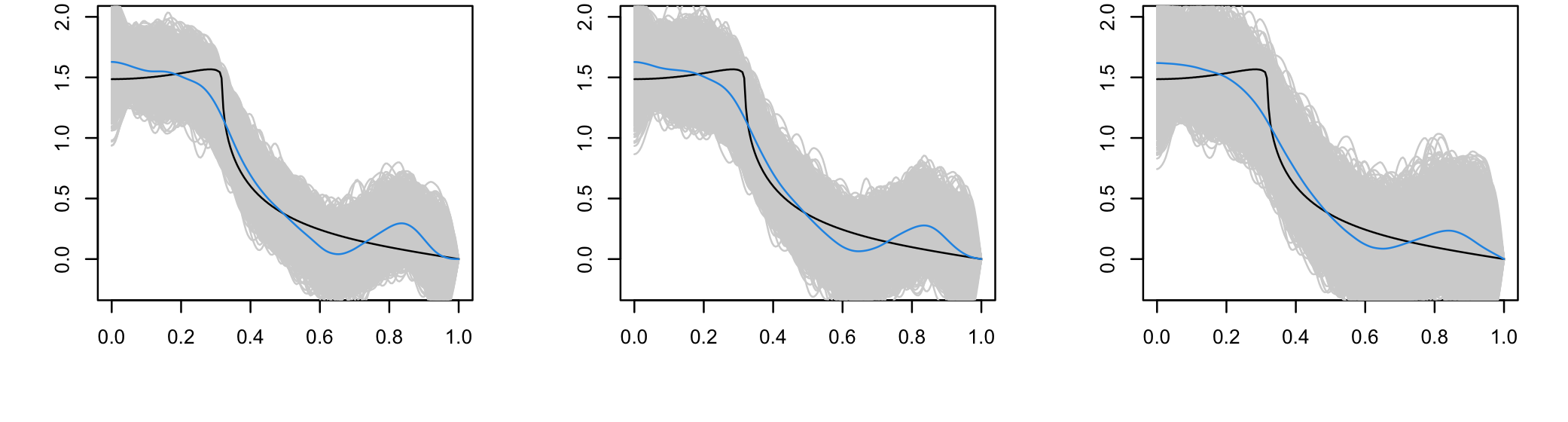}
    \includegraphics[width=0.97\textwidth]{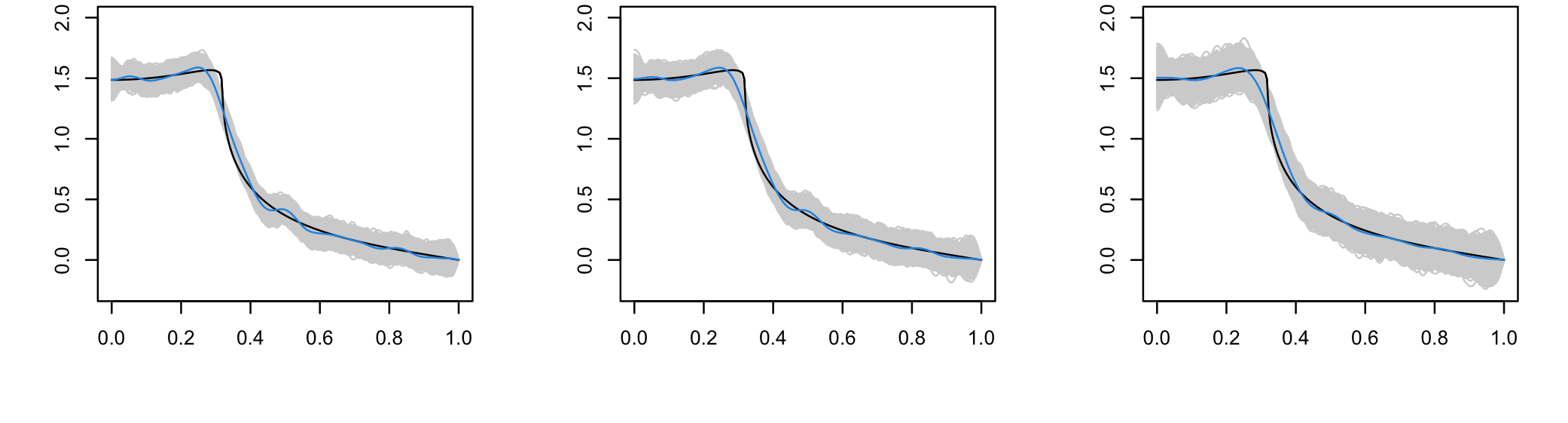}
    \includegraphics[width=0.97\textwidth]{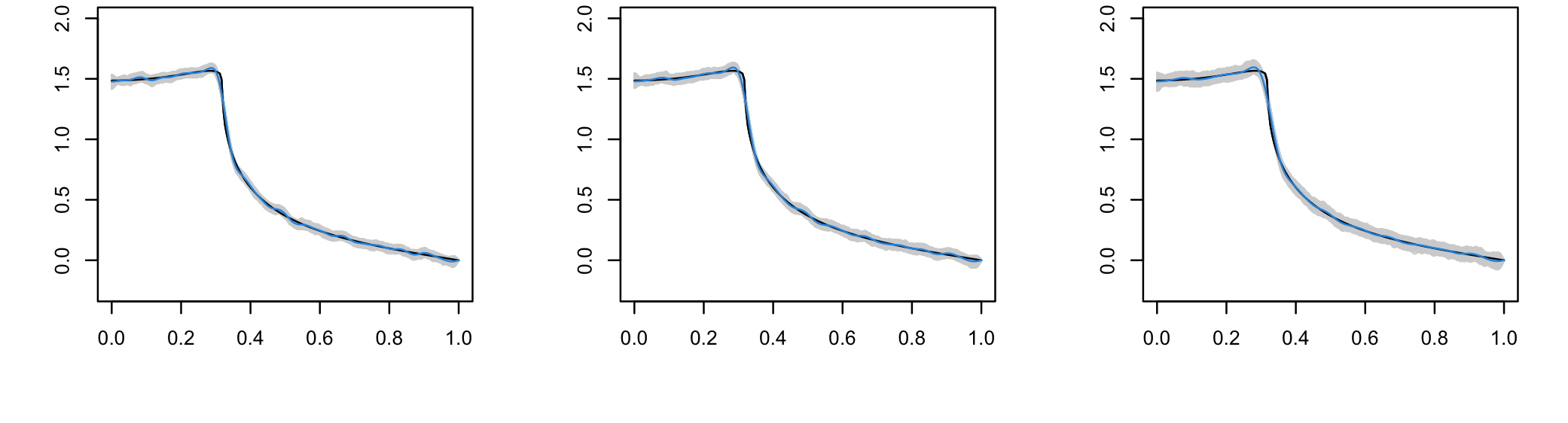}
    \includegraphics[width=0.97\textwidth]{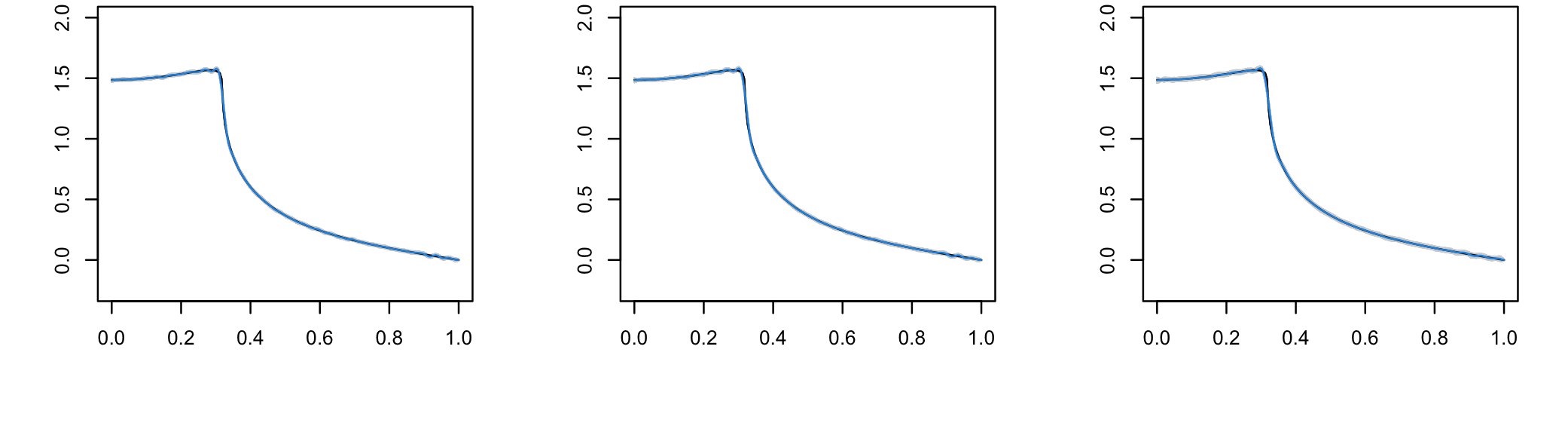}
    \caption{White noise model: true function (black), mean (blue), 95\% credible regions (grey) for the $\rho$-posterior arising for the OT prior, for $\rho=1, 0.6, 0.2$ left to right and for $n=10^3, 10^5, 10^7, 10^9, 10^{11}$ top to bottom.}
    \label{fig-rhopostSob-d2}
\end{figure}

As an illustration of `tails--adaptation', we compare in Figure \ref{fig:gpvsht} the behaviour of the HT$(\al)$ prior and a corresponding Gaussian process prior both with the same $\al=5$. We see that, as expected, adaptation does not occur in the Gaussian tail case; it does occur for the HT prior, as predicted by our theory. 

\begin{figure}[hbt!]
    \centering
    \includegraphics[width=0.77\textwidth]{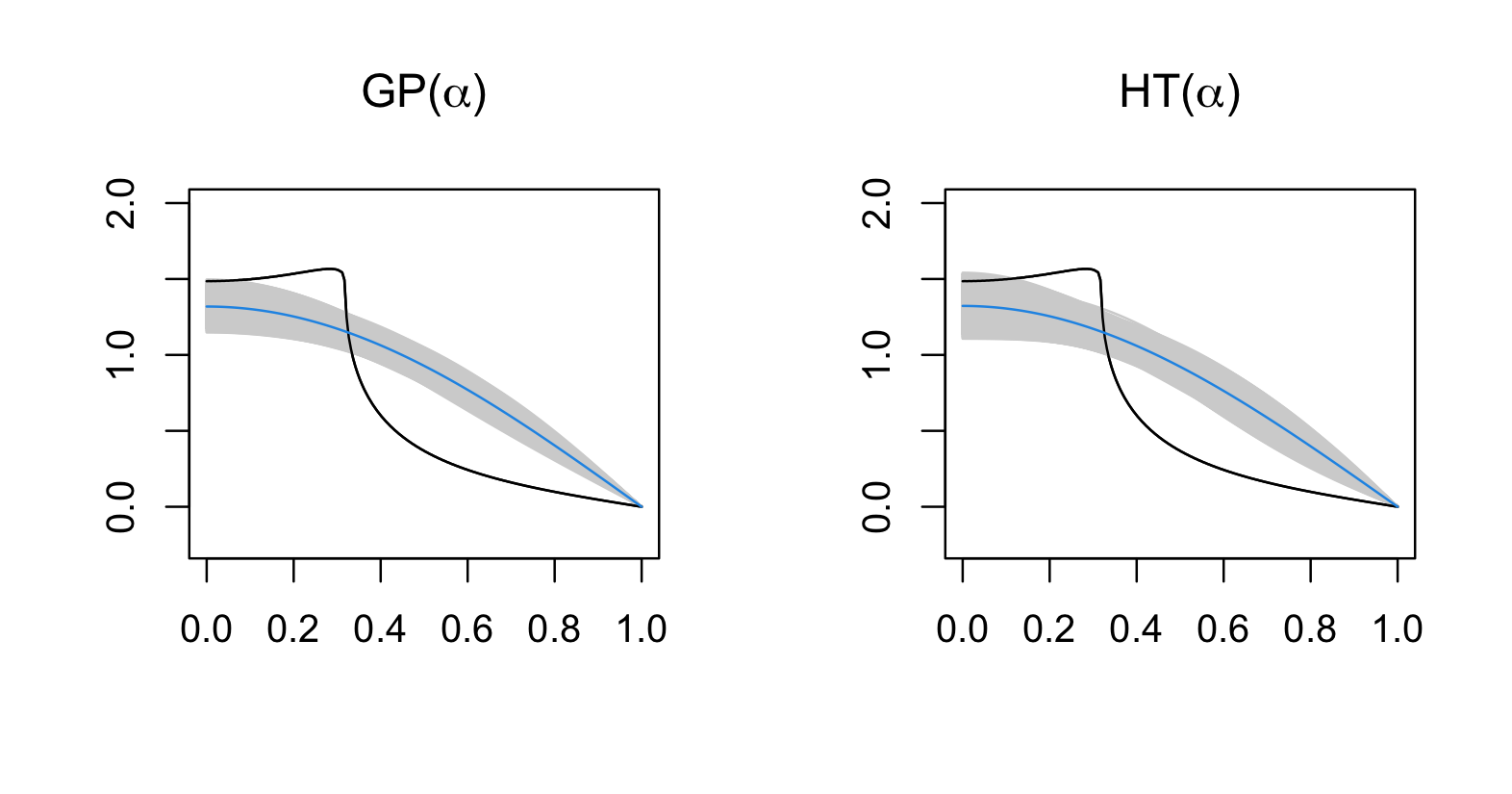}\vspace{-.55cm}
    \includegraphics[width=0.77\textwidth]{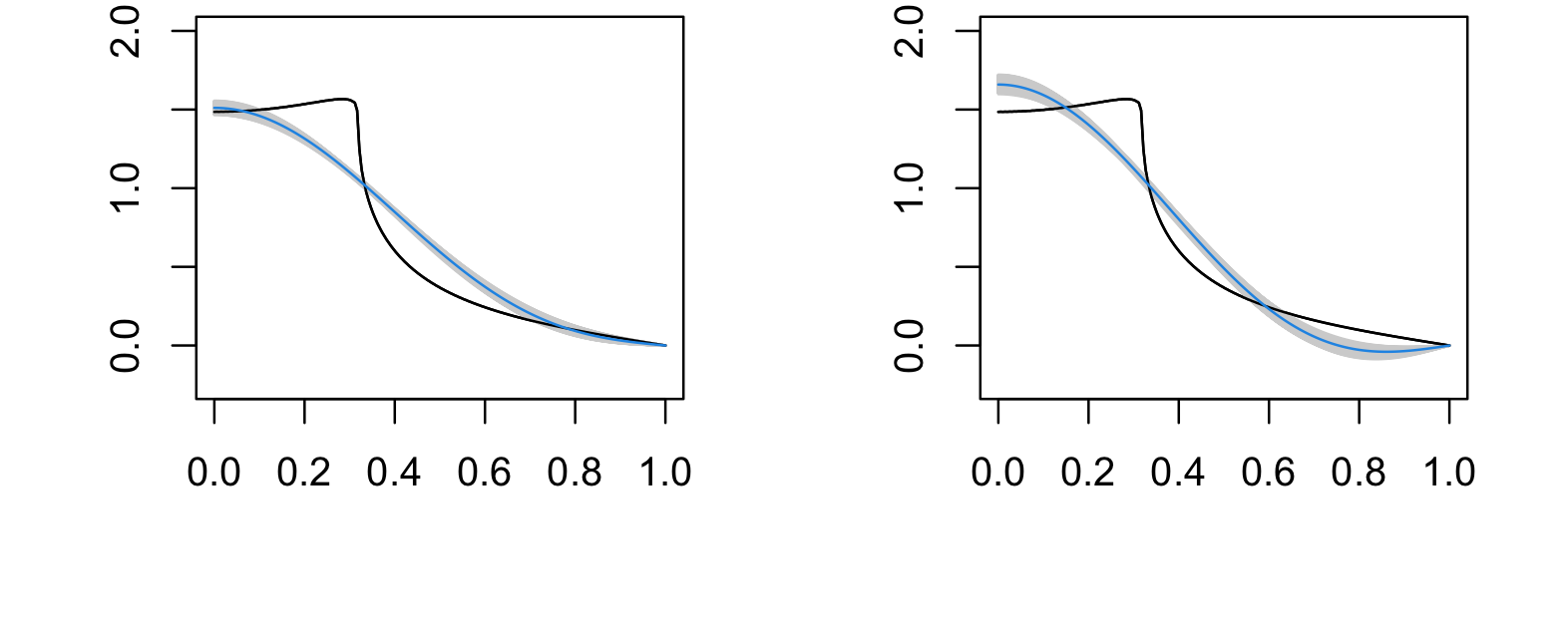}\vspace{-.55cm}
    \includegraphics[width=0.77\textwidth]{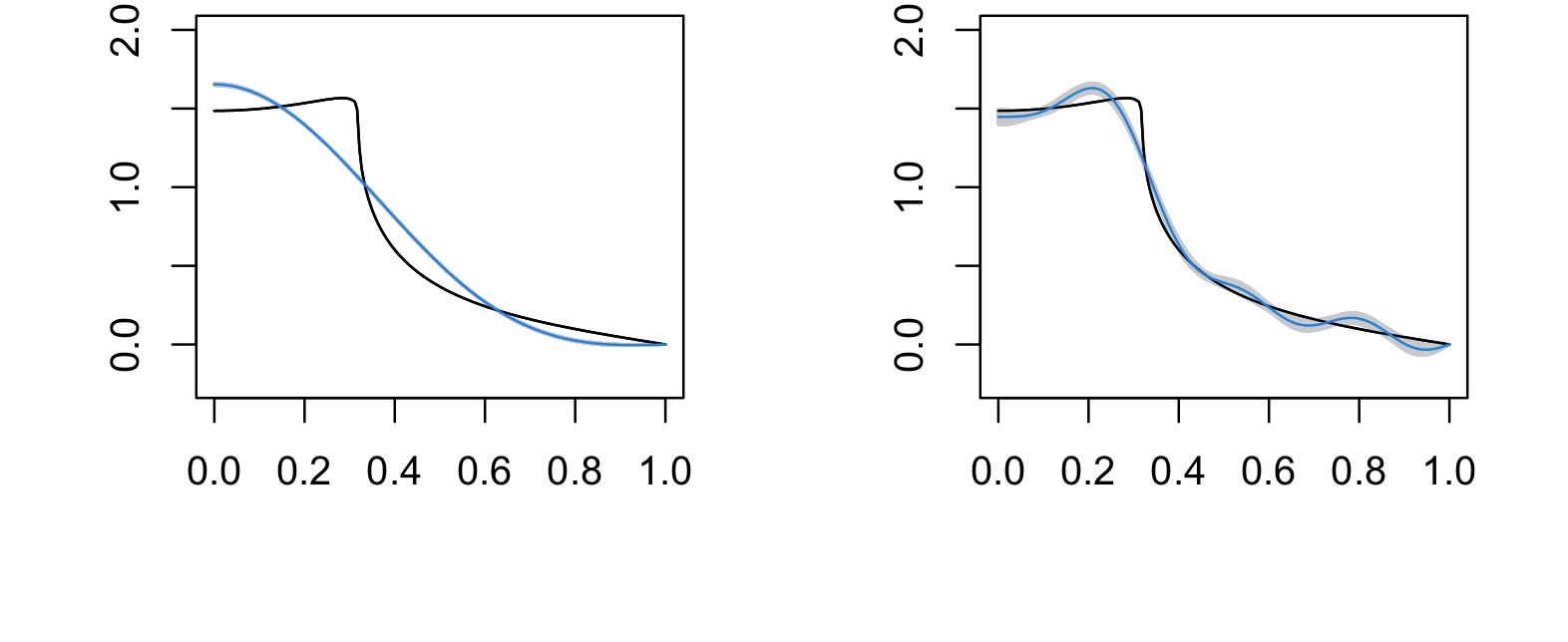}\vspace{-.55cm}
    \includegraphics[width=0.77\textwidth]{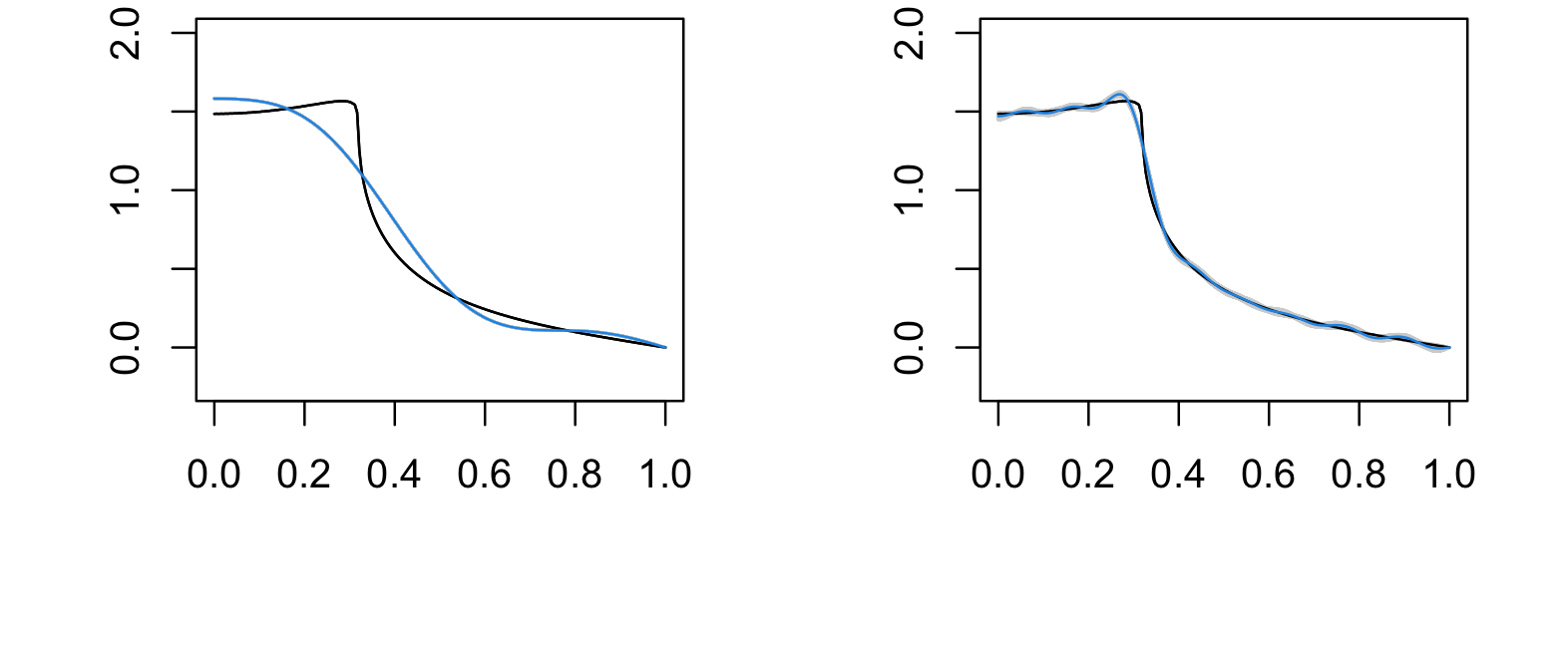}\vspace{-.55cm}
    \includegraphics[width=0.77\textwidth]{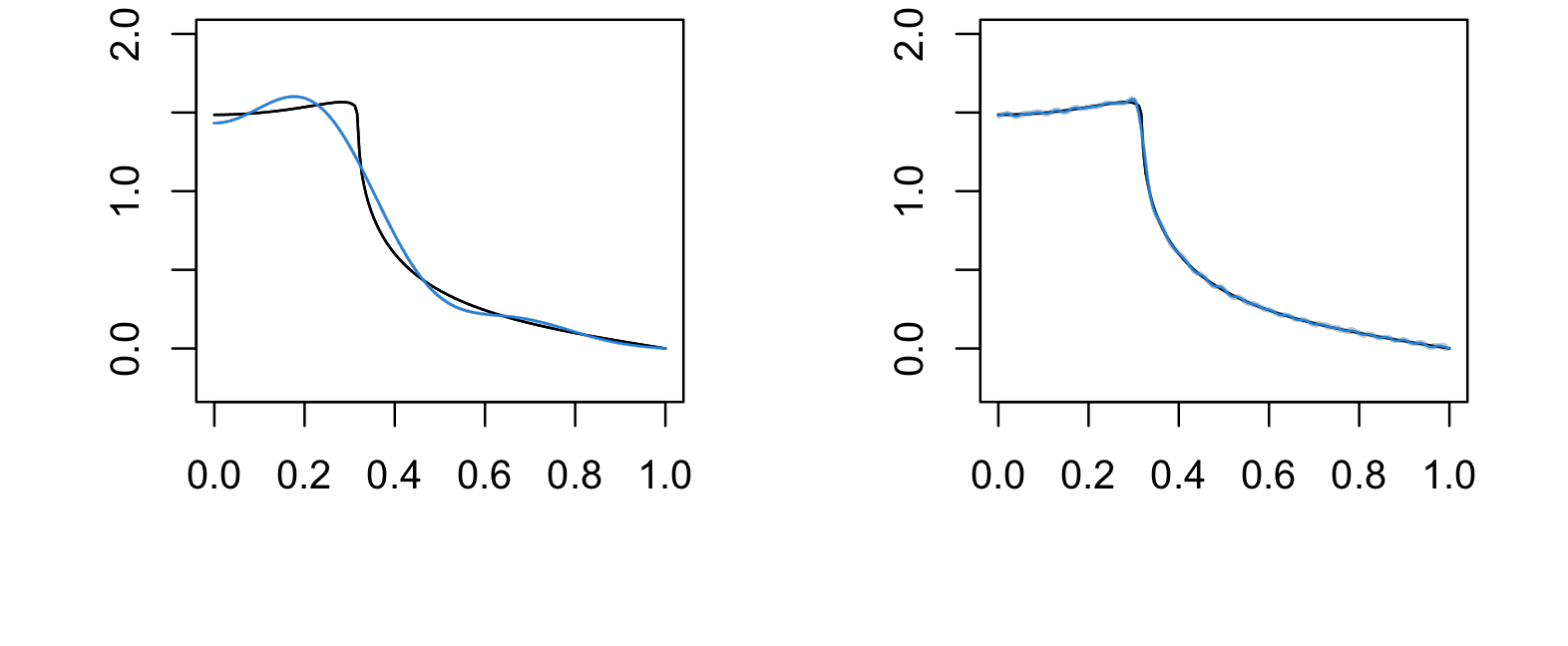}\vspace{-.55cm}
    \caption{White noise model: true function (black), posterior mean (blue), 95\% credible regions (grey), for $n=10^3, 10^5, 10^7, 10^9, 10^{11}$ top to bottom and for the $\alpha$-smooth Gaussian prior (left) and the HT$(\alpha)$ prior (right), where $\alpha=5$.}
    \label{fig:gpvsht}
\end{figure}

\subsection{White noise model with spatially inhomogeneous truth}\label{sim:spin}
We next consider function estimation in the (direct) white noise model, with four spatially inhomogeneous true functions introduced in \cite{DJ94} for testing wavelet thresholding algorithms. The four functions (Blocks, Bumps, HeaviSine and Doppler) can be seen in Figure \ref{fig-dj94}, while their formulae can be found in \cite[Table 1]{DJ94}. In all four cases, we expand the functions in the Daubechies-8 maximally symmetric wavelet basis \cite{daubechiesbook} (Symmlet-8) and add standard normal noise on each wavelet coefficient. We use 2048 coefficients and coarse level 5. Each function has been appropriately rescaled to get a signal-to-noise ratio (as captured by the ratio of the $L^2$--norms of the function to the noise) approximately equal to 7, as in \cite{DJ94}. We thus have a normal sequence model as in (\ref{normseq2}) with $n=1$. The noisy observations can be seen in Figure \ref{fig-dj94-noisy}. Analysis and synthesis of the wavelet expansions is performed in Wavelab850 \cite{wavelab}.

We consider priors on the wavelet coefficients of the form $f_{lk}=s_l\zeta_{lk}$ for i.i.d. $\zeta_{lk}$, with the following choices of the standard deviations $s_{l}$ and/or the distribution of $\zeta_{00}$:
\begin{itemize}
\item Gaussian hierarchical prior: $s_l=\tau 2^{-l(1/2+\alpha)}$ with $\tau\sim$Inv-Gamma$(1,1)$, $\alpha\sim {\rm Exp}(1)$, $\zeta_{00}$ standard normal;
\item OT prior: $s_l=2^{-l^{1+\delta}}$, with $\delta=1/2$, $\zeta_{00}$ distributed according to the standard Cauchy distribution.
\end{itemize}
For the OT prior we use $\delta=1/2$ instead of $\delta=1$ used in our analysis. As noted in Section \ref{sec:intro}, the contraction rates are identical for any $\delta>0$, however we found `the finite' $n$ behaviour to be slightly better for $\delta=0.5$ compared to $\delta=1$, so we kept this choice through the simulations.

To sample the posterior arising from the Gaussian hierarchical priors, we employ non-centered parametrizations of Metropolis-within-Gibbs samplers which update in turn $f|\alpha, \tau, X$, $\tau|f,\alpha,X$ and $\alpha|f,\tau,X$, \cite{abps14}; these samplers are initialized using prior draws. For the Cauchy prior, we use the whitened $\infty$-MALA algorithm \cite[Algorithm 3]{cdps18} initialized at the observed data. The Metropolis Adjusted Langevin Algorithm (MALA) is a Metropolis-Hastings algorithm with a proposal distribution which uses gradient information from the likelihood. The vanilla version of $\infty$-MALA is a dimension robust algorithm suitable for Gaussian priors \cite{crsw13}, while the whitened version introduced in \cite{cdps18} is suitable for non-Gaussian priors which admit a so-called prior orthogonalizing transformation $f=T(\xi)$, {where $\xi$ is a sequence of i.i.d. standard normal variables.}  and $f$ is a random draw from the non-Gaussian prior of interest. It can be shown that Cauchy series priors admit such an orthogonalizing transformation using that if $\xi$ is a standard normal random variable then $T(\xi):=\tan\big(\pi(1-2\Phi(\xi))/2\big)$, where $\Phi$ is the cumulative distribution function of the standard normal, follows the standard Cauchy distribution; for details see \cite[Section 5.2]{cdps18}. The prior orthogonalizing transformation allows to employ a function space algorithm suitable for Gaussian priors on the level of the latent Gaussian process $\xi$ (here we use $\infty$-MALA while in density estimation and binary classification below we use pCN), while we can transform $\xi$ to obtain samples for $f$ as well. 

In Figure \ref{fig-dj94-posterior}, we present posterior sample means as well as 95\% credible regions for the OT prior, computed by taking the 95\% out of the 100000 draws (after burn-in) which are closest to the mean in $L^2$-sense. In Figure \ref{fig-dj94-posterior-gaussian} we present posterior sample means for the Gaussian hierarchical prior (based on $400000$ draws, half of which were discarded as burn-in). It clealy appears that the OT prior does a better job in denoising the signals while reconstructing the spatially varying features than the Gaussian hierarchical prior. The Gaussian hierarchical prior was expected to perform poorly in this setting, since the contraction rates achieved by Gaussian priors in the small noise limit over Besov spaces of spatially inhomogeneous functions are known to be suboptimal \cite{aw21}. In Table \ref{tab:spinh} we show $L^2$-errors of the posterior mean, which confirm the better performance of the OT prior. In the table, we also include errors for a Laplace hierarchical prior, with the same hyper-priors as the Gaussian mixture considered here, taken from the upcoming PhD thesis of Aimilia Savva. The Laplace prior is known to be adaptive over spatially inhomogeneous Besov spaces \cite{as22}, and in the present setting performs in-between the OT prior and the Gaussian hierarchical prior.
Returning to the OT prior, the posterior means are comparable to the estimates obtained via the RiskShrink wavelet shrinkage method in \cite{DJ94}. 

\begin{table}[ht]
    \centering
  \caption{White noise model with spatially inhomogeneous truths: $L^2$-errors of posterior means.}
  \label{tab:spinh}
    \begin{tabular}{llll}
        \toprule
        & \textbf{\quad\hspace{-0.05cm} {OT} \quad} & \textbf{Gaussian hierarchical} & \textbf{Laplace hierarchical} \\
    \toprule
   Blocks & \quad 0.50 & \quad\quad\quad\quad 0.61 & \quad\quad\quad\quad 0.53 \\
   Bumps  & \quad 0.54 & \quad\quad\quad\quad 0.70 &  \quad\quad\quad\quad 0.58 \\
   HeaviSine  & \quad 0.21 & \quad\quad\quad\quad 0.26 & \quad\quad\quad\quad 0.26 \\
   Doppler  & \quad 0.33 & \quad\quad\quad\quad 0.49 & \quad\quad\quad\quad 0.40 \\

        \bottomrule
    \end{tabular}
\end{table}

\begin{figure}[htbp]
    \centering
    \includegraphics[width=0.45\textwidth]{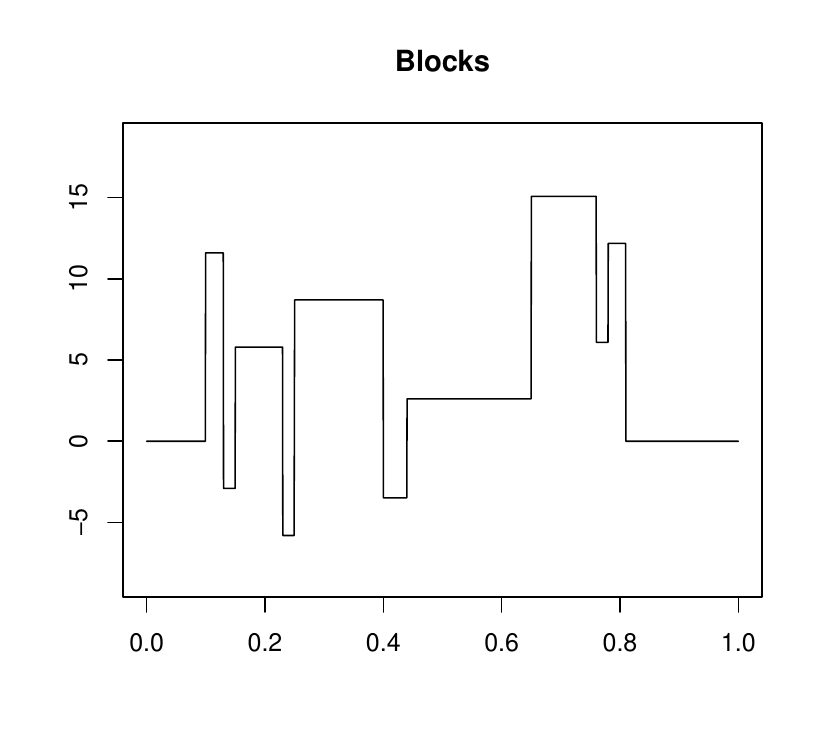}
    \includegraphics[width=0.45\textwidth]{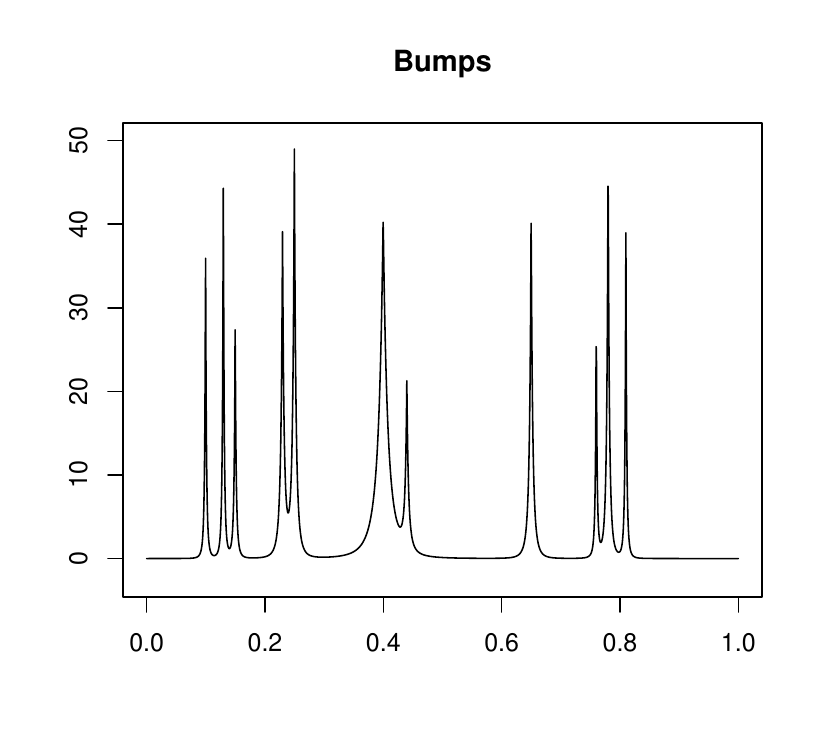}\vspace{-.8cm}
     \includegraphics[width=0.45\textwidth]{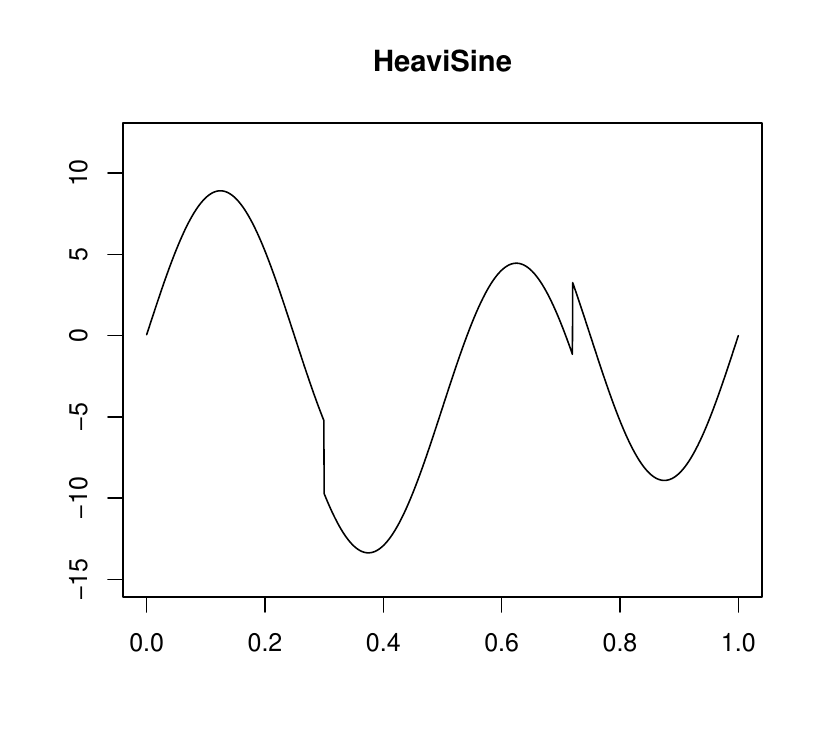}
 \includegraphics[width=0.45\textwidth]{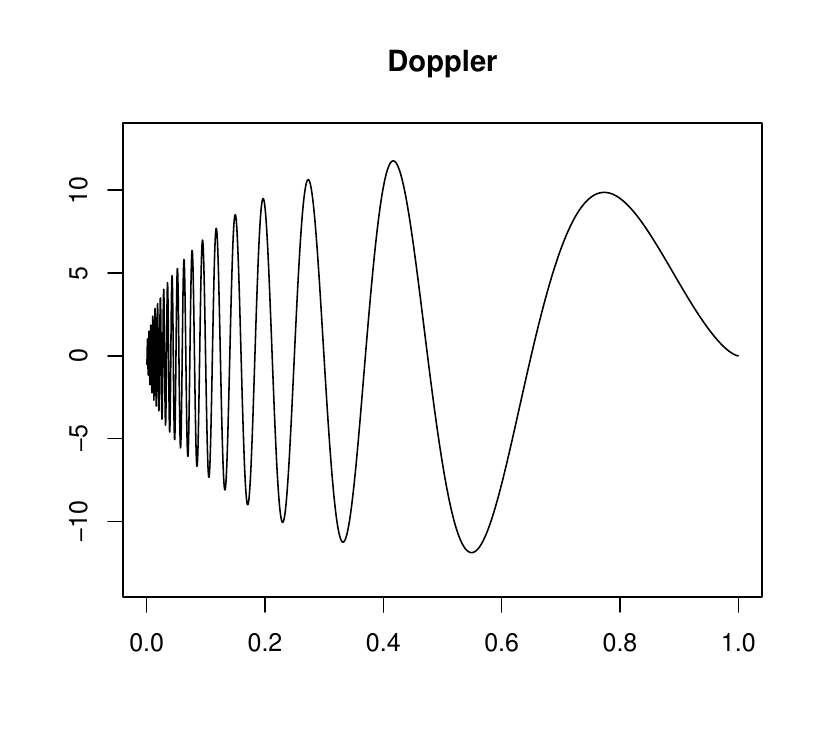}\vspace{-.5cm}
    \caption{True functions.}
    \label{fig-dj94}

\bigskip

    \includegraphics[width=0.45\textwidth]{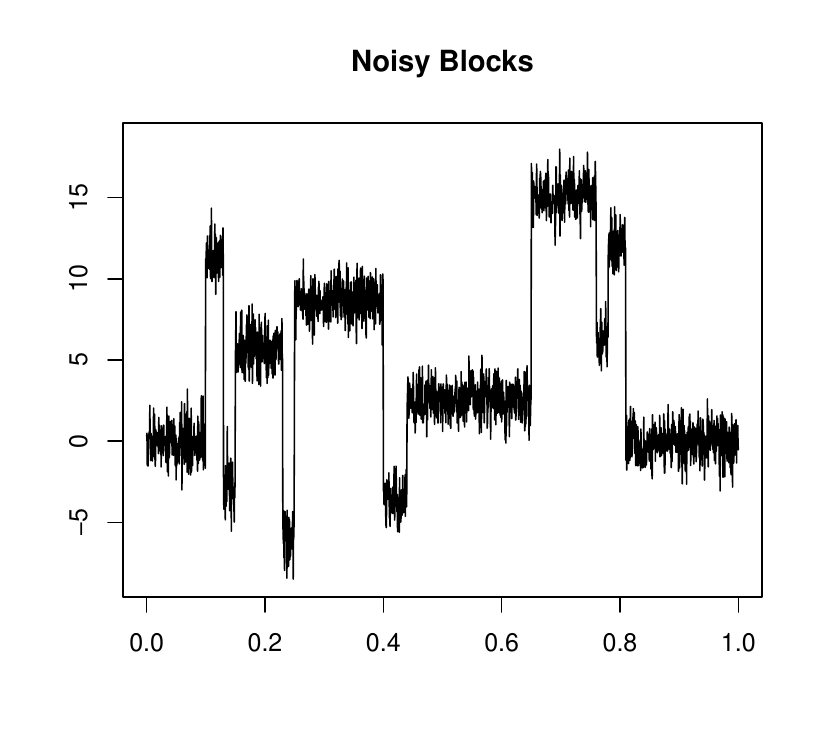}
    \includegraphics[width=0.45\textwidth]{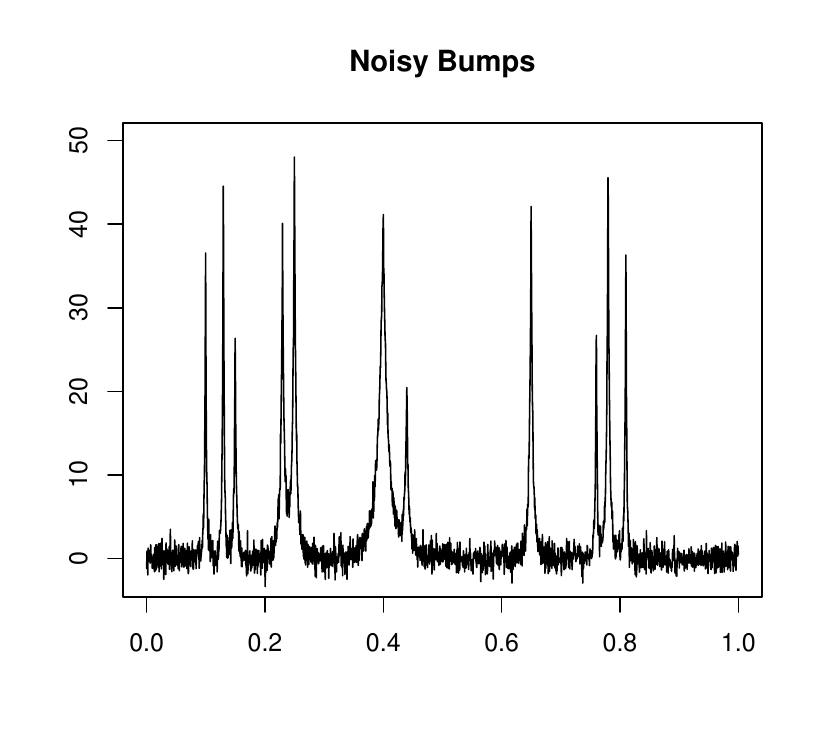}\vspace{-.8cm}
     \includegraphics[width=0.45\textwidth]{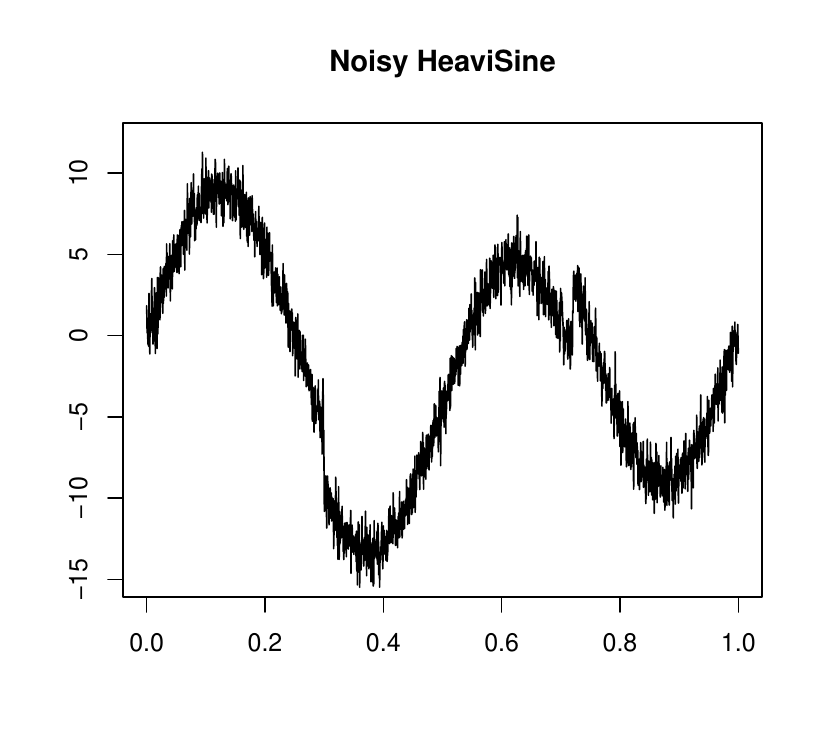}
 \includegraphics[width=0.45\textwidth]{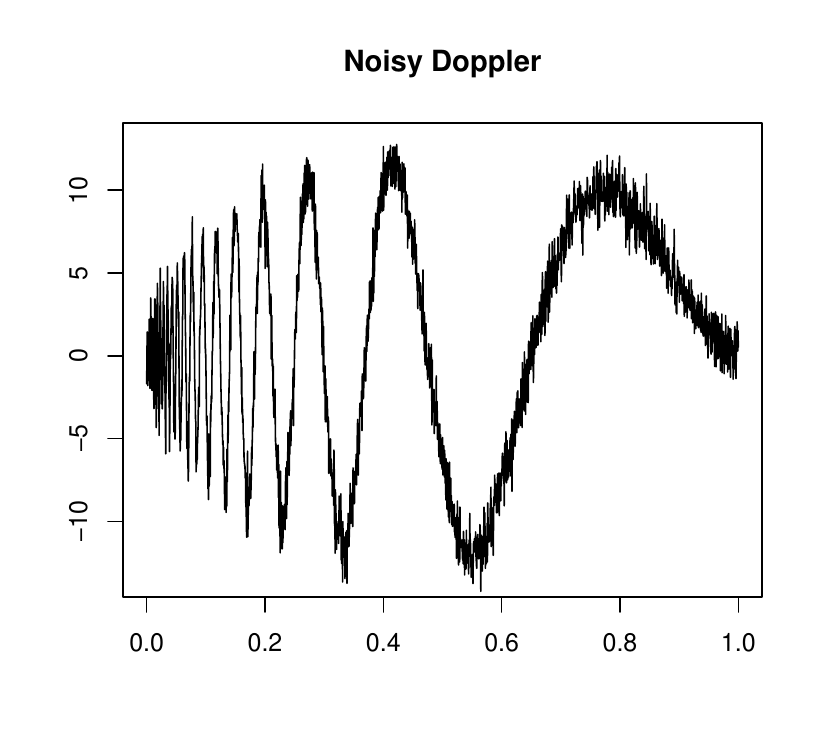}\vspace{-.5cm}
    \caption{Noisy observation, signal to white noise model approximately 7 in all four cases.}
    \label{fig-dj94-noisy}
\end{figure}

\begin{figure}[htbp]
\centering
    \includegraphics[width=0.45\textwidth]{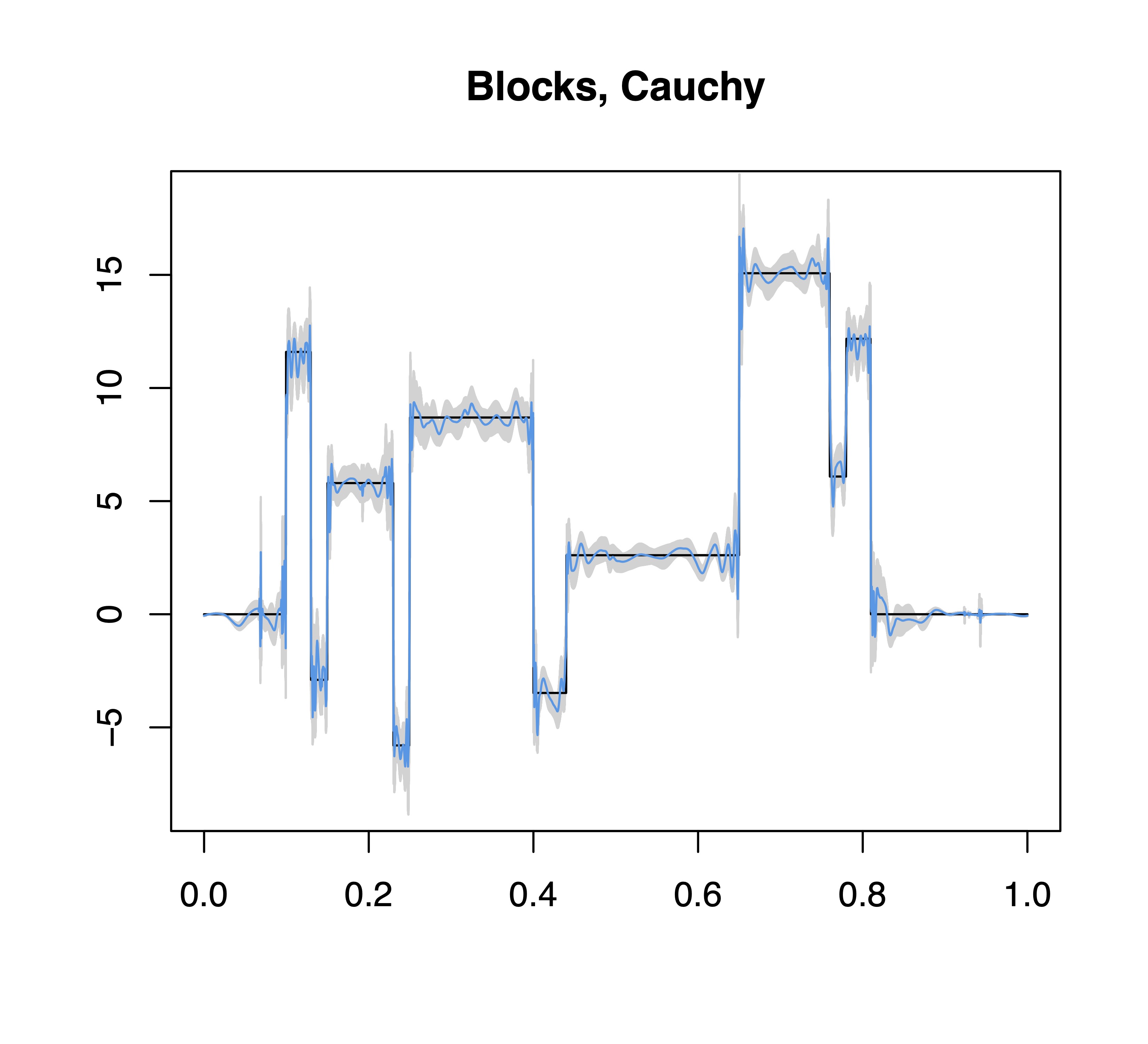}
    \includegraphics[width=0.45\textwidth]{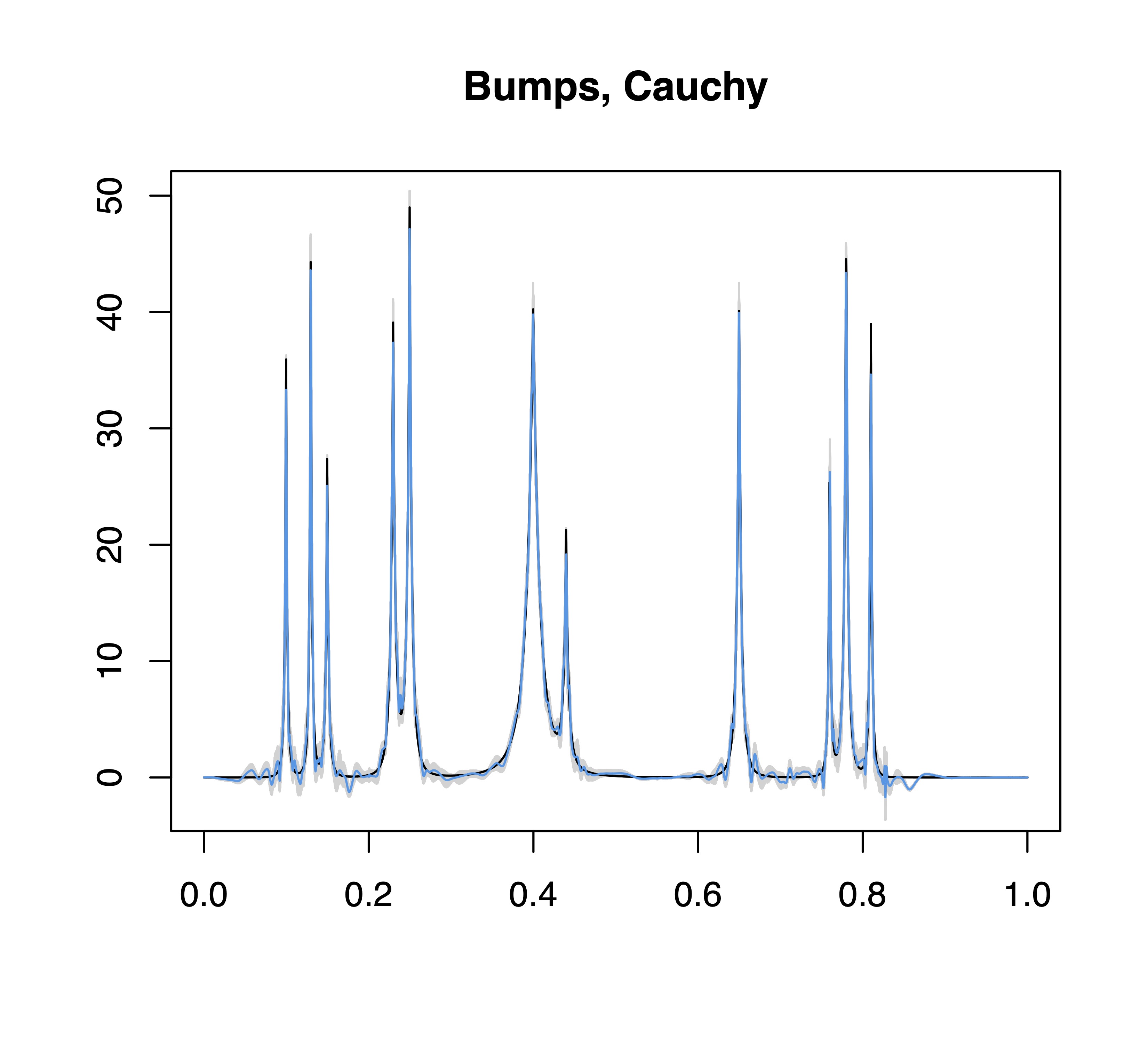}\vspace{-.8cm}
     \includegraphics[width=0.45\textwidth]{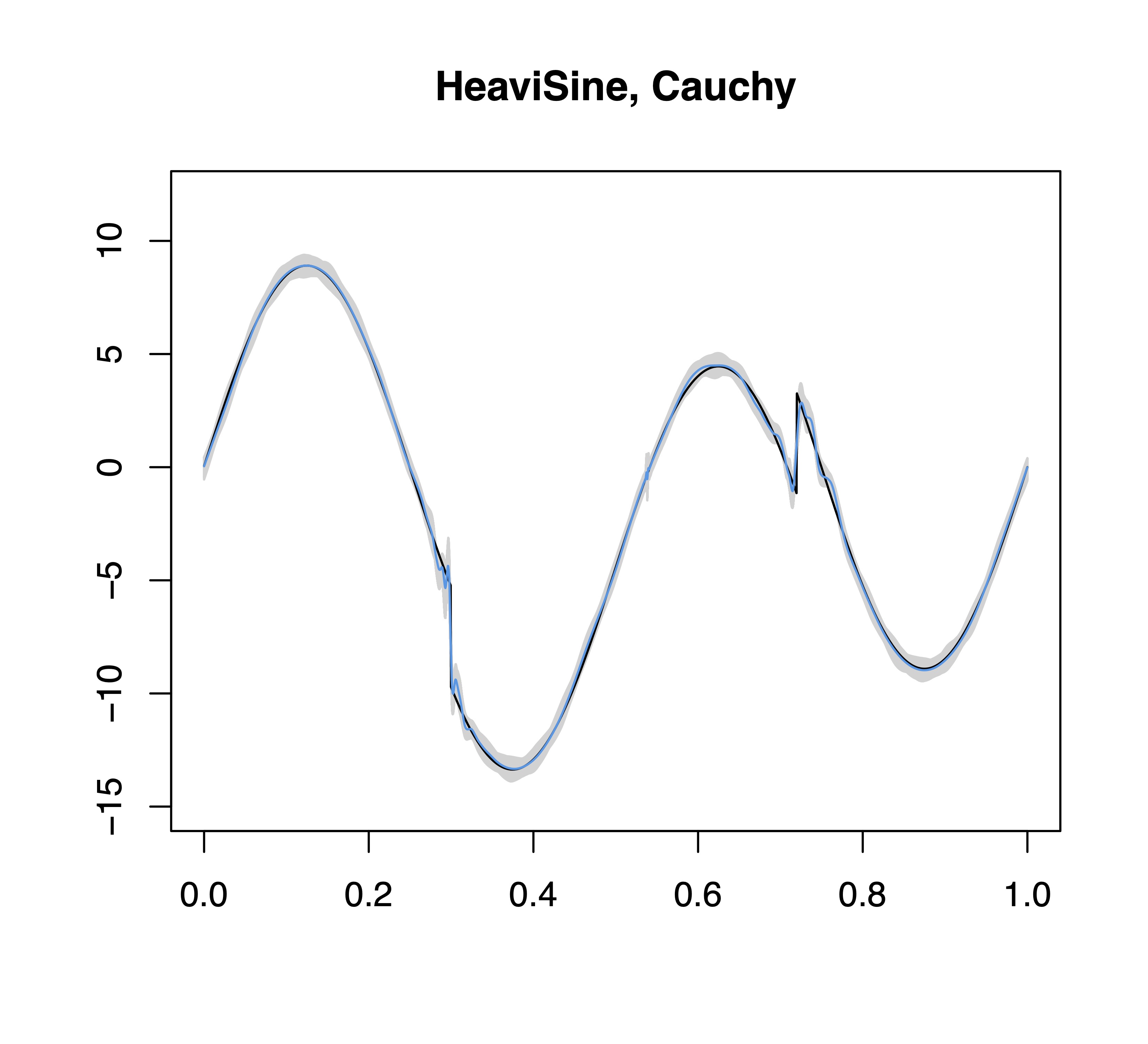}
 \includegraphics[width=0.45\textwidth]{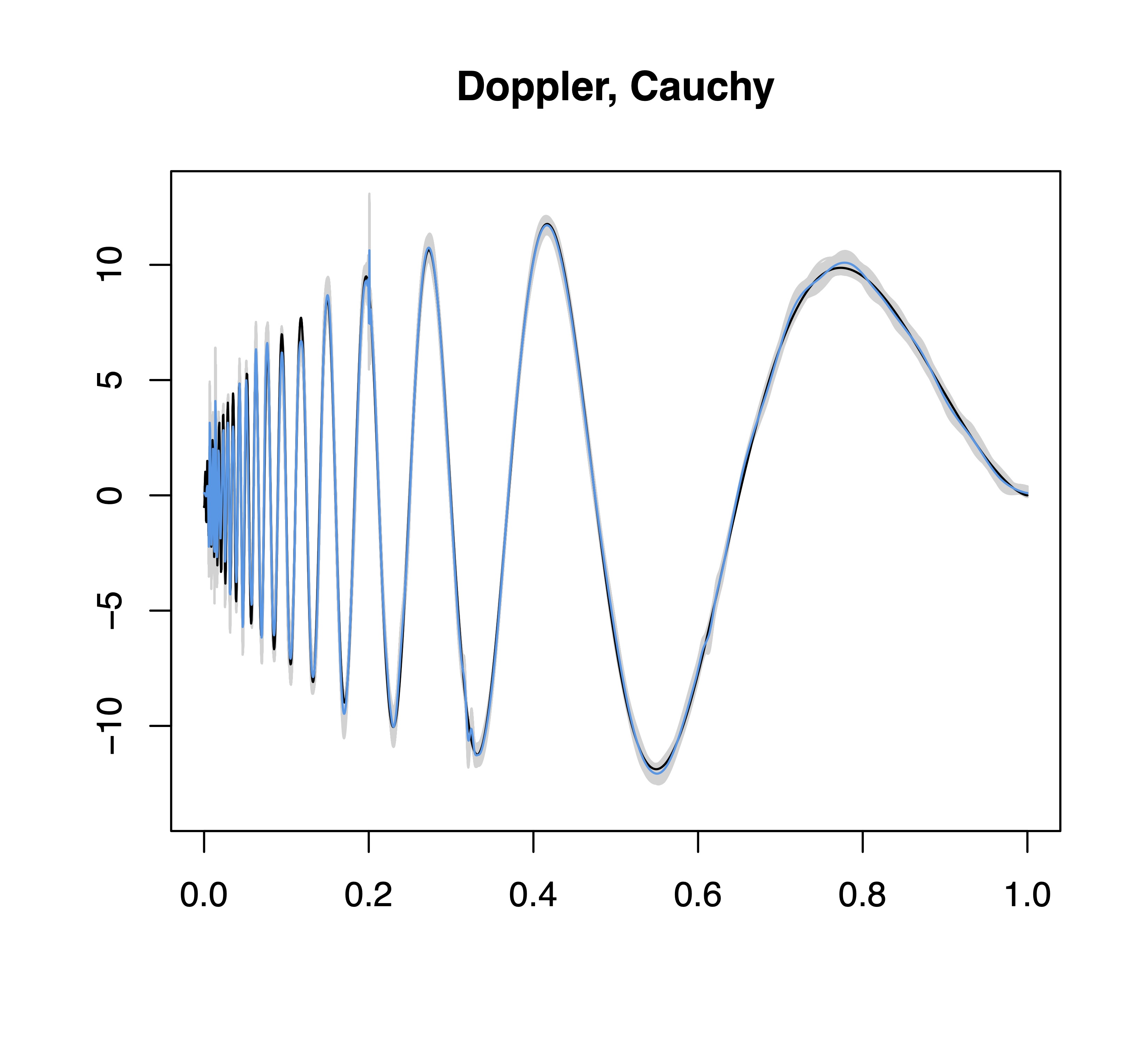}\vspace{-.5cm}
    \caption{Posteriors under Cauchy oversmoothing prior.}
    \label{fig-dj94-posterior}
    \bigskip
      \includegraphics[width=0.45\textwidth]{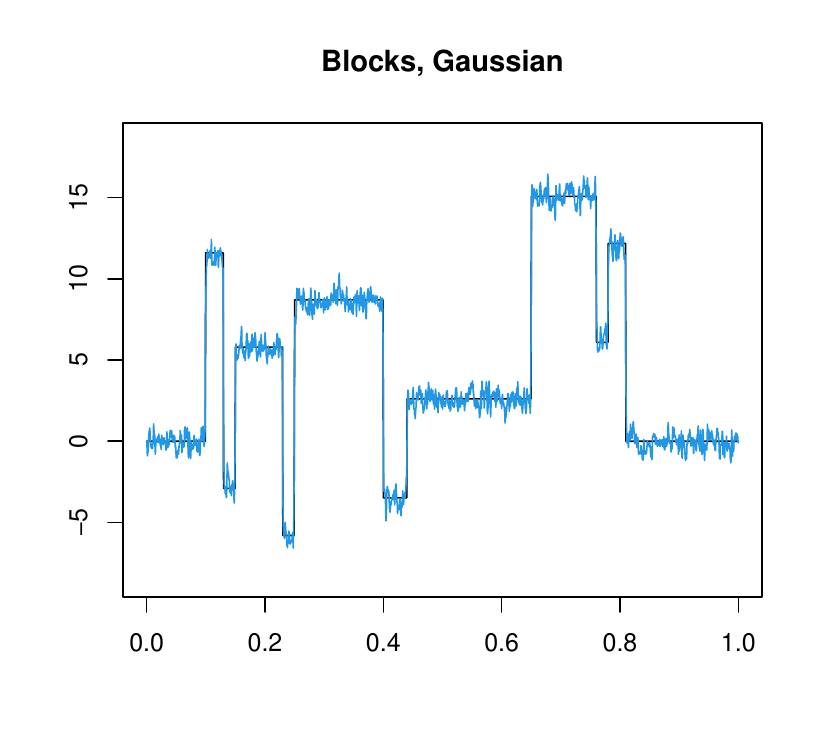}
    \includegraphics[width=0.45\textwidth]{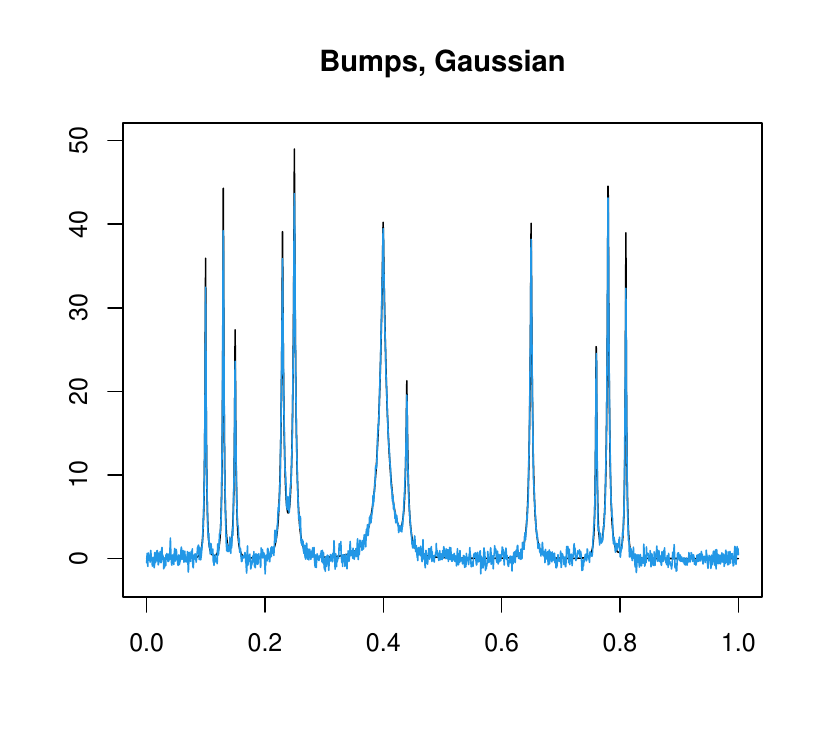}\vspace{-.8cm}
     \includegraphics[width=0.45\textwidth]{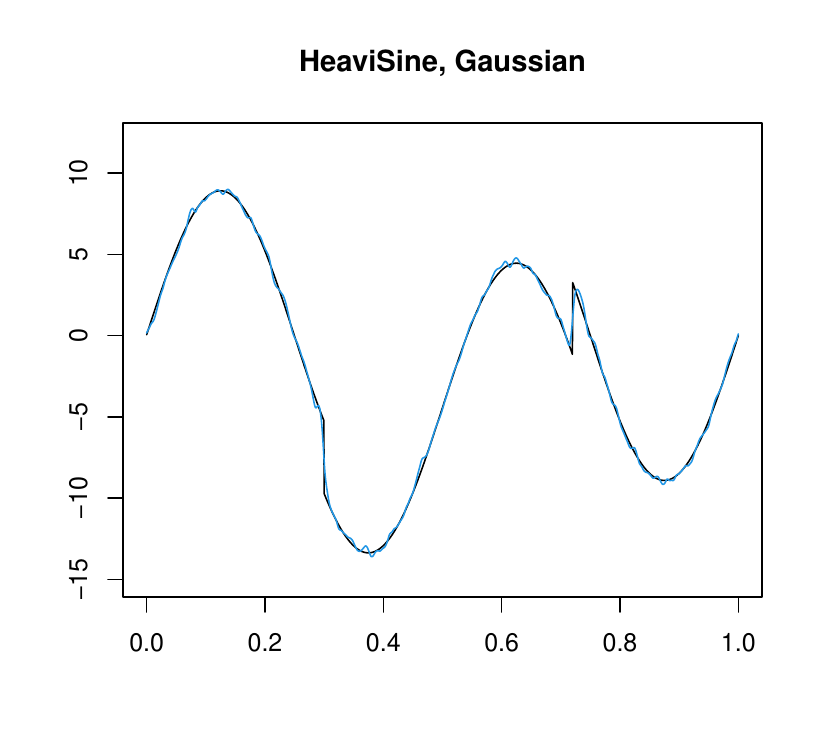}
 \includegraphics[width=0.45\textwidth]{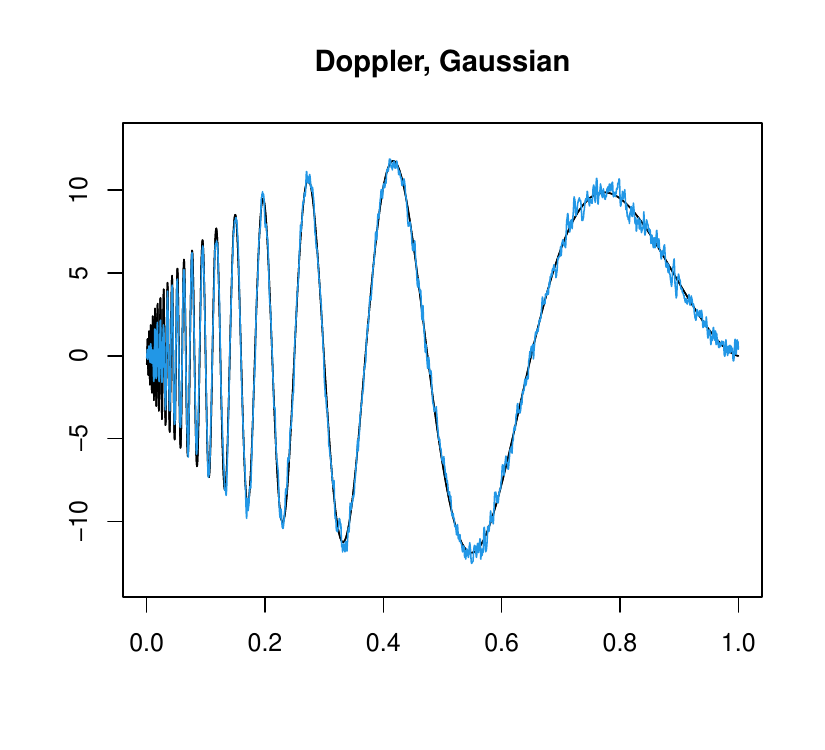}\vspace{-.5cm}
    \caption{Posterior means under Gaussian hierarchical prior.}
    \label{fig-dj94-posterior-gaussian}
\end{figure}

\subsection{Density estimation with H\"older-regular truth}\label{de:sim} 
We consider non-parametric estimation of a probability density function $p$ relative to the uniform measure on the unit interval, based on a sample of observations $X_1, \dots, X_n|p\stackrel{iid}{\sim}p$. We construct a prior $\Pi$ for such densities, by exponentiating and normalizing a random function $f(\cdot):[0,1]\to\RR$ drawn from a wavelet prior
as in  (\ref{priord}), see Subsection \ref{ssec:de}. We use the Daubechies-8 wavelet basis \cite{daubechiesbook}, while as underlying true density we use $p_0=e^{f_0}/\int_0^1e^{f_0}$ for $f_0$ with Daubechies-8 wavelet coefficients given as 
\[f_{0,lk}=4\cos^3(k)2^{-(5/2)l}.\] Note that since the Daubechies-8 wavelets have H\"older regularity higher than 2 \cite{daubechiesbook}, we have that $f_0\in \mathcal{C}^2[0,1]$.

We use three different series priors on $f$, corresponding to four different choices of $s_l$ and the distribution of $\zeta_{lk}$ in (\ref{priord}):
\begin{itemize}
\item Gaussian oversmoothing prior: $s_l=2^{-l(1/2+\alpha)}$ with $\alpha=5$, where $\zeta_{lk}$ standard normal;
\item HT$(\alpha)$ prior: $s_l=2^{-l(1/2+\alpha)}$ with $\alpha=5$, where $\zeta_{lk}$ standard Cauchy-distributed;
\item OT prior: $s_l=2^{-l^2},$ with $\zeta_{lk}$ standard Cauchy-distributed.
\end{itemize}

We truncate up to $L=10$ which (given the smoothness of the considered priors and the truth) suffices for the truncation error being of lower order compared to the estimation error. For sampling the posterior we employ the preconditioned Crank-Nicholson (pCN) algorithm, introduced in \cite{brsv08} and popularized in \cite{crsw13}, which is a derivative-free Metropolis-Hastings algorithm robust with respect to dimension (truncation level). The vanilla version of pCN is suitable for Gaussian priors, while, similarly to the previous subsection, in the Cauchy case we use the whitened version introduced in \cite{cdps18}, which is suitable for non-Gaussian priors admitting a so-called prior orthogonalizing transformation $f=T(\xi)$, {where $\xi$ is a sequence of i.i.d. standard normal variables}  and $f$ is a random draw from the non-Gaussian prior of interest, see Algorithm 2 in \cite{cdps18}. For all considered priors, we initialize the Markov chains using draws from the prior. For analyzing and synthesizing the Daubechies-8 wavelet expansions, we again employ Wavelab850 \cite{wavelab}.

In Figure \ref{fig-postSob-de-duplicate} (which is a duplicate of Figure \ref{fig-postSob-de} in the main article, included here for the reader's convenience) we present posterior sample means as well as 95\% credible regions, computed by taking the 95\% out of the 25000 draws (after burn-in) which are closest to the mean in $L^1$-sense, for the three considered priors. 
It appears that the oversmoothing Gaussian prior performs very poorly, both in terms of the posterior mean and the uncertainty quantification. The two Cauchy priors perform very well, with posterior means which {are largely competitive with} the kernel density estimates. Similarly to our simulations in the inverse regression setting, the HT$(\alpha)$ prior is again slightly overconfident for moderate sample sizes. As a benchmark for our results, in Figure \ref{fig-kde} we show kernel density estimates of $p_0$ using a Gaussian kernel with bandwidth selection based on Silverman's rule of thumb, for various observation sample sizes. Finally, in Figure \ref{fig-rhopostSob-de-CD2} we present $\rho$-posteriors for the OT prior with $\rho=1, 0.6, 0.2$;  
their performance is similar to the regular posteriors, albeit with more variability.



\begin{figure}[htbp]
    \centering
    \includegraphics[width=0.92\textwidth]{./plots/de/title.png} \quad

    \includegraphics[width=0.3\textwidth]{./plots/de/g-n2-L1.pdf}\;
        \includegraphics[width=0.3\textwidth]{./plots/de/cd1-n2-L1.pdf}\;
    \includegraphics[width=0.3\textwidth]{./plots/de/cd2-n2-L1.pdf}\vspace{0.4cm}

    \includegraphics[width=0.3\textwidth]{./plots/de/g-n4-L1.pdf}\;
        \includegraphics[width=0.3\textwidth]{./plots/de/cd1-n4-L1.pdf}\;
    \includegraphics[width=0.3\textwidth]{./plots/de/cd2-n4-L1.pdf}\vspace{0.4cm}
    
        \includegraphics[width=0.3\textwidth]{./plots/de/g-n6-L1.pdf}\;
        \includegraphics[width=0.3\textwidth]{./plots/de/cd1-n6-L1.pdf}\;
    \includegraphics[width=0.3\textwidth]{./plots/de/cd2-n6-L1.pdf}

    \caption{Density estimation: true density (black dashed), posterior mean (blue), 95\% credible regions (grey), for $n=10^2, 10^4, 10^6$  top to bottom and for the three considered priors left to right.}
    \label{fig-postSob-de-duplicate}
        
    \bigskip
    
        \includegraphics[width=0.3\textwidth]{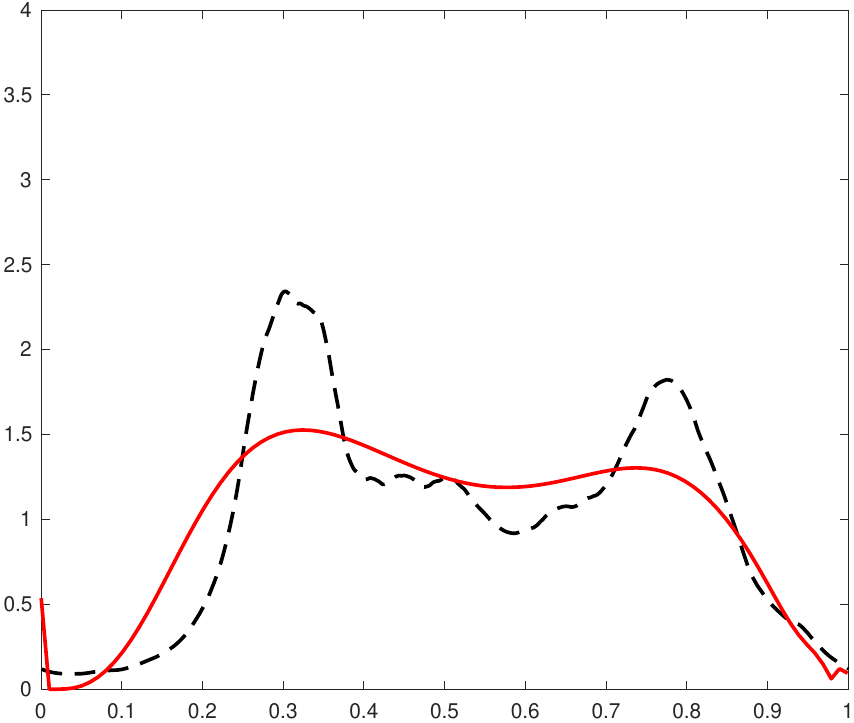}\;
    \includegraphics[width=0.3\textwidth]{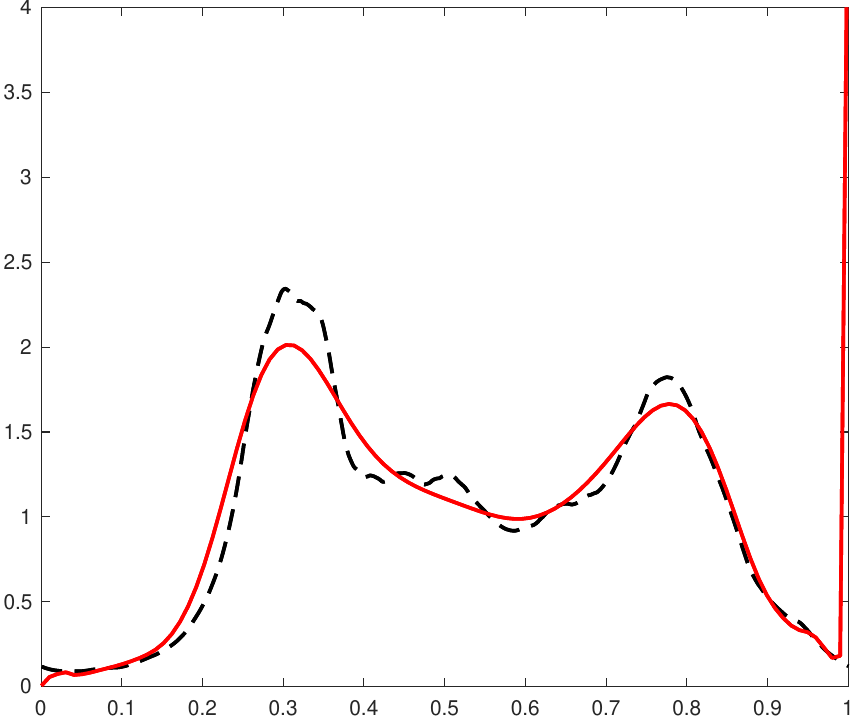}\;
    \includegraphics[width=0.3\textwidth]{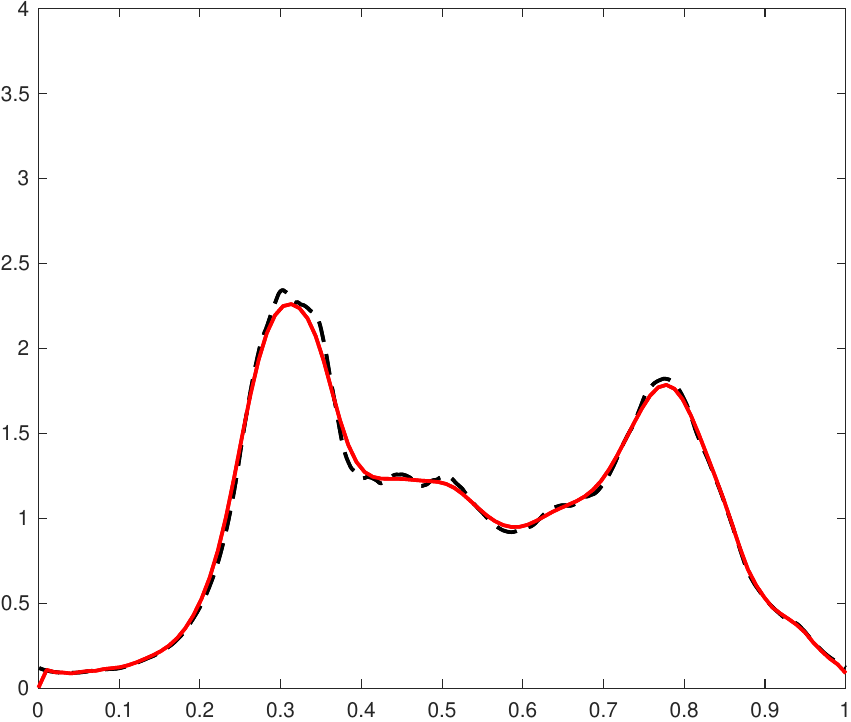}\vspace{.4cm}
    \caption{True density (blacked dashed) and kernel density estimator with Gaussian kernel (red) for $n=10^2, 10^4, 10^6$ left to right.}
    \label{fig-kde}
\end{figure}

\begin{figure}[htbp]
    \centering
     \includegraphics[width=0.935\textwidth]{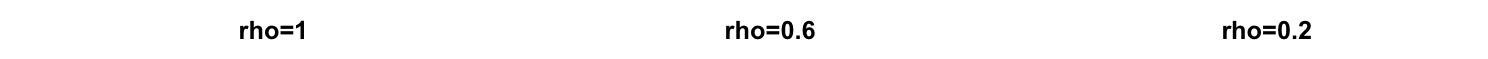} \quad\,

    \includegraphics[width=0.3\textwidth]{./plots/de/cd2-n2-L1.pdf}\;
    \includegraphics[width=0.3\textwidth]{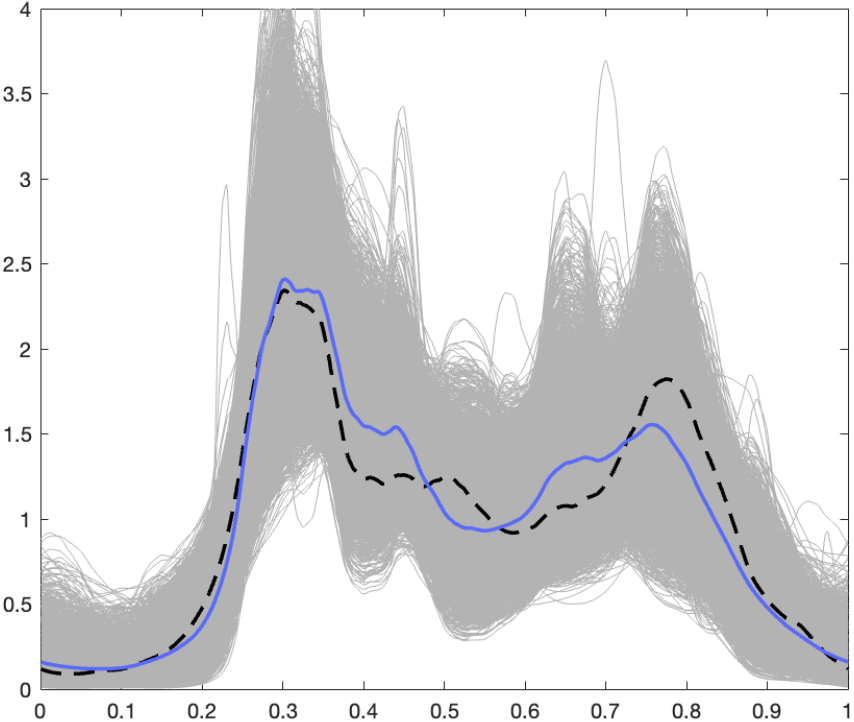}\;
    \includegraphics[width=0.3\textwidth]{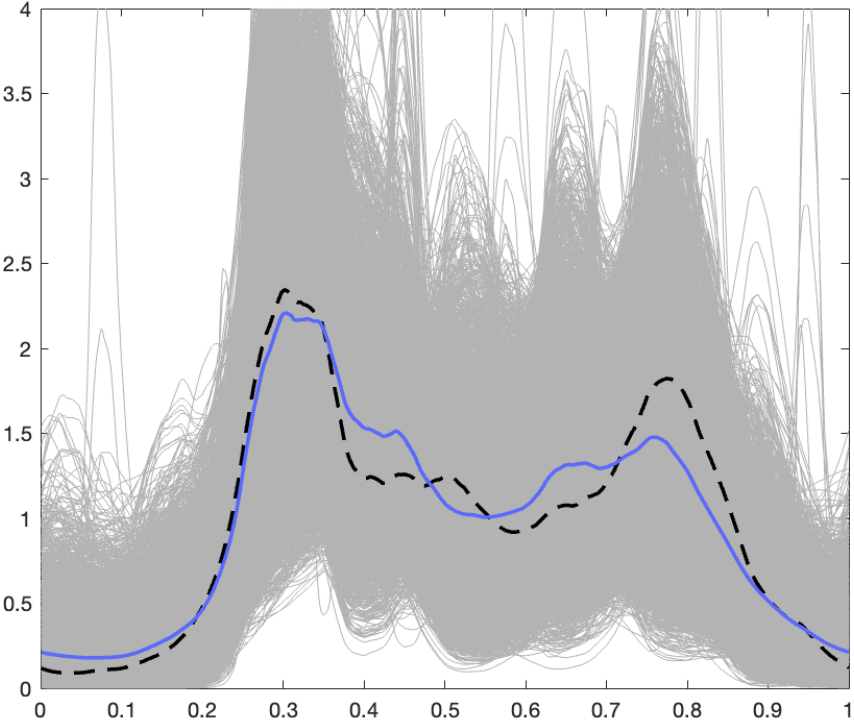}\\\vspace{.4cm}

    \includegraphics[width=0.3\textwidth]{./plots/de/cd2-n4-L1.pdf}\;
    \includegraphics[width=0.3\textwidth]{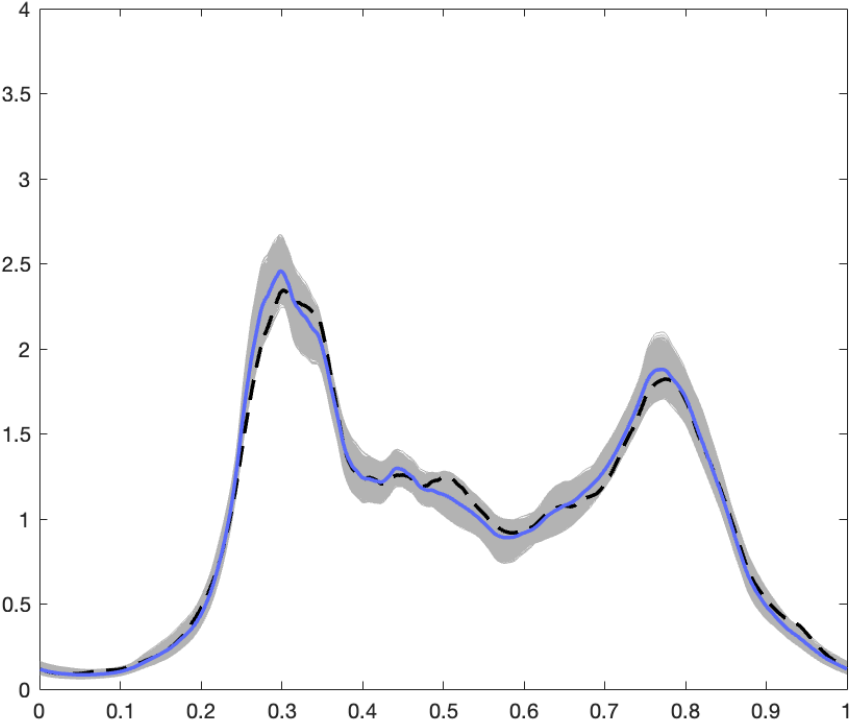}\;
    \includegraphics[width=0.3\textwidth]{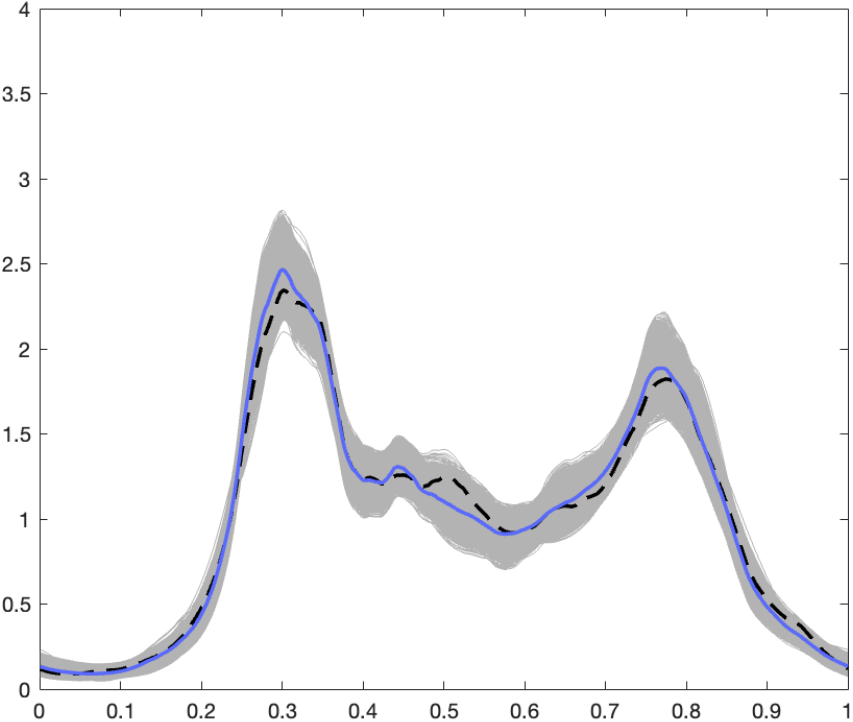}\\\vspace{.4cm}

    \includegraphics[width=0.3\textwidth]{./plots/de/cd2-n6-L1.pdf}\;
    \includegraphics[width=0.3\textwidth]{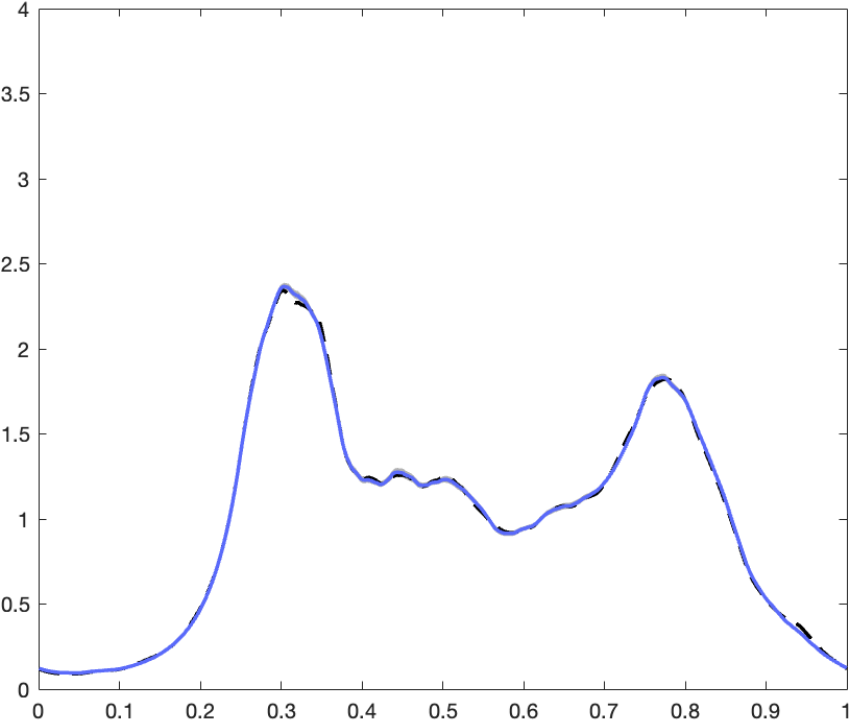}\;
    \includegraphics[width=0.3\textwidth]{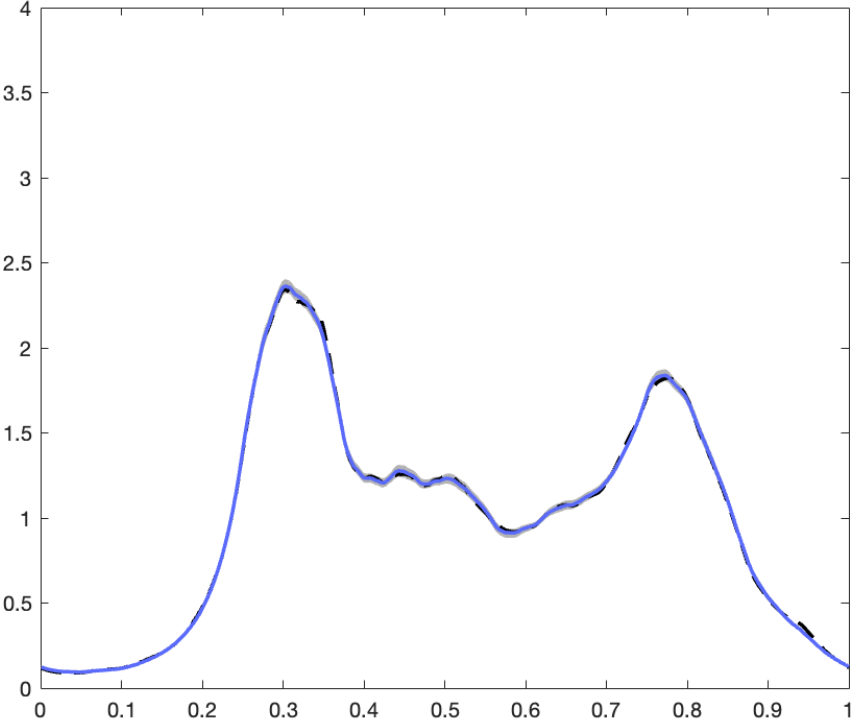}\vspace{.4cm}
    
    \caption{Density estimation: true density (black dashed), mean (blue), 95\% credible regions according  (grey) for the $\rho$-posterior arising from the Cauchy prior with $s_l=2^{-l^2}$, for $n=10^2, 10^4, 10^6$ top to bottom and $\rho=1, 0.6, 0.2$ left to right.}
    \label{fig-rhopostSob-de-CD2}
\end{figure}

\subsection{Binary classification with Sobolev-regular truth}\label{bin:sim} 
 We consider non-parametric binary classification based on observations $(X_1, Y_1),\dots,(X_n, Y_n)$, where the dependent variable $Y$ is binary and the predictor $X$ is uniformly distrbuted in the unit interval. We construct a prior $\Pi$ for the binary regression function $h_0(x)=P(Y=1|X=x)$, using the logistic link function $\Lambda(u)=1/(1+e^{-u})$ applied to a random function $f(\cdot):[0,1]\to\RR$ drawn from a wavelet prior
as in  (\ref{priord}), see Subsection \ref{sec:classif}. We use the Daubechies-8 wavelet basis \cite{daubechiesbook}, while as underlying true regression function we use $h_0=\Lambda(f_0)$ for $f_0$ the same function as in our simulations in the density estimation setting in the previous subsection. In particular, $f_0$ can be thought of as having Sobolev regularity (almost) $\beta=2$.
The true regression function and $n=10^2$ observations can be seen in Figure \ref{fig-bin}.

We consider the same three series priors on $f$ as in the previous subsection, truncated again up to $L=10$ and for sampling the resulting posteriors we again use the (whitened) pCN algorithm initialized using prior draws.

In Figure \ref{fig-postSob-bin} we present posterior sample means as well as 95\% credible regions, computed by taking the 95\% out of the 25000 draws (after burn-in) which are closest to the mean in $L^1$-sense, for the three considered priors. 
In Figure \ref{fig-rhopostSob-bin-CD2} we present $\rho$-posteriors for the OT prior with $\rho=1, 0.6, 0.2$. The conclusions are similar to the density estimation setting in the previous subsection.

\begin{figure}[htbp]
    \centering
    \includegraphics[width=0.5\textwidth]{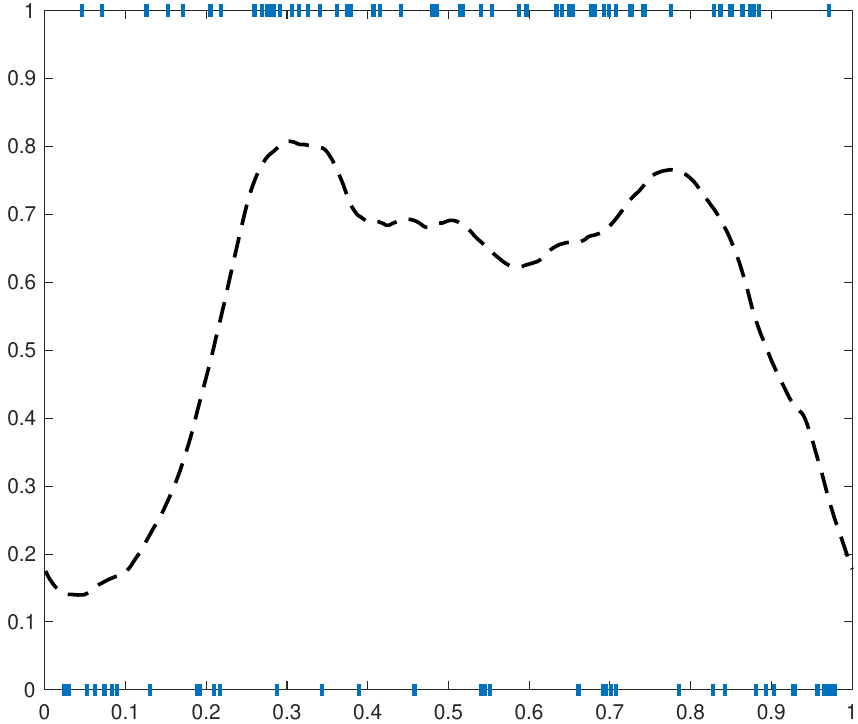}\;\vspace{0.4cm}
    \caption{True binary regression function (blacked dashed) and $n=10^2$ observations (blue vertical line segments).}
    \label{fig-bin}
    
\bigskip

    \includegraphics[width=0.92\textwidth]{./plots/de/title.png} \quad

    \includegraphics[width=0.3\textwidth]{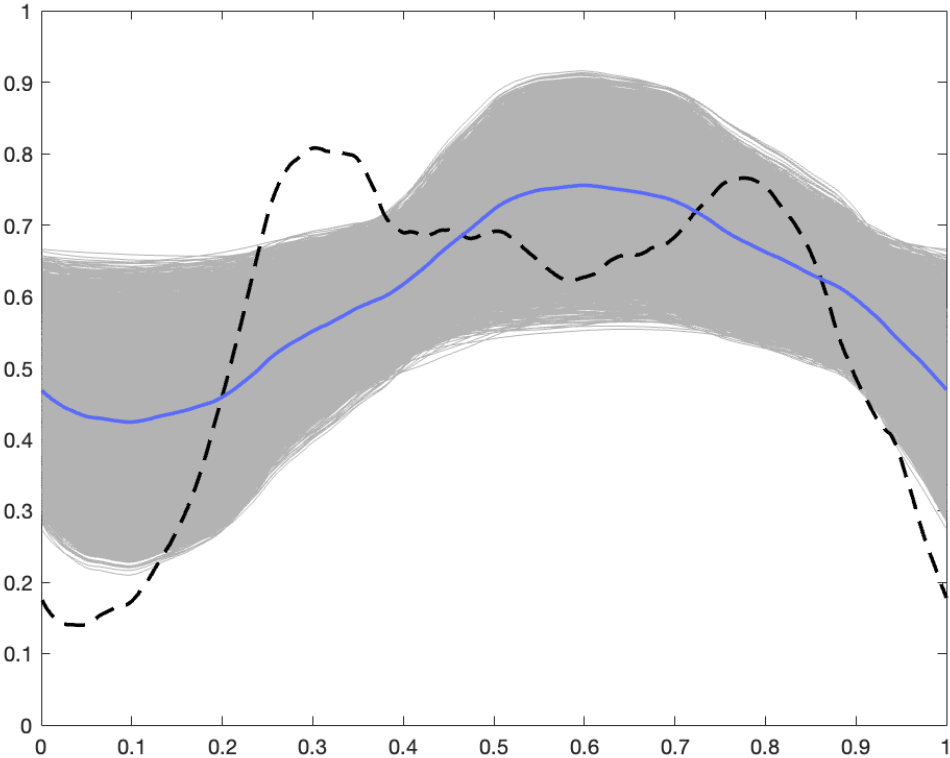}\;
        \includegraphics[width=0.3\textwidth]{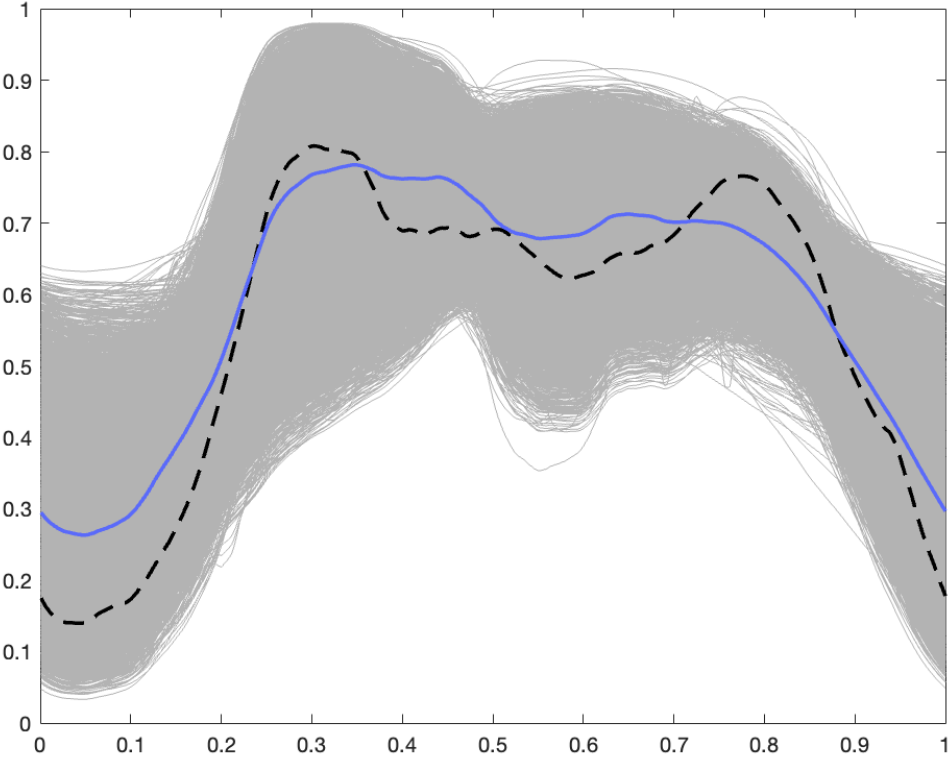}\;
    \includegraphics[width=0.3\textwidth]{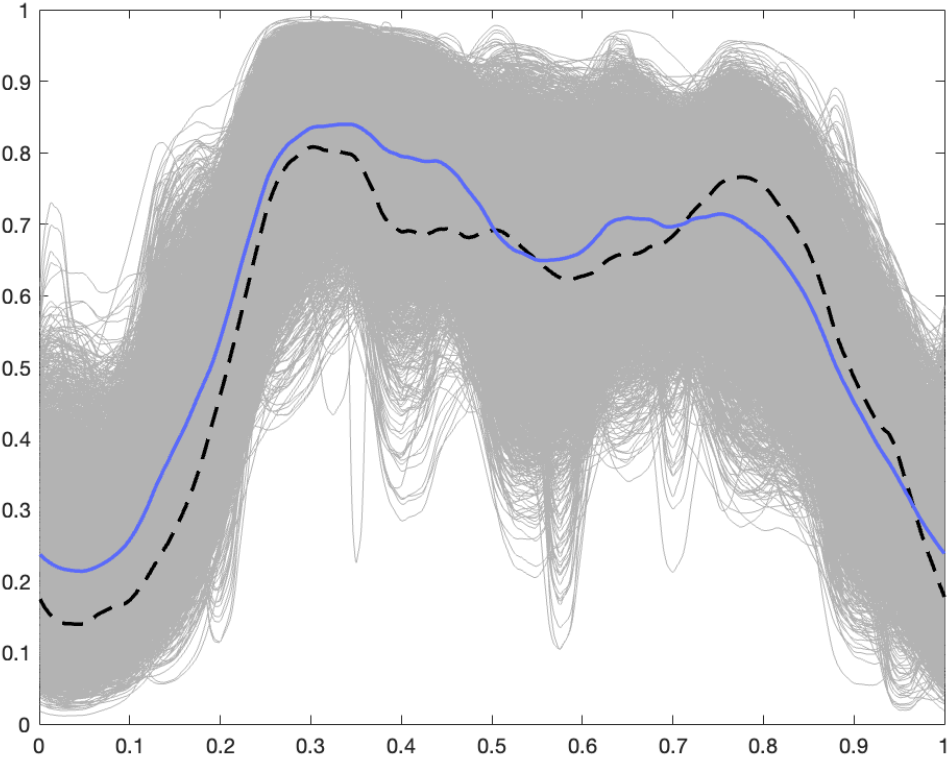}\vspace{0.4cm}

    \includegraphics[width=0.3\textwidth]{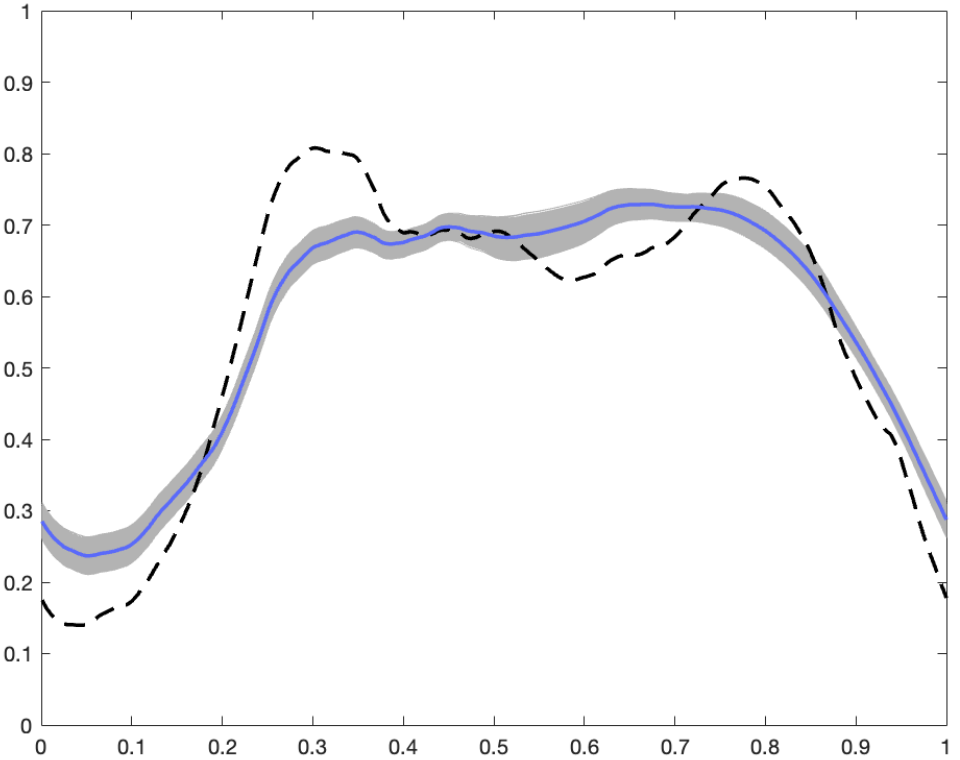}\;
        \includegraphics[width=0.3\textwidth]{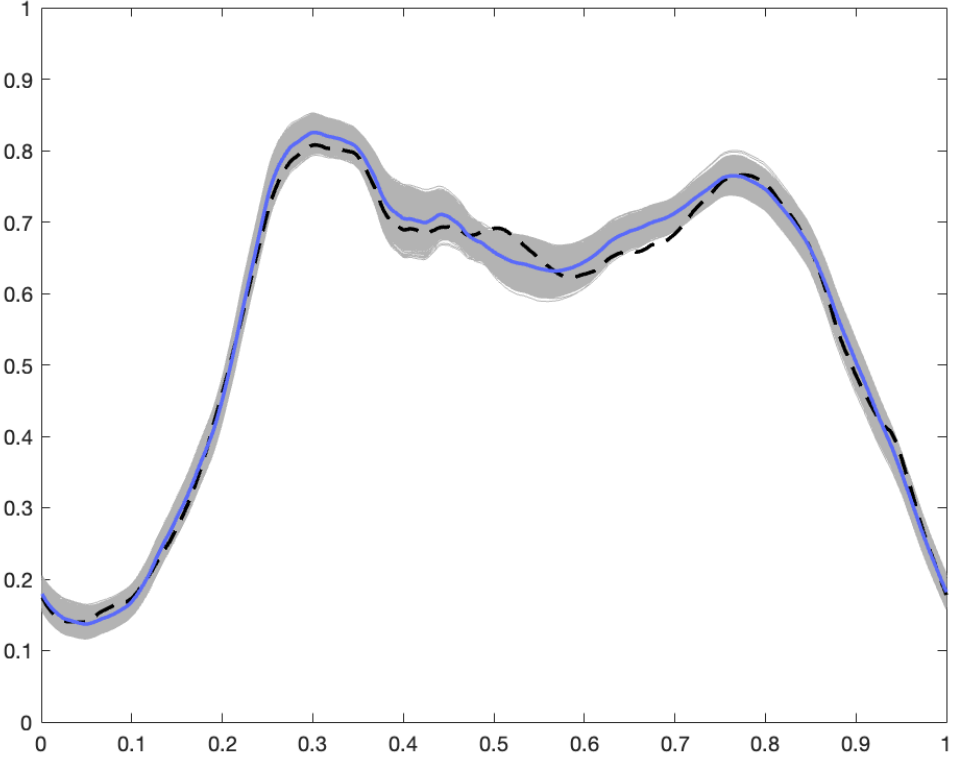}\;
    \includegraphics[width=0.3\textwidth]{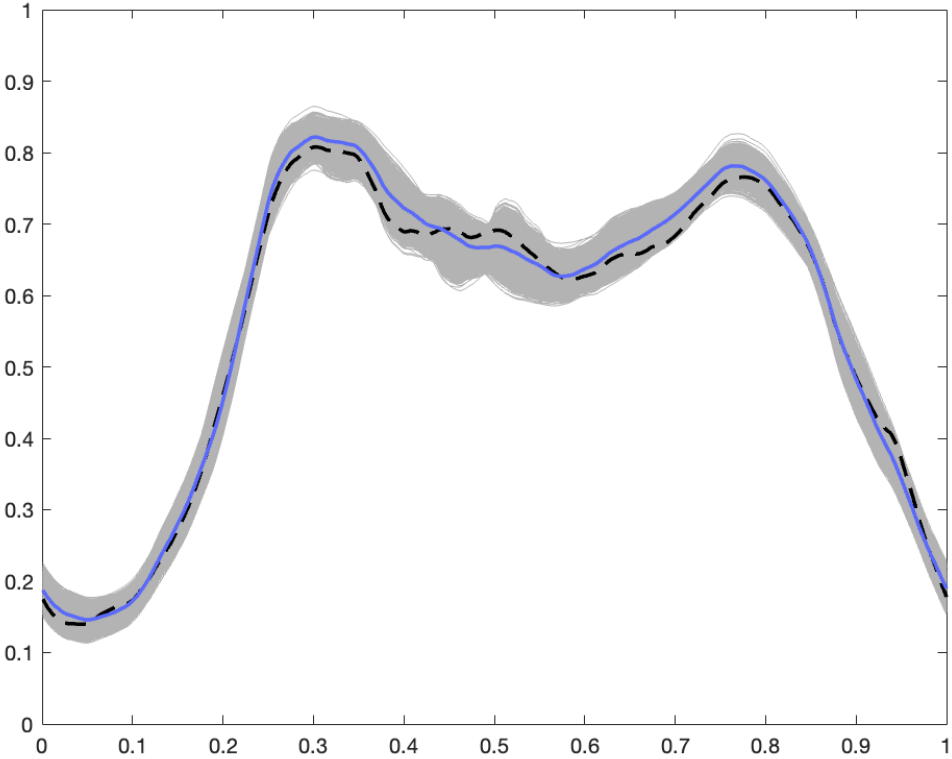}\vspace{0.4cm}
    
        \includegraphics[width=0.3\textwidth]{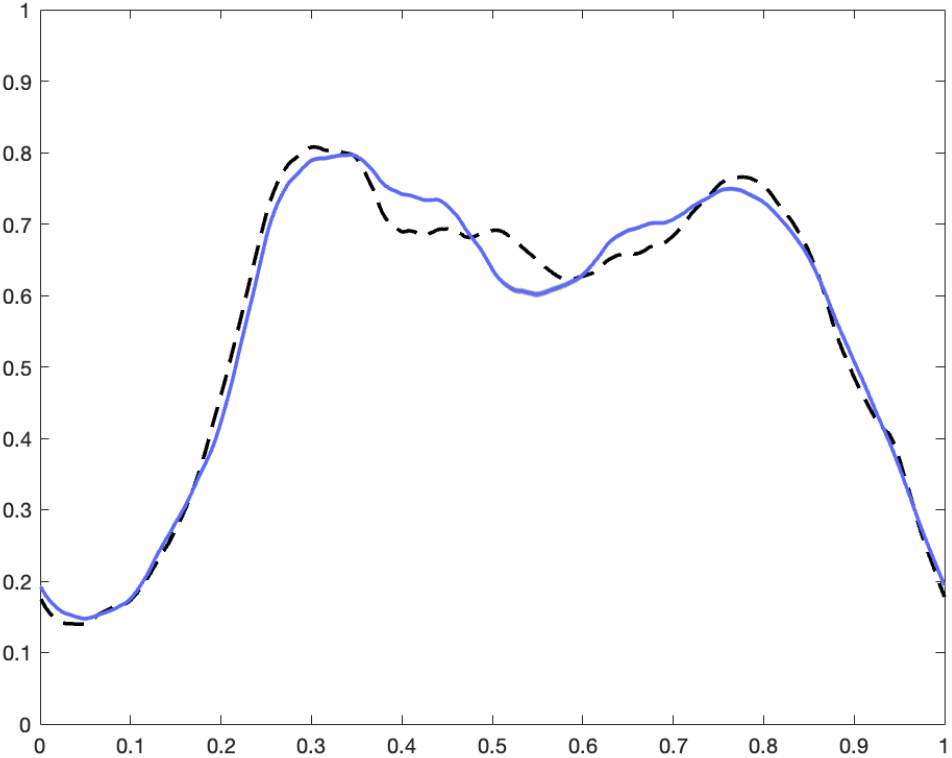}\;
        \includegraphics[width=0.3\textwidth]{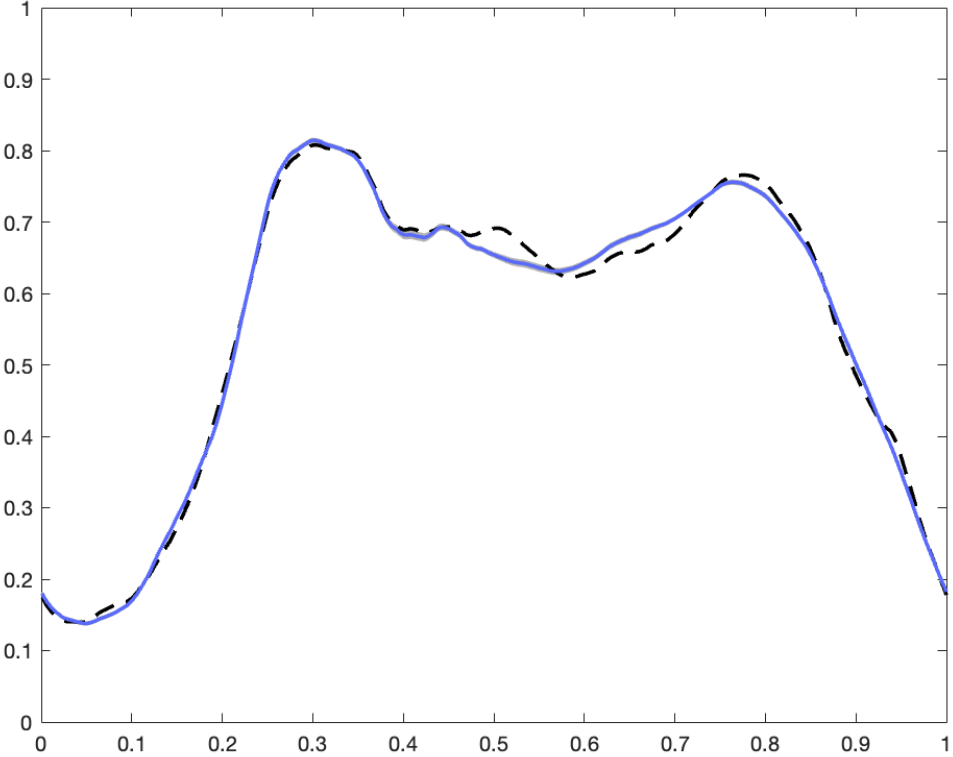}\;
    \includegraphics[width=0.3\textwidth]{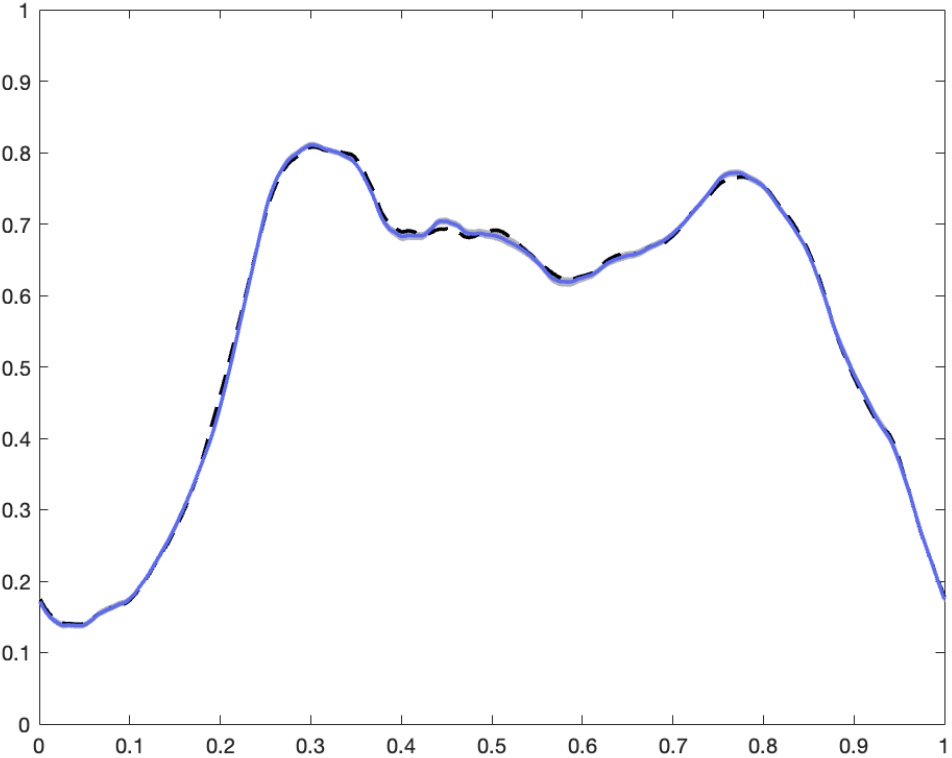}\vspace{0.4cm}

    \caption{Binary classification: true regression function (black dashed), posterior mean (blue), 95\% credible regions (grey), for $n=10^2, 10^4, 10^6$ top to bottom and for the three considered priors left to right.}
    \label{fig-postSob-bin}
\end{figure}

\begin{figure}
    \centering
        \includegraphics[width=0.935\textwidth]{./plots/bin/title2.png} \quad\,

    \includegraphics[width=0.3\textwidth]{./plots/bin/bin-c2-n2-L1.pdf}\;
    \includegraphics[width=0.3\textwidth]{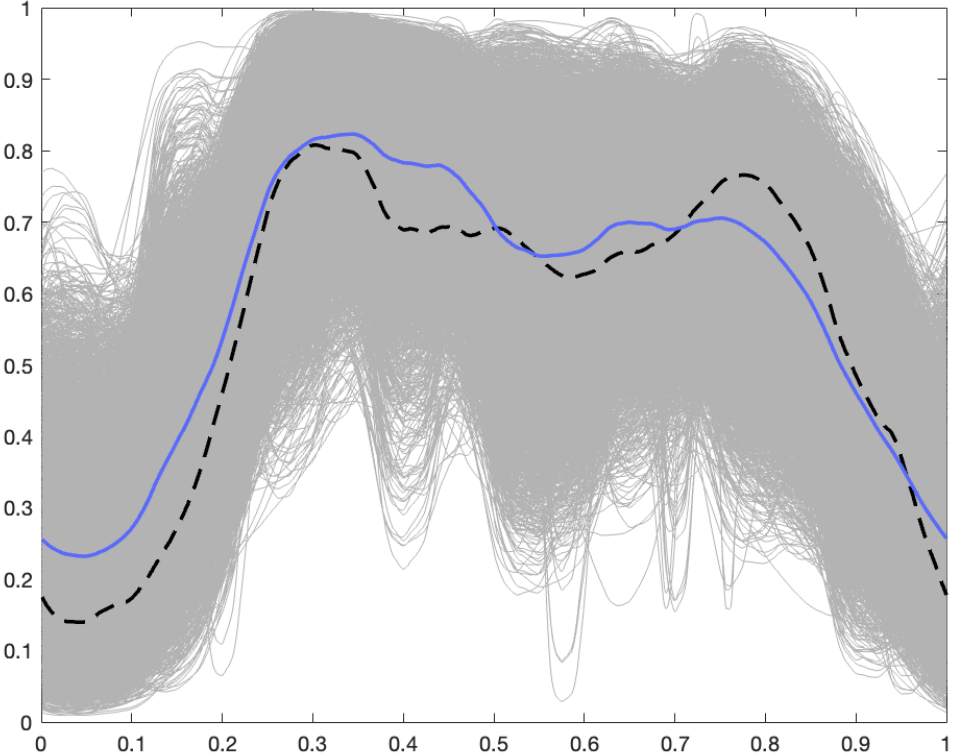}\;
    \includegraphics[width=0.3\textwidth]{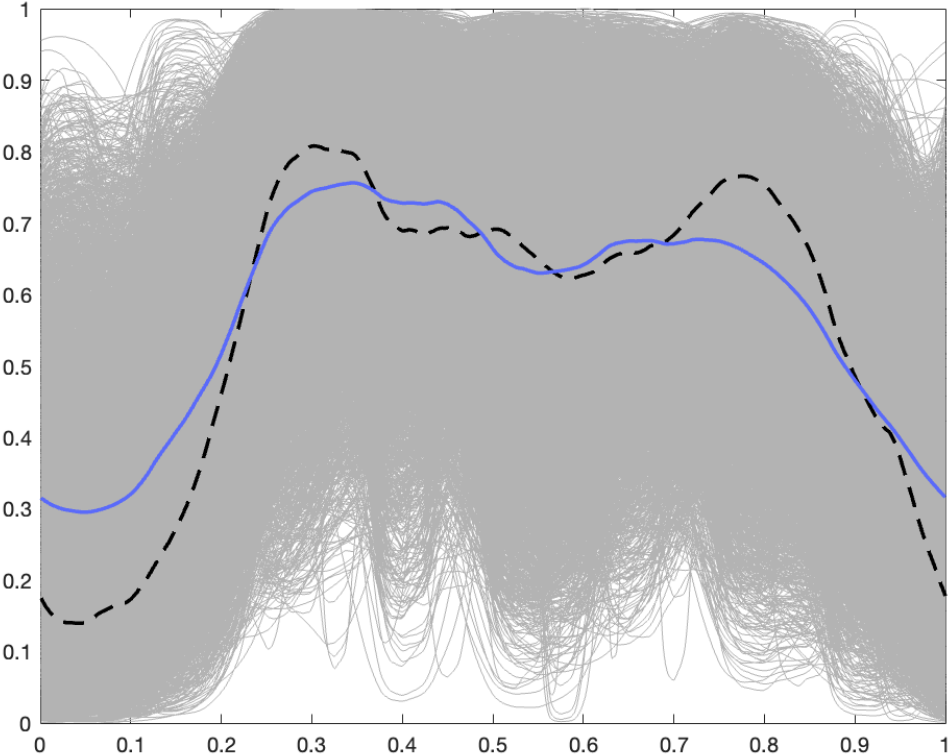}\\\vspace{.4cm}

    \includegraphics[width=0.3\textwidth]{./plots/bin/bin-c2-n4-L1.pdf}\;
    \includegraphics[width=0.3\textwidth]{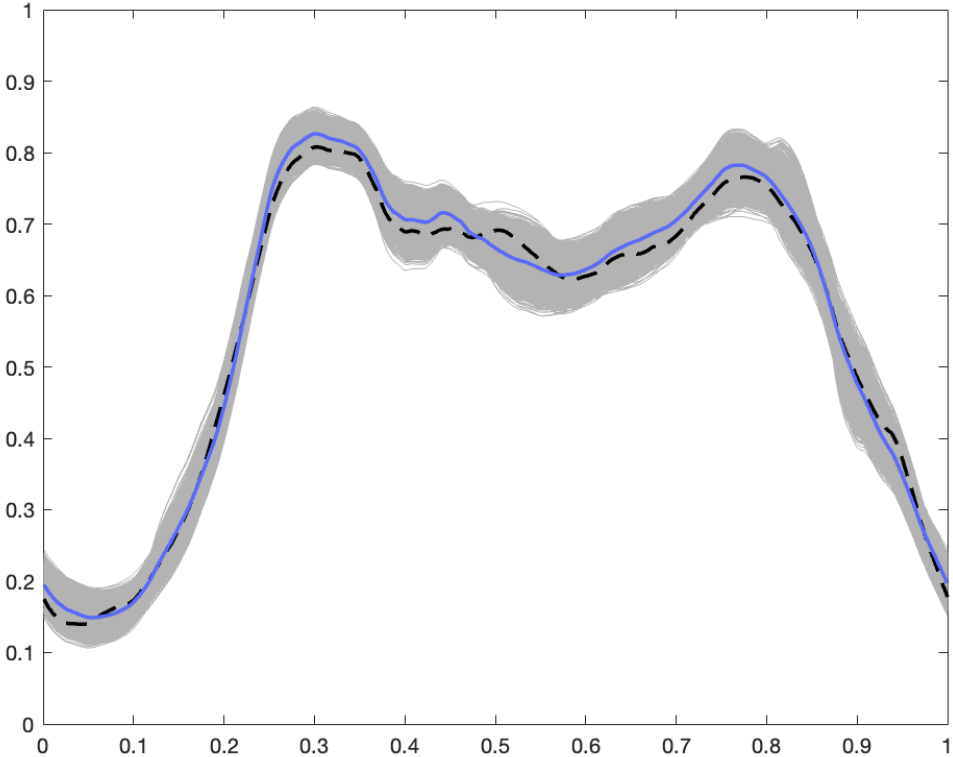}\;
    \includegraphics[width=0.3\textwidth]{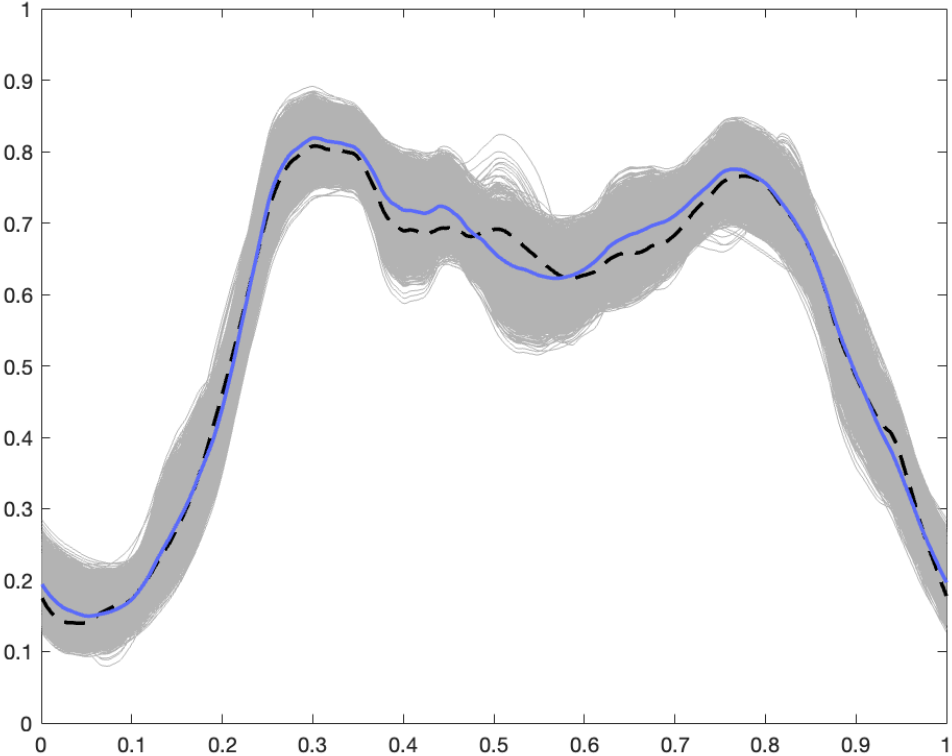}\\\vspace{.4cm}

    \includegraphics[width=0.3\textwidth]{./plots/bin/bin-c2-n6-L1.pdf}\;
    \includegraphics[width=0.3\textwidth]{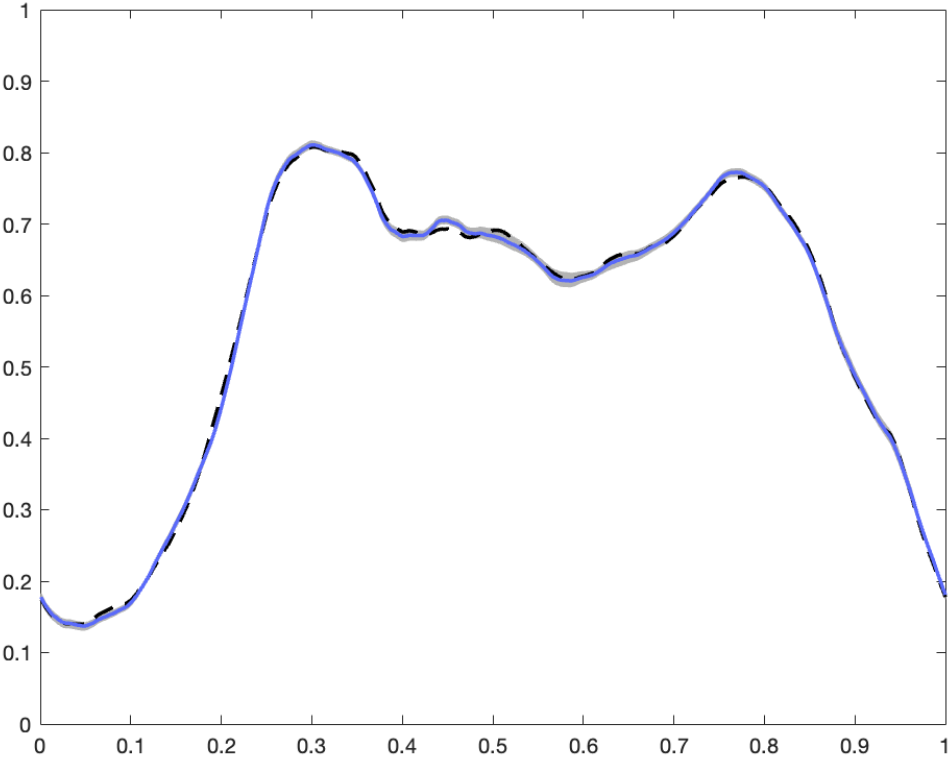}\;
    \includegraphics[width=0.3\textwidth]{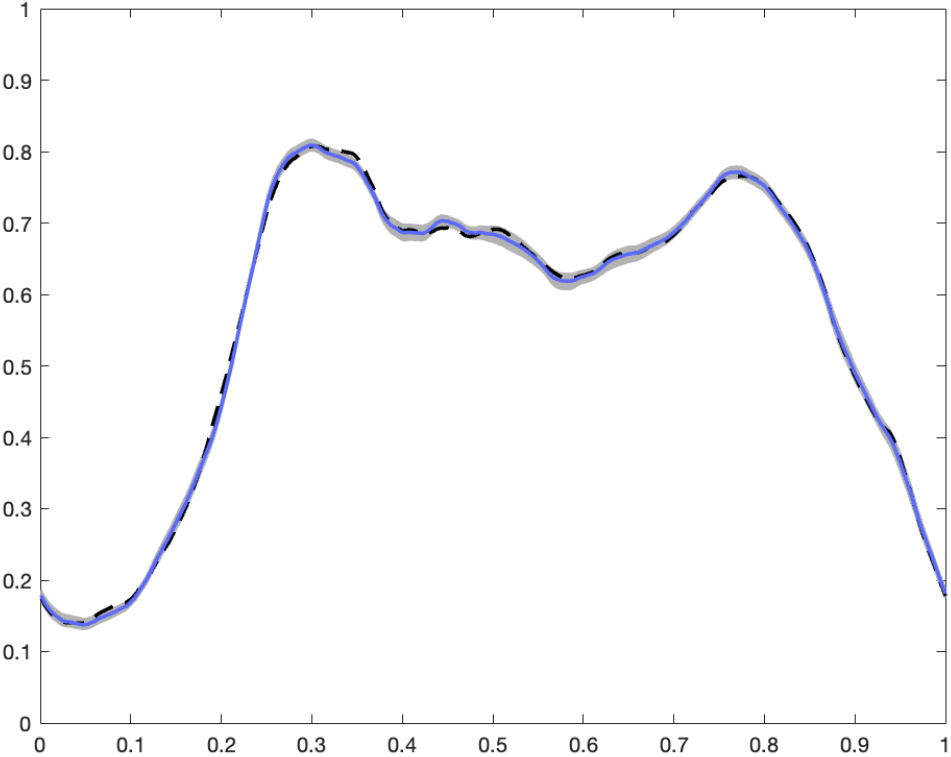}
    
    \caption{Binary classification: true regression function (black dashed), mean (blue), 95\% credible regions according  (grey) for the $\rho$-posterior arising from the Cauchy prior with $s_l=2^{-l^2}$, for $n=10^2, 10^4, 10^6$ top to bottom and $\rho=1, 0.6, 0.2$  left to right.}
    \label{fig-rhopostSob-bin-CD2}
\end{figure}

\subsection{Discussion}
In the provided illustrations, we have employed off the shelf Markov chain Monte Carlo (MCMC) samplers, not particularly tailored to heavy tailed priors and/or multimodal posteriors. In the recent work \cite{BMNW23}, it has been shown that, when used to sample posteriors arising from non-logarithmically concave priors (such as our heavy tailed priors), MCMC algorithms based on random walk or even gradient steps such as the (whitened) pCN and MALA algorithms employed in Subsections \ref{sim:spin}-\ref{bin:sim}, when initialized in the tail of the posterior (`cold started') can take an exponentially long time to reach the area on which the bulk of the posterior mass is located. The difficulty giving rise to this undesirable phenomenon is volumetric, and in particular can arise even if the posterior is unimodal. 

For our priors, we observed a behaviour that could resemble this type of phenomenon only in the simulation setting of Subsection \ref{sim:spin}, where we used whitened MALA to sample the posterior in the white noise model with spatially inhomogeneous truths. When initializing from a prior draw, the chains moved very slowly and the generated samples remained very far from the truth even after very long runs. We addressed this by initializing from the observed data (`warm start'), which resulted in a sequence of samples increasingly denoised and, eventually, convergence to the mean and credible regions seen in Figure \ref{fig-dj94-posterior}. In the density estimation and classification settings of Subsections \ref{de:sim} and \ref{bin:sim}, respectively, a start of the employed whitened pCN algorithm from a prior draw was sufficient, and the results were robust with respect to the initializer. In the linear inverse problem with Sobolev/spatially homogeneous truth setting of Subsection \ref{ip:sim} (see also Section \ref{sec:sim}), where we employed Stan to sample the univariate posteriors on each coefficient, we used Stan's default initialization from a uniform draw in $(-2,2)$ and the performance was again robust to the initializer. Note that the latter coefficient-wise approach can also be applied in the setting of Subsection \ref{sim:spin}, leading to similar results to the ones obtained by the `warm started' whitened MALA function space sampler. 

A number of MCMC methods have been proposed recently for sampling high-dimensional multimodal and/or heavy tailed posteriors, which may be better suited to our priors; we mention \cite{TMR21} and \cite{SSCF22} which employ tempering (or cooling/annealing) techniques and \cite{YLR22} which exploits a projection of the original problem on a sphere.
A more in-depth investigation of the numerical aspects of the proposed heavy tailed priors is certainly a very interesting future direction, nevertheless, this section illustrates that even with the employed generic algorithms one obtains very promising results.

\bibliographystyle{imsart-number}
\bibliography{bibht}

\end{document}